\documentclass[12pt]{amsart}

\usepackage[margin=1in]{geometry}  % set the margins to 1in on  all sides
\usepackage{graphicx, xcolor}              % to include figures
\usepackage{amsmath}               % great math stuff
\usepackage{amsfonts}              % for blackboard bold, etc
\usepackage{amsthm}  
 \usepackage{float} 
\usepackage{caption}
\usepackage{subcaption}
\usepackage{color}

\usepackage{mathrsfs}
\usepackage{amsbsy}
\usepackage{url}
\usepackage{cite}
\theoremstyle{plain}
\newtheorem{thm}{Theorem}[section]
\newtheorem{lem}[thm]{Lemma}
\newtheorem{prop}[thm]{Proposition}
\newtheorem{cor}{Corollary}[section]
\newtheorem{definition}{Definition}

\newtheorem{rem}{Remark}

\newtheorem{example}{Example}[section]
\newtheorem{conj}[thm]{Conjecture}	% better theorem environments
\newtheorem*{remark}{Remarks}
\newcommand{\ackname}{Acknowledgements}

\DeclareMathOperator{\Inv}{Inv}
\DeclareMathOperator{\sgn}{sgn}
\DeclareMathOperator{\writhe}{wr}
\DeclareMathOperator{\Odd}{Odd}
\DeclareMathOperator{\signn}{sign}
\pdfoutput=1

\begin{document}

\nocite{*}
\raggedbottom

\title{New Invariants of Knotoids}

\author{Neslihan G{\"u}g{\"u}mc{\"u}}
\author{Louis H.Kauffman}

\address{N.G{\"u}g{\"u}mc{\"u}:Department of Mathematics, National Technical
University of Athens, Zografou Campus, GR-157 80 Athens, Greece}
\address{Louis H.Kauffman:Department of Mathematics, Statistics and Computer
Science, University of Illinois at Chicago, 851 South Morgan St., Chicago
IL 60607-7045, U.S.A.}
\email{nesli@central.ntua.gr} \email{kauffman@math.uic.edu}

%\author{Neslihan G{\"u}g{\"u}mc{\"u}%\footnote{National Technical University of Athens}, \,\href{mailto:nesli@central.ntua.gr}} 
%\ and Louis H.Kauffman}%\footnote{University of Illinois at Chicago, \,\href{kauffman@uic.edu}}

 

%\thanks{This research  has been co-financed by the European Union (European Social Fund - ESF) and Greek national funds through the Operational Program "Education and Lifelong Learning" of the National Strategic Reference Framework (NSRF) - Research Funding Program: THALES: Reinforcement of the interdisciplinary and/or inter-institutional research and innovation. }
\begin{abstract}
In this paper we construct new invariants of knotoids including the odd writhe, the parity bracket polynomial, the affine index polynomial and the arrow polynomial, and give an introduction to the theory of virtual knotoids. The invariants in this paper are defined for both classical and virtual knotoids in analogy to corresponding invariants of virtual knots. We show that knotoids in $S^2$ have symmetric affine index polynomials. The affine index polynomial and the arrow polynomial provide bounds on the height (minimum crossing distance between endpoints) of a knotoid in $S^2$. 
\end{abstract}
\maketitle
\section{Introduction}
The theory of knotoids was introduced by V.Turaev \cite{Tu} in 2012. Knotoids in $S^2$ (classical knotoids) are open ended knot diagrams, forming a new diagrammatic theory that is an extension of the classical knot theory. A standard $1$-$1$ tangle (or long knot) has its endpoints in a single region of the diagram. A knotoid diagram generalizes the $1$-$1$ tangle and allows the endpoints to be in different regions of the diagram. This gives rise to a new theory and many new questions.

 Taking knotoids up to the equivalence described in the body of the paper, we can ask how far apart the endpoints need to be in all instances of diagrams for the equivalence class. We call the least such distance (in terms of crossing the boundaries of regions) the {\it height} of the knotoid. This sort of question about knotoids and their diagrams is a matter of combinatorial topology. Accordingly, we shall consider here a number of combinatorial topological invariants of knotoids, including the Jones polynomial.

 We find it natural to examine knotoids in the context of virtual knot theory. Virtual knots are knots in thickened surfaces (or knot diagrams in surfaces) taken up to handle stabilization. There is a diagrammatic theory for virtual knots, as we explain in the body of the paper. We extend knotoids to virtual knotoids, and also use methods from virtual knot theory to study knotoids in $S^2$. This connection between classical knotoids and virtual knot theory is fundamental, and will be the subject of work beyond the present paper. Knotoids are certainly a part of geometric three-dimensional knot theory and, as such, are related to open-ended embeddings of intervals in three dimensional space. We discuss this point of view as well in Section 2.1. Many classical invariants of knots and links can be defined/applied for knotoids in $S^2$, the knotoid group and the Jones polynomial \cite{Tu} are primary examples. We concentrate on a combinatorial topological approach in the present paper, and plan to consider more geometric approaches elsewhere. 

 In this paper we conjecture that the natural extension of the Jones polynomial to knotoids in $S^2$ detects the unknot (in the category of knotoids). It appears that this conjecture 
has the same level of difficulty as the conjecture about knot-detection for the Jones polynomial of classical knots. This indicates the importance of knotoids for understanding the Jones polynomial and its relationship with geometric topology. We now summarize the contents of the paper.

 The second section and the third section of this paper are recollections of classical knotoids, and virtual knot theory, respectively. In the second section we recall the basics of knotoids following V.Turaev's paper \cite{Tu}. Then we give a geometric interpretation of knotoids as open oriented curves embedded in $3$-dimensional space. Given an arbitrary oriented smooth open curve embedded in $\mathbb{R}^3$ we can project it to many planes, obtaining a multiplicity of knotoid diagrams in those planes. This collection of knotoids is a new measure (and definition) of the knottedness of an open-ended oriented smooth curve in space.

 The third section continues with the presentation of virtual knotoids. Then we introduce flat knotoid diagrams which will be a key ingredient for proving the theorems of the following sections. The third section ends with a discussion on the virtual closure map. The virtual closure determines a well-defined map from knotoids in $S^2$ to virtual knots, so is a key to define invariants for knotoids via virtual knot invariants. We note that the virtual closure map is a non-injective and a non-surjective map.

 The fourth section is a discussion about parity for knotoids and introduces the odd writhe which is a simple and useful invariant of knotoids. Using the parity of knotoids, we re-define the parity bracket polynomial \cite{Ma} for both classical and virtual knotoids. We show examples of nontrivial virtual knotoids whose non-triviality and also virtuality can be detected by the parity bracket polynomial. It is important to note that crossings of knotoids in $S^2$ can have parity. This is a remarkable appearance of parity in an essentially classical setting.
 
 In the fifth section we define the affine index polynomial for knotoids, and prove its symmetry for knotoids in $S^2$. We show how the symmetry of the affine index polynomial can be used to detect virtuality of a virtual knotoid. Also we discuss the image of the virtual closure map by using the symmetry of the affine index polynomial.

 In the sixth section we introduce the arrow polynomial of knotoids. We are already familiar with both the affine index and the arrow polynomial from virtual knot theory, and we show that they are nontrivial and useful invariants for knotoids as well. 
In both fifth and sixth sections, we show inequalities that relate these polynomials with the height (a minimal crossing distance between endpoints) of knotoids in $S^2$. These inequalities can often be used to determine the height of a knotoid.

 We end the paper with a discussion of problems and directions for classical and virtual knotoids.
\section{Knotoids}
 A \textit{knotoid diagram} in $S^2=\mathbb{R}^2\cup\infty$ or $\mathbb{R}^2$ is a generic immersion of the unit interval $[0,1]$ into $S^2$ or $\mathbb{R}^2$ with finitely many transversal double points endowed with over/under- crossing data. Such a double point is called a \textit {classical crossing} of the knotoid diagram. The points that are the images of $0$ and $1$ are distinct from each other, and from any of the crossings. These two points are the \textit{endpoints} of a knotoid diagram and they are called the \textit{tail} and the \textit{head}, respectively. A projection of an open curve is \textit{generic} if it is $1$-$1$ at all points including the endpoints, except finitely many points where it is 2-1 and transversal in the sense that the tangent directions at the double point are distinct. A knotoid diagram can be seen as a generic projection of an open oriented curve embedded in $S^2\times I$ to $S^2$. Knotoid diagrams are oriented from the tail to the head. The \textit{trivial knotoid diagram} is an embedding of the unit interval into $S^2$ (or in $\mathbb{R}^2$). It is depicted by an arc without any crossings as shown in Figure \ref{fig:knotoid}(a).  
\begin{figure}[H]
\centering  \scalebox{0.9}{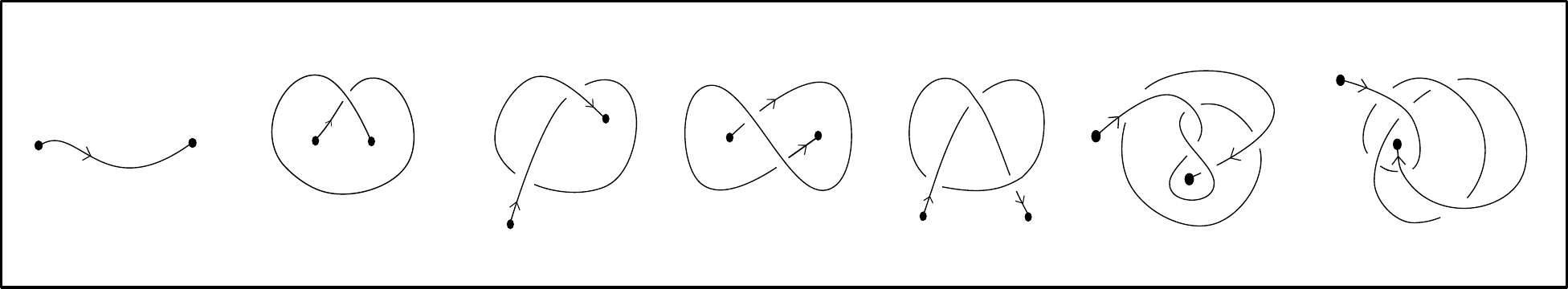}
\caption{\bf Knotoid diagrams}
\label{fig:knotoid}
\end{figure}
 The three Reidemeister moves, denoted by $\Omega_1$, $\Omega_2$, $\Omega_3$, see Figure 2(a), are defined on knotoid diagrams. We refer them as $\Omega$-moves. These moves modify a knotoid diagram within small disks surrounding the local diagrammatic regions as shown in Figure and they do not utilize the endpoints. It is forbidden to pull the strand adjacent to an endpoint over/under a transversal strand as shown in Figure 2(b). These moves are called \textit{forbidden knotoid moves}, and denoted by $\Phi_+$ and $\Phi_-$, respectively. Notice that if both $\Phi_+$ and $\Phi_-$- moves are allowed, any knotoid diagram in $S^2$ (or in $\mathbb{R}^2$) can be turned into the trivial knotoid diagram. 
\begin{figure}[H]
    \centering
		\hspace{10cm}

   \begin{subfigure}[c]{0.3\textwidth}
		\Large{
        \centering  \scalebox{0.5}{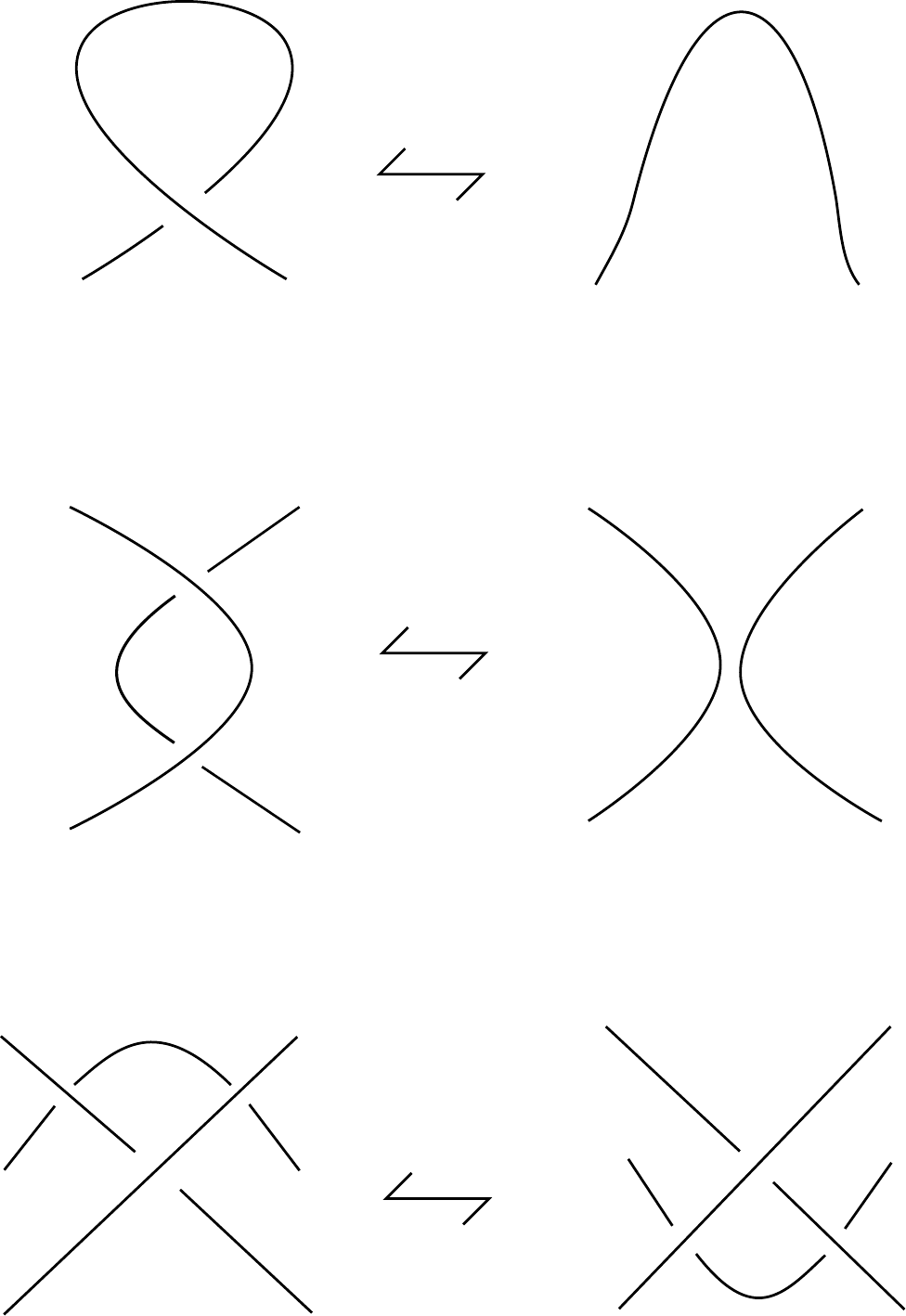}}
        \caption{\bf $\Omega_{i=1,2,3}$- moves}
        \label{subfig:om}
    \end{subfigure}
		\hspace{2cm} \begin{subfigure}[c]{0.5\textwidth}
		\Large{
        \centering  \scalebox{0.6}{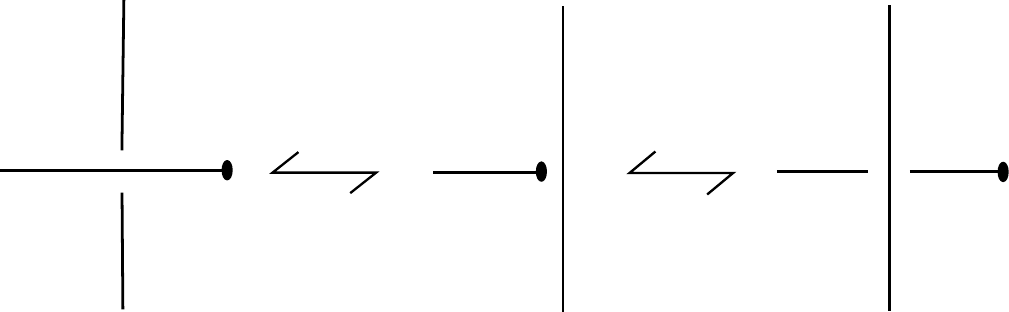}
				}
        \caption{\bf Forbidden knotoid moves}
        \label{subfig:for}
    \end{subfigure}
		\caption{The moves on knotoid diagrams in $S^2$}
    ~ %add desired spacing between images, e. g. ~, \quad, \qquad, \hfill etc. 
\end{figure}
 The $\Omega$-moves plus isotopy of $S^2$ (isotopy of $\mathbb{R}^2$) generate an equivalence relation on knotoid diagrams in $S^2$ (on knotoid diagrams in $\mathbb{R}^2$, respectively). Two knotoid diagrams in $S^2$ are said to be \textit{equivalent} if they are related to each other by a finite sequence of $\Omega$-moves and isotopy of $S^2$ (isotopy of $\mathbb{R}^2$ for knotoid diagrams in $\mathbb{R}^2$). Knotoid diagrams in $S^2$ or in $\mathbb{R}^2$ with only classical crossings and their equivalence classes are called \textit{classical knotoid diagrams} and \textit{classical knotoids}, respectively. A knotoid diagram with \textit{virtual crossings}, see Figure \ref{fig:projection}, is called a \textit{virtual knotoid diagram}, see Section 3. In this paper, we consider primarily knotoids in $S^2$, and $S^2$ is always endowed with the orientation extending the counterclockwise orientation of $\mathbb{R}^2$. We emphasize where a classical knotoid is defined, by writing a \textit{knotoid in $S^2$} or a \textit{knotoid in $\mathbb{R}^2$}, when necessary, otherwise we shall write classical knotoid diagrams and classical knotoids to refer to 
both spherical and planar knotoid diagrams and knotoids, respectively.

 The sets comprising all  knotoids in $S^2$ and in $\mathbb{R}^2$ are denoted by $K(S^2)$ and $K$($\mathbb{R}^2$), respectively. There is an inclusion map between the two sets of classical knotoids, $\iota:\mathcal{K}(\mathbb{R}^2)\rightarrow\mathcal{K}(S^2)$, that is induced by the inclusion ${\mathbb{R}^2}\hookrightarrow {S^2}\cong\mathbb{R}^2\cup\infty$ \cite{Tu}. A knotoid diagram \textit{represents} a knotoid if it is in the equivalence class of the knotoid. Any knotoid in $S^2$ can be represented by a knotoid diagram in $\mathbb{R}^2$ by pushing a representative diagram in $S^2$ away from $\infty\in{S^2}$. Considering the equivalence class of this planar representation in ${\mathcal{K}}({\mathbb{R}^2})$, there is also a well-defined map $\rho:\mathcal{K}(S^2)\rightarrow\mathcal{K}({\mathbb{R}^2})$. It is clear that $\iota\circ\rho=id$ so that the inclusion map $\iota$ is surjective. However, there are examples of nontrivial knotoids in $\mathbb{R}^2$ which are trivial in $\mathcal{K}(S^2)$. For instance, the knotoid diagram given in Figure \ref{fig:knotoid}(b) represents a nontrivial planar knotoid \cite{Tu} whilst it represents the trivial knotoid in $S^2$. In fact, it can be turned into the trivial diagram by an isotopy of $S^2$ followed by an $\Omega_1$-move. This tells that the map $\iota$ is not an injective map.
   
\begin{definition} 
\normalfont
 Let $\mathcal{M}$ be a category of mathematical structures (e.g. polynomials, Laurent polynomials, the integers modulo five, commutative rings, groups, $\cdots$). An \textit{invariant} of classical knotoids is a mapping I: Classical Knotoids $\rightarrow$ $\mathcal{M}$ such that equivalent knotoids map to equivalent structures in $\mathcal{M}$.
\end{definition}
 Every knotoid diagram in $S^2$ is associated with at least one knot in $\mathbb{R}^3$ (a classical knot) \cite{Tu}. One way to obtain a classical knot diagram from a given knotoid diagram $K$, is to connect the endpoints of $K$ by an arc embedded in $S^2$ which is declared to go under each strand it meets during the connection. Such an embedded arc is called a \textit{shortcut} of $K$. The classical knot diagram obtained by connecting the endpoints of $K$ with a shortcut is called an \textit{underpass closure of} $K$. Any two shortcuts of $K$ are isotopic to each other, by a classical topological argument. It is clear that the isotopy between two shortcuts induces Reidemeister moves and planar isotopies between the resulting knot diagrams. Furthermore, any application of $\Omega$-moves on $K$ is realized as Reidemeister moves on the resulting knot diagram. Thus, assigning to a knotoid diagram $K$ its underpass closure induces a well-defined map from the set of knotoids in $S^2$ to the set of knots (isotopy classes) in $\mathbb{R}^3$. This map is called \textit{the underpass closure map} and is denoted by $\omega_-$. We say that a classical knot $\kappa$ is represented by a knotoid diagram $K$ if $\kappa$ is the underpass closure of $K$.

%\begin{center}
%$\omega_-$: Classical Knotoids $\rightarrow$ Knots in $\mathbb{R}^3$. 
%\end{center}
 Every classical knot can be represented by a knotoid diagram in $S^2$. Let $\kappa$ be a knot in $\mathbb{R}^3$. We take an oriented diagram $D$ of $\kappa$ in $S^2$. Cutting out an open arc from $D$ that contains no crossing or one or more overcrossings (that is, the arc is an underpassing arc) results in a classical knotoid diagram. By cutting out different underpassing arcs of the knot diagram $D$, we obtain a number of classical knotoid diagrams that go back to $D$ via the map $\omega_-$, in other words, they all represent the knot $\kappa$. Note that these knotoid diagrams may be nonequivalent to each other, depending on the arc we cut. For instance, the knotoid diagrams given in Figure 1(c) and Figure 1(e), represent the same trefoil knot. It can be shown that they are in fact nonequivalent diagrams by the knotoid invariants introduced in this paper, such as the odd writhe, the affine index polynomial and the arrow polynomial. Thus, the map $\omega_-$ is a surjective but not an injective map. The underpass closure map suggests to represent a classical knot via the knotoid diagrams. Obviously, the knotoid diagrams may have fewer crossings than the crossings of the knot they represent. For this reason, working with knotoid diagrams is advantageous for computation of many knot invariants which are related with the crossing number of the knot, such as the knot group and the Seifert genus \cite{Tu}. %Exploring properties of the knotoids representing the same knot is an interesting task in knotoid theory.

 One other way to obtain a classical knot diagram from a given knotoid diagram is to connect the endpoints of the knotoid diagram with an embedded arc in $S^2$ that goes over each strand it meets during the connection. The resulting oriented classical knot diagram is called the \textit{overpass closure} of the knotoid diagram. Assigning to a classical knotoid diagram its overpass closure induces a well-defined map from the set of knotoids in $S^2$ to the set of classical knots as well. This map is called the \textit{overpass closure} \textit{map}. The overpass closure and the underpass closure of a knotoid diagram may give rise to non-isotopic knots. For instance, the overpass closure of the knotoid diagram in Figure \ref{fig:knotoid}(c) is the trivial knot but the underpass closure of the knotoid is a trefoil knot. To have a well-defined representation of classical knots via knotoids, we fix the connection type as either the overpass or the underpass closure. We shall be working with the underpass closure map in this paper.

 The theory of knotoids in $S^2$ is an extension of the theory of classical knots. There is a well-defined injective map $\alpha$
\begin{center}
$\alpha$:~Classical Knot Diagrams in ${S^2}$/$\left\langle R1,R2,R3 \right\rangle$ $\rightarrow$ Knotoids in $S^2$,
\end{center} 
where $\left\langle R1,R2,R3 \right\rangle$ denotes the equivalence relation generated by the usual Reidemeister moves. Let $D$ be an oriented knot diagram in $S^2$, representing a classical knot $\kappa$. The map $\alpha$ is induced by cutting out an open arc of $D$ which is apart from the crossings of $D$, and assigning to $D$ the resulting knotoid diagram in $S^2$. Neither the choice of the knot diagram representing $\kappa$ nor the choice of the open arc to be cut out from the chosen diagram alters the resulting knotoid \cite{Tu}. Therefore the map $\alpha$ is well-defined. For the injectivity of $\alpha$, it is sufficient to see that both underpass and overpass closures of any knotoid that is in the image of the map $\alpha$, are equivalent knot diagrams \cite{Tu}.

 \begin{definition}\normalfont
The knotoids in $S^2$ that are in the image of the map $\alpha$ are called the \textit{knot-type knotoids}, and the knotoids that are not in the image are called the \textit{pure} or \textit{proper knotoids} (we shall prefer to say proper).
\end{definition}
Let $K$ be a classical knotoid diagram with $n$ crossings. By ignoring the over/under information at each crossing of the diagram $K$ and regarding crossings as vertices, we obtain a connected planar graph with $n+2$ vertices, $n$ of which correspond to the crossings and two of which correspond to the endpoints of $K$. This graph is called the\textit{ underlying graph} of the knotoid diagram $K$. The reader can easily check that the underlying graph divides $S^2$ (or $\mathbb{R}^2$) into $n+1$ local regions. We call these regions the \textit{regions of the knotoid diagram} $K$. Each knot-type knotoid has a knotoid diagram in its equivalence class whose endpoints are located in the same local region of the diagram. Such a knotoid diagram is called a \textit{knot-type knotoid diagram}. The endpoints of a proper knotoid can be in any but different local regions of any of its representative diagrams. Figures \ref{fig:knotoid}(a),(b),(e), when they are considered in $S^2$, illustrate some examples of knot-type knotoid diagrams and Figures \ref{fig:knotoid}(c),(d),(f),(g) illustrate some examples of proper knotoid diagrams. Note that a knot-type knotoid can be represented by a proper knotoid diagram, in other words, $\Omega$-moves and isotopy of $S^2$ can change the placement of the endpoints.

 The set of knotoids, $\mathcal{K}(S^2)$ can be regarded as the union of the set of knot-type knotoids and the set of proper knotoids. The set of classical knots is in $1-1$-correspondence with the set of knot-type knotoids via the map $\alpha$. Note that a knot-type knotoid can be thought as a $1-1$ tangle or a long knot. It is well-known that a classical long knot carries the same knotting information as the classical knot obtained by closing the two endpoints of the long knot \cite{Kaw,Va,CDM,Man}. From this it is immediate to conclude that a knot-type knotoid can be considered the same as the classical knot it represents. For proper knotoids this is no longer true. There are proper knotoids (so nontrivial) representing the trivial knot. The knotoid given by the diagram in Figure \ref{fig:knotoid}(d) is a nontrivial proper knotoid \cite{Tu} but it represents the trivial knot (via the underpass closure map). It is one of the fundamental problems in our paper to determine the type of a given knotoid.
\begin{definition}
\normalfont
A \textit{multi-knotoid diagram} is defined to be a knotoid diagram in $S^2$ with multiple circular components \cite{Tu}. We generalize the concept of multi-knotoids to include multiple open-ended components. The equivalence relation on knotoid diagrams in $S^2$ extends naturally to an equivalence relation on multi-knotoid diagrams and a \textit{multi-knotoid} is defined to be an equivalence class of multi-knotoid diagrams. 
\end{definition}
Many of the invariants discussed in this paper extend to multi-knotoids. We will remark on this as the paper proceeds.
\subsection{An Interpretation of Classical Knotoids in 3-Dimensional Space}
%We regard $S^2$ as the 1-point compactification of $\mathbb{R}^2$. We take the planar representation of a given knotoid in $S^2$.
 Let $K$ be a knotoid diagram in $\mathbb{R}^2$. The plane of the diagram is identified with $\mathbb{R}^2\times\{0\}\subset\mathbb{R}^3$. $K$ can be embedded into $\mathbb{R}^3$ by pushing the overpasses of the diagram into the upper half-space and the underpasses into the lower half-space in the vertical direction. The tail and the head of the diagram are attached to the two lines, t$\times$$\mathbb{R}$ and h$\times$$\mathbb{R}$ that pass through the tail and the head, respectively and is perpendicular to the plane of the diagram. Moving the endpoints of $K$ along these special lines gives rise to embedded open oriented curves in $\mathbb{R}^3$ with two endpoints of each on these lines. \vspace{-0.43cm}
%\hspace{1cm}
\begin{figure}[H]
    \centering
    \begin{subfigure}[c]{0.25\textwidth}
		\centering  \scalebox{0.15}{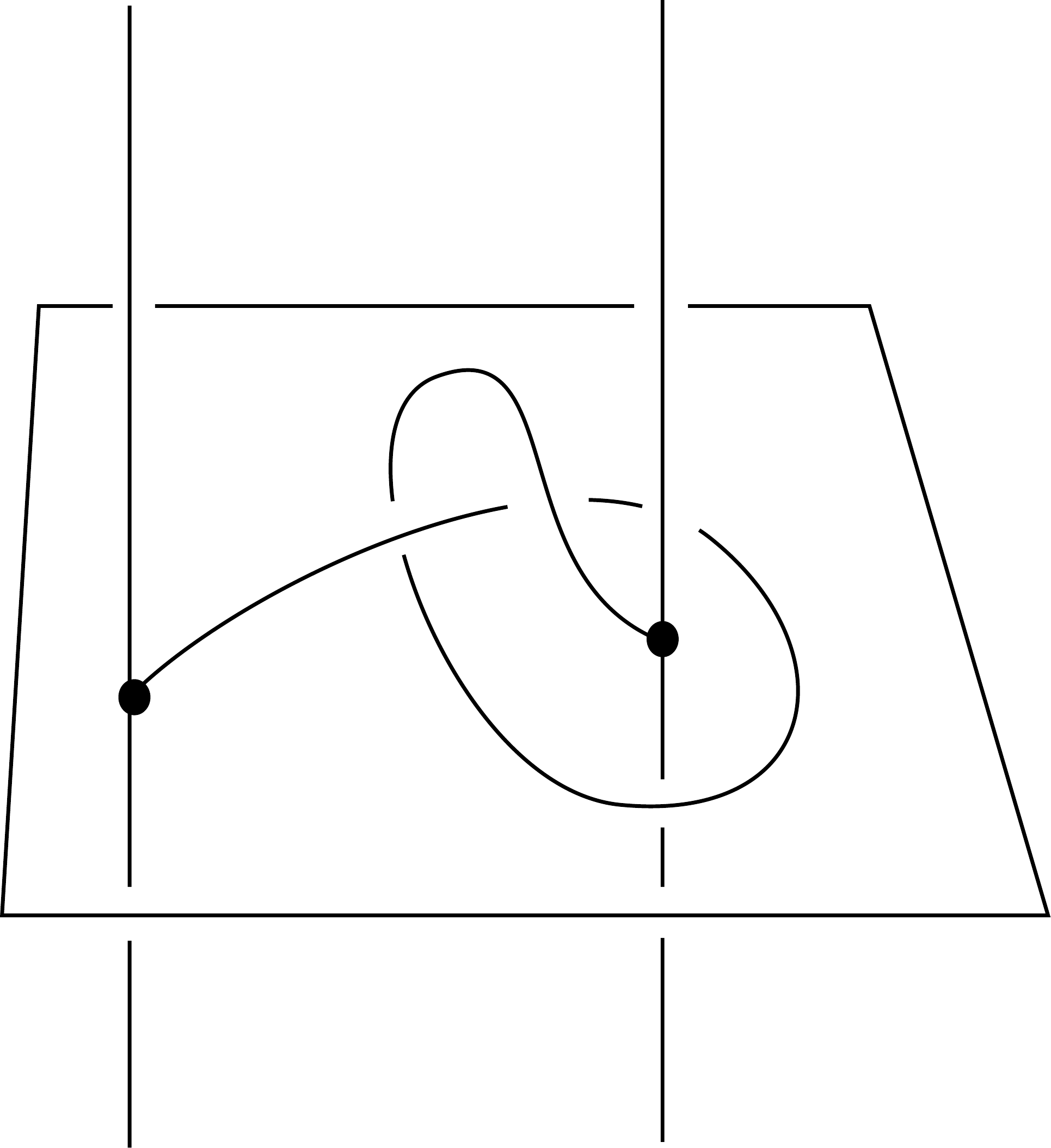}
      \caption{\bf}
       \label{subfig:}
    \end{subfigure}
		%\hfill
		\begin{subfigure}[c]{0.25\textwidth}
			\centering  \scalebox{0.15}{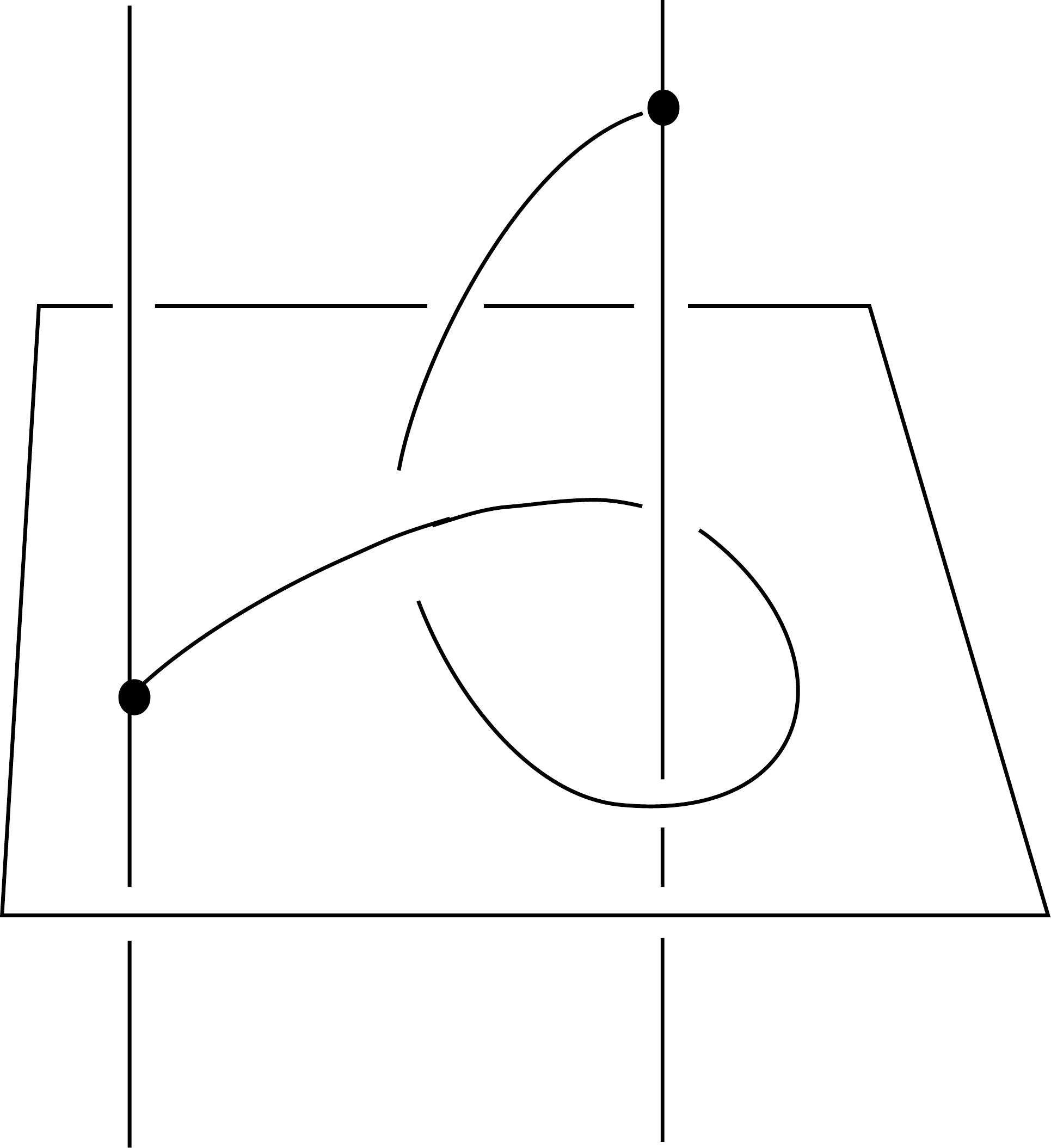}
        \caption{\bf }
        \label{subfig:b}
    \end{subfigure}
		%\hfill
    ~ %add desired spacing between images, e. g. ~, \quad, \qquad, \hfill etc. 
    %(or a blank line to force the subfigure onto a new line)\\
\begin{subfigure}[c]{0.25\textwidth}
			\centering  \scalebox{0.15}{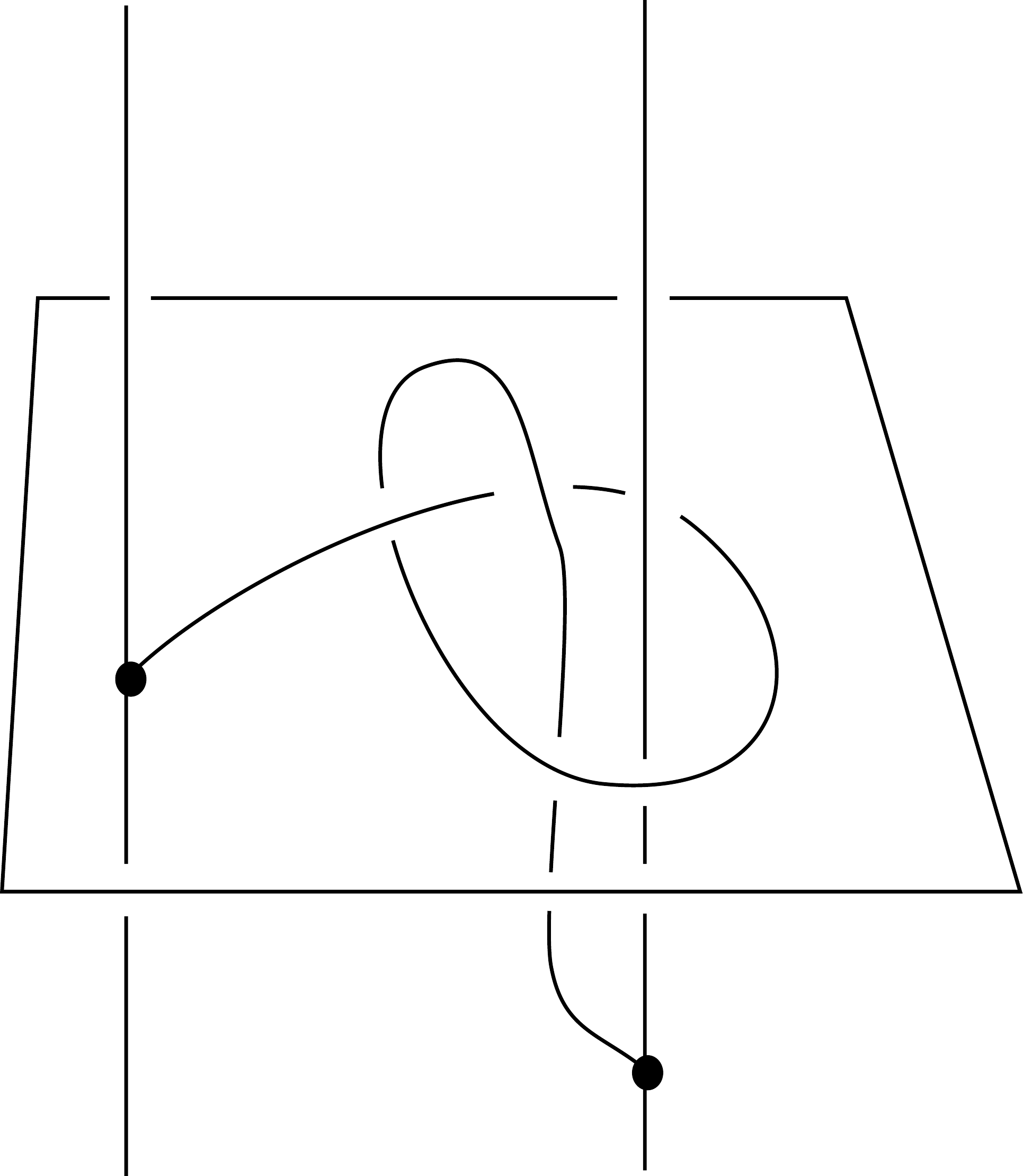}
        \caption{\bf }
        \label{subfig:c}
    \end{subfigure}
		%\hfill
\caption{\bf Curves in $\mathbb{R}^3$ obtained by the knotoid diagram in Figure \ref{fig:knotoid}(c)}
\label{fig:curv}
\end{figure}
 Two smooth open oriented curves embedded in $\mathbb{R}^3$ with the endpoints that are attached to two special lines, are said to be \textit{line isotopic} if there is a smooth ambient isotopy of the pair (${\mathbb{R}^3} \setminus\{t\times{\mathbb{R}}, h\times{\mathbb{R}}$\}, $t\times{\mathbb{R}}\cup h\times{\mathbb{R}}$), taking one curve to the other curve in the complement of the lines, taking endpoints to endpoints, and taking lines to lines; $t\times{\mathbb{R}}$ to $t\times{\mathbb{R}}$ and $h\times{\mathbb{R}}$ to $h\times{\mathbb{R}}$. The curves given in Figure \ref{fig:curv} are line isotopic to each other.

 Conversely, let be given an open oriented embedded curve in $\mathbb{R}^3$ with a generic projection to the $xy$- plane. The endpoints of the curve determine two lines passing through the endpoints and perpendicular to the plane. The generic projection of the curve to the $xy$- plane along the lines which is endowed with over and under-crossing data, is a knotoid diagram in $\mathbb{R}^2$. We call a smooth open embedded curve in $\mathbb{R}^3$ that has a generic projection to the $xy$- plane a \textit{generic curve with respect to the $xy$-plane}. Such a curve has a line isotopy class as described in the previous paragraph.  
\begin{thm}
Two smooth open oriented curves in $\mathbb{R}^3$ that are generic with respect to the $xy$-plane are line isotopic with respect to the lines passing through the endpoints if and only if their generic projections to the $xy$-plane (along the lines) are equivalent knotoid diagrams, that is, they are related by $\Omega$- moves and isotopy of the plane.
\end{thm}
\begin{proof}
 Since everything is set in the smooth category, we can switch to the piecewise linear category. Open curves are defined as \textit{piecewise linear curves} in $\mathbb{R}^3$, that is, as the union of finitely many edges: $[p_1, p_2]$,...,$[p_{n-1}, p_n]$ such that each edge intersects one or two other edges at the points, $p_i$, $i=2$,...,$n-1$ and $p_1$ and $p_n$ are the endpoints of the curve. We define the \textit{triangle move} in $3$- dimensional space. Given an open curve with endpoints on the lines, let $[p_i, p_{i+1}]$ be an edge of the curve and $p_0$ be a point in the complement of the curve and the two lines. The edge is transformed to two edges $[p_i, p_0]$ and $[p_0, p_{i+1}]$ which form a triangle, whenever this triangle is not pierced by another edge of the curve or by the lines. In the reverse direction, a consecutive sequence of two edges may be transformed to one edge by a triangle move. An ambient isotopy of a piecewise linear curve in the complement of the two lines can be expressed by a finite sequence of triangle moves.

 By using triangle moves we can subdivide the edges into smaller edges as shown in Figure \ref{fig:sub}. Any triangle move can be factored into a sequence of smaller triangular moves by \textit{subdividing} the triangles and the edges accordingly. Consider the projection of a curve to the plane, triangular regions that triangular moves take place are projected to non-singular triangles and these triangles possibly contain many strands which are the projection of other edges. The entire ambient isotopy of the curve can be reduced to the shadow cases in the plane shown in Figure \ref{fig:tri}, by subdivision. Inducting on the strands inside the triangles shows that triangle moves are generated by $\Omega$- moves, shown in the left of the figure and the right side shows some cases that are finite combinations of $\Omega$- moves. 
\end{proof}
%Triangular moves can be accomplished by $\Omega_{i=1,2,3}$- moves in the plane, this correspondence is depicted in the left side of the figure. It is seen that the cases that are shown in the right side of the figure are generated by the $\Omega_i$- moves, that is, each is a combination of $\Omega_{i=1,2,3}$- moves. Induction on the strands inside triangular regions that the triangular move forms is the essence for proving one side of the theorem. It is obvious that $\Omega_i$- moves in plane correspond to triangular moves in space. This completes the proof.
\begin{cor}
There is a one-to-one correspondence between the set of knotoids in $\mathbb{R}^2$ and the set of line-isotopy classes of smooth open oriented curves in $\mathbb{R}^3$ with two endpoints attached to lines that pass through the endpoints and perpendicular to the $xy$-plane.
\end{cor}
 It may be the case that a smooth open oriented curve embedded in $\mathbb{R}^3$ is not generic with respect to the $xy$-plane but can be generic with respect to many other planes. Projecting the curve generically to these planes gives a set of knotoid diagrams in the planes. The line isotopy can be generalized to all the curves that is generic with respect to some plane and the theorem above generalizes as follows.
\begin{thm}
Two open oriented curves embedded in $\mathbb{R}^3$ that are both generic to a given plane, are line isotopic (with respect to the lines determined by the endpoints of the curves and the plane) if and only if the projections of the curves to that plane are equivalent knotoid diagrams in the plane.
\end{thm}
 We say that a knotoid in a plane \textit{represents} an open oriented embedded curve in $\mathbb{R}^3$ if the knotoid is in the equivalence class of the generic projection of the curve to some plane.

 The equivalence classes of knotoids in the planes all representing the same open curve embedded in $3$-dimensional space, can vary with respect to the projection plane. For instance, the projection of the curve represented in Figure $3$(b) to the $yz$-plane gives a knotoid diagram with the tail and the head in the unbounded region of the plane and one can see that it is equivalent to the trivial knotoid in the $yz$- plane. The projection to the $xy$- plane, however, is the knotoid diagram given in Figure \ref{fig:knotoid}(c) and we will show in Section 3.3 that this knotoid is a nontrivial knotoid in $S^2$ so is nontrivial in $\mathbb{R}^2$.
\begin{figure}[H]
 \begin{center}
   \begin{tabular}{c}
     \centering  \scalebox{0.6}{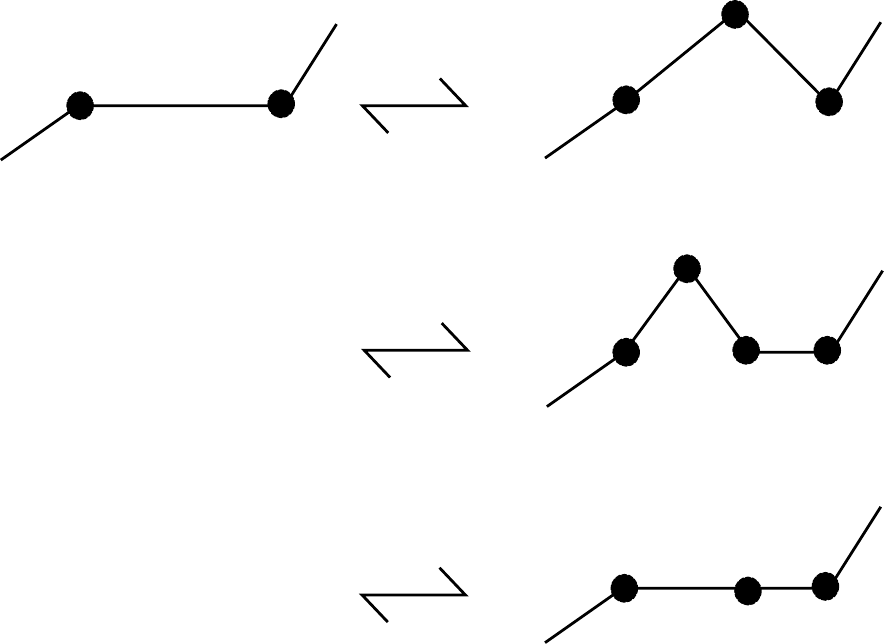}
     \end{tabular}
     \caption{\bf Subdivision of an edge }
     \label{fig:sub}
\end{center}
\end{figure}

\begin{figure}[H]
 \begin{center}
   \begin{tabular}{c}
     \centering  \scalebox{0.75}{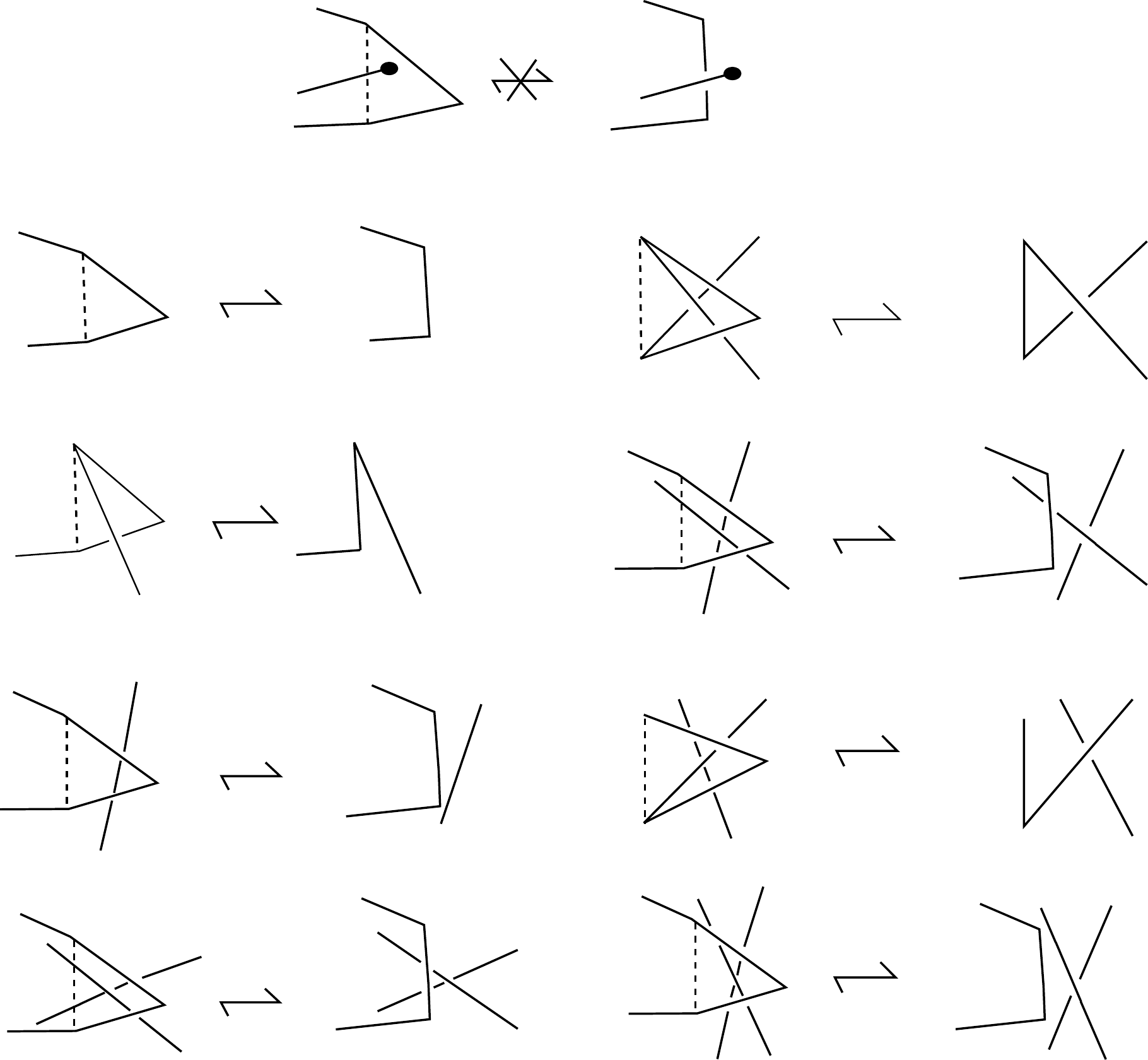}
     \end{tabular}
     \caption{\bf Shadow of triangular moves  }
     \label{fig:tri}
\end{center}
\end{figure}
 Given a smooth open oriented curve embedding $C$ in $\mathbb{R}^3$, we define $\mathcal{P(C)}$ to be the set of all knotoid equivalence classes obtained from generic projection of C to planes in $\mathbb{R}^3$ that are outside a ball containing C. We take $\mathcal{P(C)}$ as a measure of the ‘knottedness’ of $C$. Each knotoid in $\mathcal{P(C)}$ can represent three knots; two classical knots via the underpass and overpass closures, and a virtual knot with genus at most one via the virtual closure map, see Section 3.3 for the virtual closure map. The set of knots that are represented by knotoids in $\mathcal{P(C)}$ can be used as possible closures of $C$. This is a new way to measure the knottedness of an open curve and may have significant applications in the study of tangled physical systems. The reader should compare our definition with the work of \cite{De}.

 Figure \ref{fig:bow} shows a semi-realistic picture of an open curve in space depicting a bowline knot and a possible closure of that knot. The same picture can be viewed (by making the lines more abstract) as a knotoid projection of the bowline. The right-hand picture is the underpass closure of this knotoid that is an eight crossings classical knot. On the other hand, the overpass closure is a classical knot with six crossings. The interested reader is encouraged to find the corresponding classical knots in the knot table \cite{KnI}.  
Moreover the virtual closure of this knotoid is a nontrivial, genus one virtual knot. (The reader can verify that the virtual closure of the knotoid is not a classical knot since the odd writhe of the resulting knot is $-2$). We suggest that a collection of closures of an open curve in space can be obtained by considering both underpass and overpass and also virtual closures of all the generic knotoid projections of the given open curve.
\begin{figure}[H]
\begin{center}
     \begin{tabular}{c}
     \includegraphics[width=7cm]{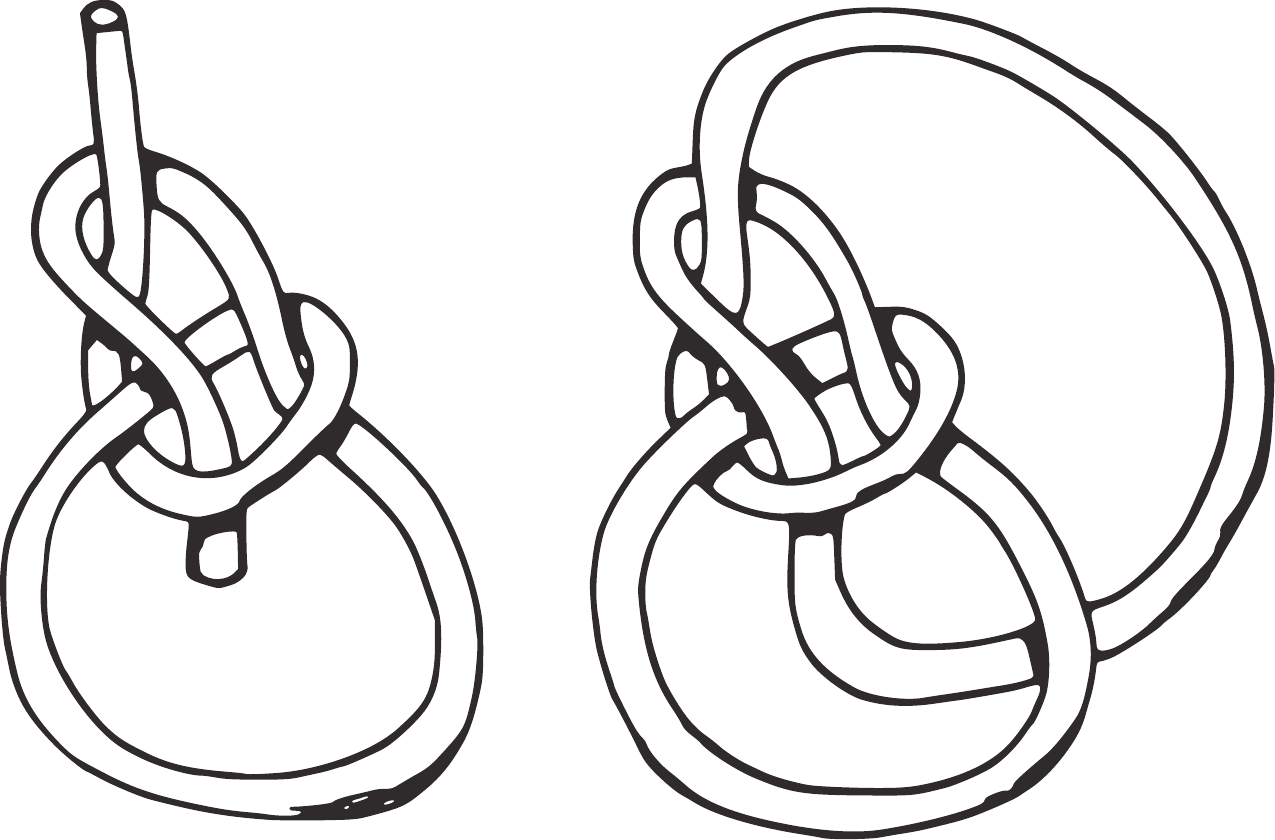}
     \end{tabular}
     \caption{\bf A closure of bowline knot projection }
     \label{fig:bow}
\end{center}
\end{figure}
\section{Virtual Knots and Virtual Knotoids}
%Virtual knot theory which was introduced by L.H.Kauffman in 1996. 
%Virtual knots/links can be defined as the equivalence classes of embeddings of closed curves in   thickened surfaces of arbitrary genus upto isotopy and the addition and removal of empty handles from the surface.
\normalfont
 The theory of virtual knots, introduced by L.H.~Kauffman \cite{Ka1,Ka2} in 1996, studies the embeddings of circles in thickened surfaces modulo isotopies and diffeomorphisms of the surface and one-handle stabilization of the surfaces. Virtual knot theory has a diagrammatic formulation. %It can be defined as a diagrammatic theory.
In the diagrammatic theory, virtual knots and links are represented by diagrams with finitely many transversal crossings called the \textit{classical crossings} and \textit{virtual} crossings that are neither an over-crossing nor an under-crossing. A virtual crossing is an artifact of the representation of the virtual knot diagram in the plane (or equivalently in $S^2$), and it is indicated by two crossing segments with a small circle placed around the crossing point.

 The moves on virtual diagrams are generated by the usual Reidemeister moves plus the detour move. The detour move allows a segment with a consecutive sequence of virtual crossings to be excised and replaced any other such a segment with a consecutive virtual crossings, as shown in Figure 7(a). 
 Virtual knot and link diagrams that can be connected by a finite sequence of these moves are said to be \textit{equivalent} or \textit{virtually isotopic}.

 Virtual knots and links can also be represented by embeddings without any virtual crossings in thickened orientable surfaces just as non-planar graphs may be embedded in surfaces of some genus. There is a unique abstract knot/link diagram assigned to each virtual knot/link diagram. Abstract knot/link diagrams are associated to thickened closed connected orientable surfaces in which the knot diagram is embedded. For more details on abstract diagrams and their association with thickened surfaces, see \cite{KAKA,CKS}. Here we state the following theorem. 
 \begin{thm} (\cite{Ka1, Ka3, Ka4, CKS})
Two virtual link diagrams are virtually isotopic if and only if their surface embeddings are equivalent up to isotopy in the surface, diffeomorphisms of the surface, and addition/removal of empty handles.
\end{thm}
\subsection{Virtual Knotoids}
 We extend classical knotoid diagrams to virtual knotoid diagrams in a combinatorial way. \textit{Virtual knotoid diagrams} are defined to be the knotoid diagrams in $S^2$ with an extra combinatorial structure called \textit{virtual crossings}. A virtual crossing is indicated by a circle around the crossing point of two strands, as in the case of virtual knots. Figure 9(c) and \ref{fig:projection} depict some examples of virtual knotoid diagrams. The moves on virtual knotoid diagrams are generated by the $\Omega$-moves and the detour move, shown in Figure \ref{fig:moo}. Some of the special cases of the detour move is depicted in Figure 7(b) and 7(c). The first three moves in Figure 7(b) are referred as virtual $\Omega_{i=1,2,3}$-moves, and the fourth move is referred as a \textit{partial virtual move}. The move given in Figure 7(c), is called the virtual $\Omega$-move and denoted by $\Omega_v$. The virtual $\Omega$-move enables to slide back/forth the strand which is adjacent to the tail or the head, deleting/creating virtual crossings located consecutively on the strand. The $\Omega_v$- move decreases or increases the number of virtual crossings and changes the locations of the endpoints. Figure $7$(b) and Figure $7$(c) depict all special cases of the detour move. These moves together with the $\Omega$-moves are called the \textit{generalized $\Omega$}-moves.

 The generalized $\Omega$- moves define an equivalence relation on virtual knotoid diagrams. We say that two virtual knotoid diagrams are \textit{virtually equivalent} if one can be obtained from the other by a finite sequence of the generalized $\Omega$-moves and isotopy of $S^2$. A \textit{virtual knotoid} is an equivalence class of virtual knotoid diagrams under this equivalence.
\vspace{-0.30cm} 
\begin{figure}[H]
    \centering
	\begin{subfigure}[a]{0.5\textwidth}
        \centering  \scalebox{0.27}{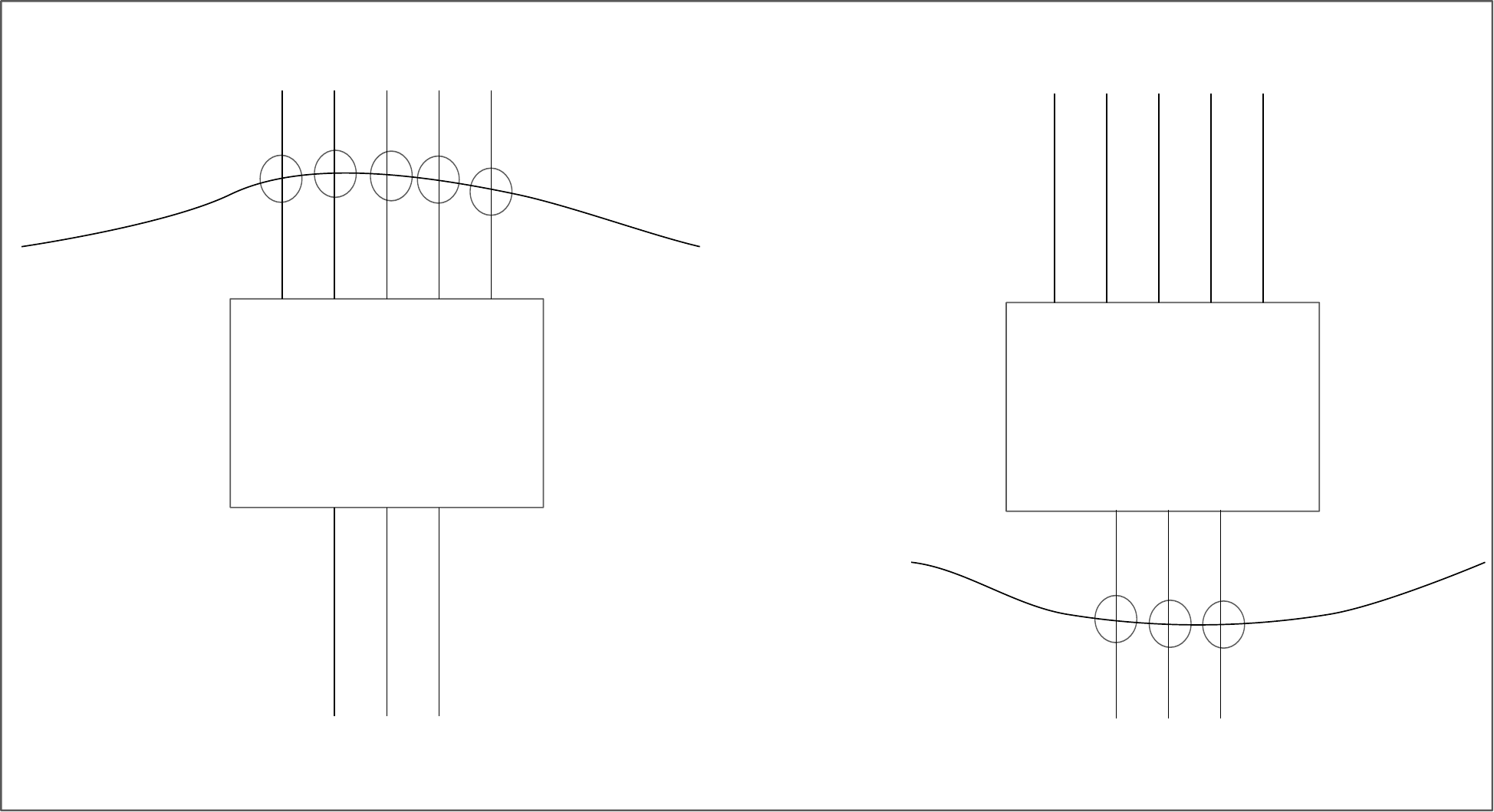}
        \caption{\bf The detour move}
        \label{subfig:det}
    \end{subfigure}
		 
			\begin{subfigure}[b]{0.75\textwidth}
			\centering  \scalebox{0.27}{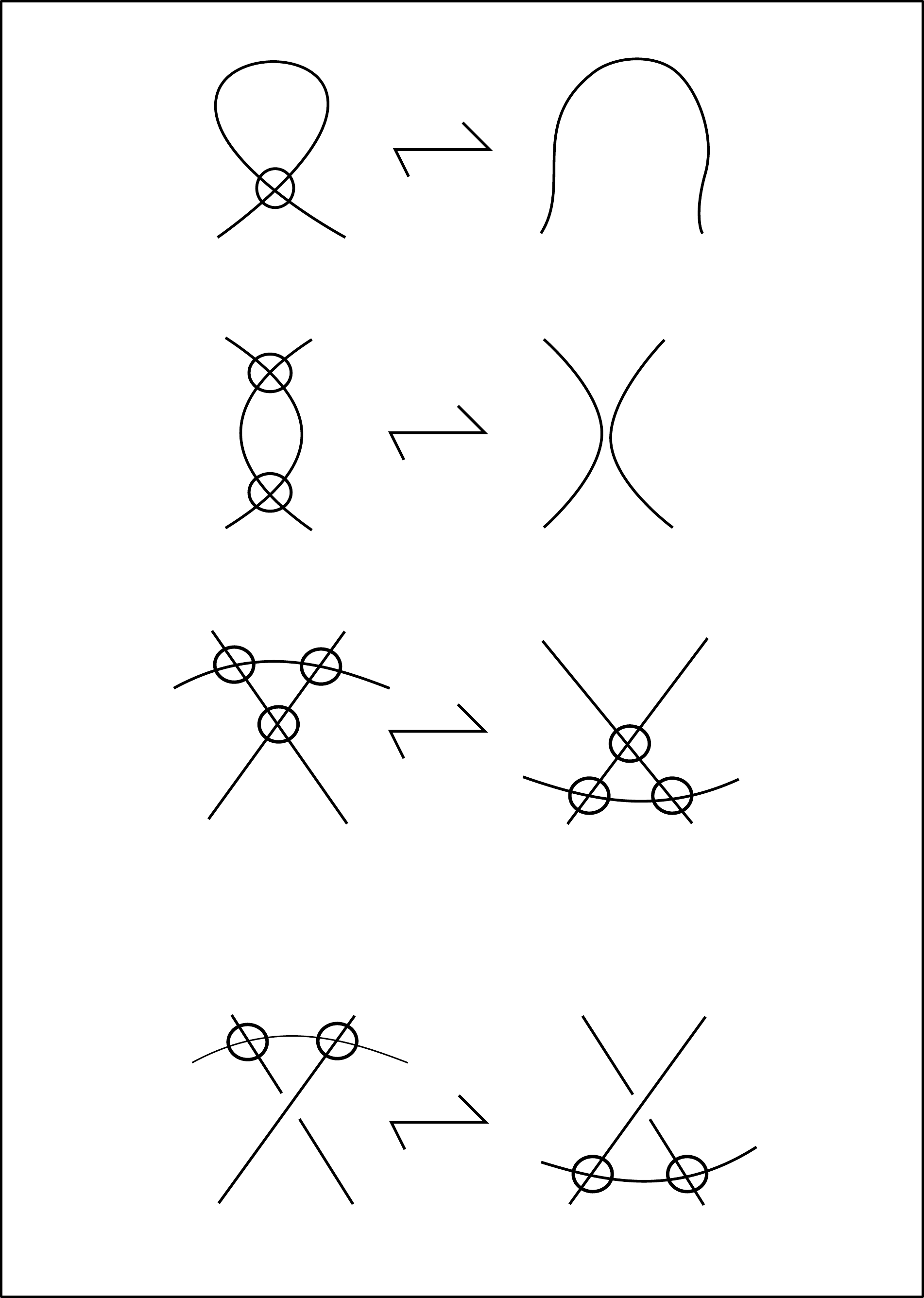}
        \caption{\bf Virtual $\Omega_{i=1,2,3}$-moves and a partial virtual move}
        \label{subfig:m_a}
    \end{subfigure}
		
		\begin{subfigure}[c]{0.5\textwidth}
      \centering  \scalebox{0.27}{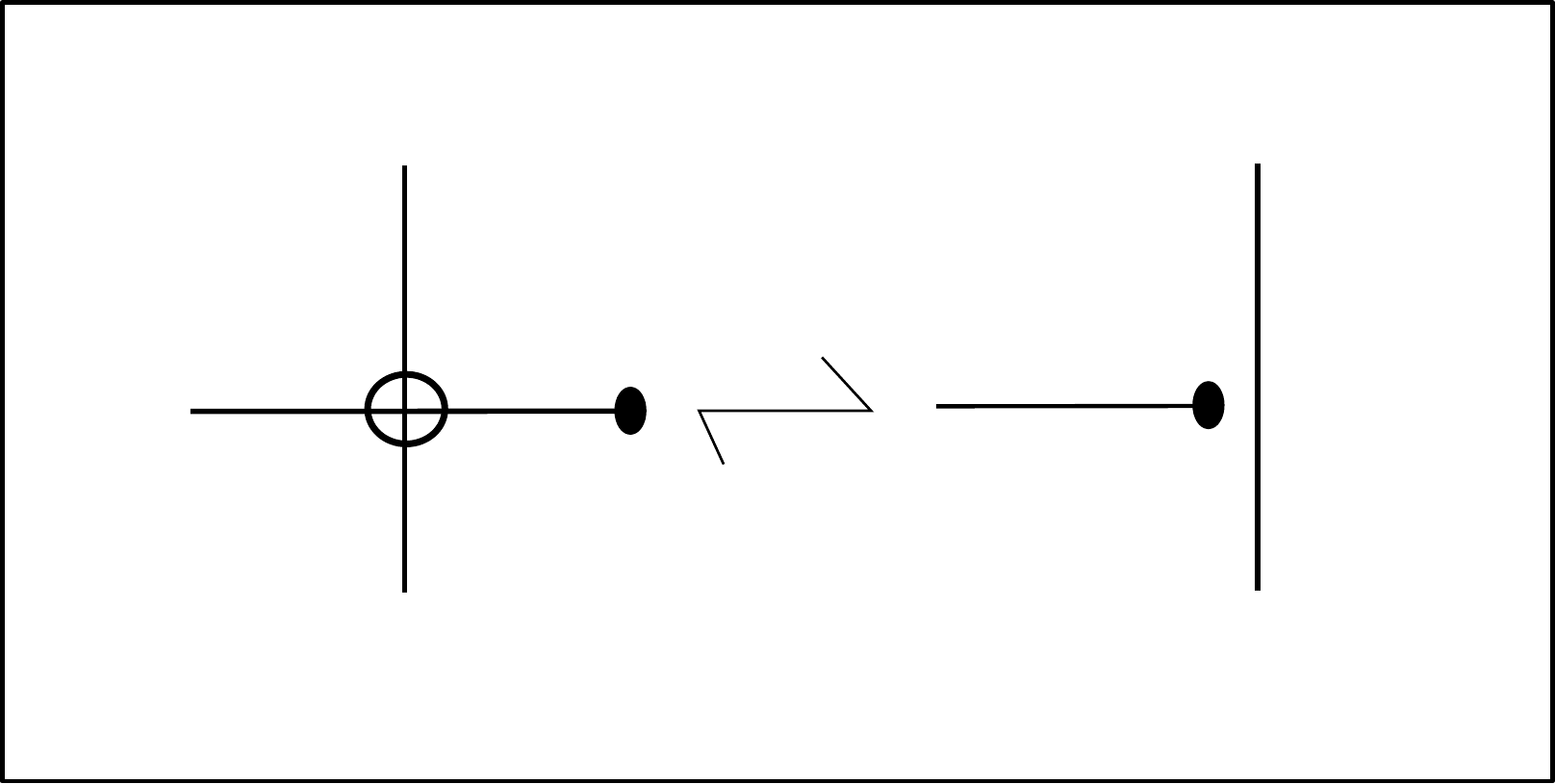}
		\caption{\bf $\Omega$-virtual move}
        \label{subfig:vir}
    \end{subfigure}
		  %\begin{subfigure}[b]{0.3\textwidth}
       % \centering  \scalebox{0.20}{\input{forbid.pdf_tex}}
        %\caption{\bf The Virtual Forbidden Moves}
        %\label{}
    %\end{subfigure}\\
		%\qquad
    %\begin{subfigure}[b]{0.3\textwidth}
     %   \centering  \scalebox{0.15}{\input{forbid.pdf_tex}}
      %  \caption{\bf The Virtual Forbidden Moves}
       % \label{subfig:forbid}
    %\end{subfigure}\\
	    ~ %add desired spacing between images, e. g. ~, \quad, \qquad, \hfill etc. 
	\caption{\bf Moves on virtual knotoid diagrams}
	\label{fig:moo}
\end{figure}
 There are two more moves on virtual knotoid diagrams shown in Figure \ref{fig:forbid} which resemble the Reidemeister moves but do not result from any of the $\Omega$-moves or the detour move.  We call them \textit{virtual forbidden moves}. The virtual forbidden moves slide either an underpassing or overpassing under/over a virtual crossing and they are denoted by $\Phi_{under}$ and $\Phi_{over}$, respectively. These moves are the forbidden moves of closed virtual knots/links since the allowance of these moves trivializes the theory of virtual knots \cite{Ne}. By an observation on the effect of the virtual forbidden moves on the corresponding chord diagrams of knotoid diagrams (see Section 4.1 for the chord diagrams), it can be shown that any virtual knotoid diagram can be transformed to the trivial knotoid diagram. The utilization of only the over-forbidden move, $\Phi_{over}$ yields a nontrivial theory called \textit{welded knot theory} \cite{Ka1, Ro}. We define the corresponding \textit{welded knotoid theory}.
\begin{definition}\normalfont
Two virtual knotoid diagrams are said to be \textit{w-equivalent} if they can be obtained from one another by a sequence of the generalized $\Omega$-moves, the over-forbidden move, $\Phi_{over}$ and the $\Phi_-$- move (see Figure $2$(b)). The corresponding equivalence classes are called \textit{welded virtual knotoids}.
\end{definition} 
 S.~Satoh \cite{Sa} defines \textit{w-equivalence} on virtual knotoid diagrams (named as \textit{virtual arc diagrams} in \cite{Sa}) just in the same way. The fundamental group of a virtual knotoid diagram is given by the generators associated to the overpasses of the diagram and at each classical crossing there is a relation defined in the same way with the relations of Wirtinger presentation \cite{GP} of knot groups. Note that the fundamental group of any knotoid diagram $K$ in $S^2$ is invariant under the $\Omega$-moves and the $\Phi_-$-move, and the fundamental group of $K$ is isomorphic to the fundamental group of the classical knot represented by the underpass closure of $K$, see \cite{Tu} for more details and also \cite{Sa} in which this concept was given in terms of w-equivalences of classical arc diagrams. Satoh shows that any two w-equivalent virtual knotoid diagrams represent equivalent ribbon $2$-knots in $\mathbb{R}^4$ and the fundamental group of the complement of any ribbon $2$-knot is isomorphic to the fundamental group of the associated welded virtual knotoid.
 \begin{figure}[H]
\begin{center}
     \begin{tabular}{c}
		%\Huge{
		\centering\scalebox{0.30}{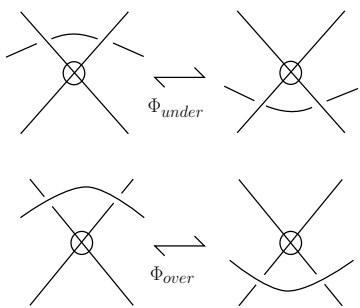}
		%}
		\end{tabular}
     \caption{\bf Virtual forbidden moves}
     \label{fig:forbid}
\end{center}
\end{figure}
\begin{definition}\normalfont
 Let $\mathcal{M}$ be a category of mathematical structures.
A virtual knotoid invariant is a mapping $I$: Virtual Knotoids $\rightarrow$ $\mathcal{M}$ such that virtually equivalent knotoids map to equivalent structures in $\mathcal{M}$.
\end{definition}
  The theory of virtual knotoids has a topological interpretation. Knotoid diagrams can be defined in higher genus, closed, connected, orientable surfaces as generic immersions of the unit interval, with two distinct endpoints as the images of $0$ and $1$, and with finitely many transversal double points each endowed with over/under-data so that they are classical crossings. %For fluency in writing, unless otherwise indicated, we mean a surface to be a closed, connected, orientable surface throughout the paper.

% we can regard each virtual crossing of a virtual knotoid diagram as a shorthand for a detour of one of the arcs of the crossing through a one-handle that is attached to the $2$-sphere of the diagram. By interpreting each virtual crossing in this way, we obtain an immersed open curve into a compact orientable surface $\Sigma_g$ where g is the number of virtual crossings in the diagram. Virtual knotoids which has a diagram that can be drawn on $S^2$ without any virtual crossings shall be called as classical knotoids. Classical knotoids are knotoids in $S^2$. 
 Let $K$ be a virtual knotoid diagram. An \textit{abstract knotoid diagram} associated to $K$, $(F,K)$ is a ribbon-neighborhood surface containing the knotoid diagram $K$. This surface is obtained by attaching a $2$-disc to each classical crossing and to the two endpoints of $K$ such that the crossings and the endpoints are contained in the discs, and connecting these discs by ribbons, as depicted in  Figure $9$. The virtual crossings are represented by ribbons that pass over one another. The abstract knotoid diagrams are pictured as embedded in $3$-dimensional space, but they are not considered as particular embeddings. The ribbons containing virtual crossings can pass over one another in either way. There is a unique abstract knotoid diagram associated to a virtual knotoid diagram. Note that an abstract knotoid diagram is a closed connected orientable surface with boundary.  
\begin{figure}[H]
\hspace{-1.5cm}
    %\centering
   \begin{subfigure}[b]{1\textwidth}
        \centering  \scalebox{0.5}{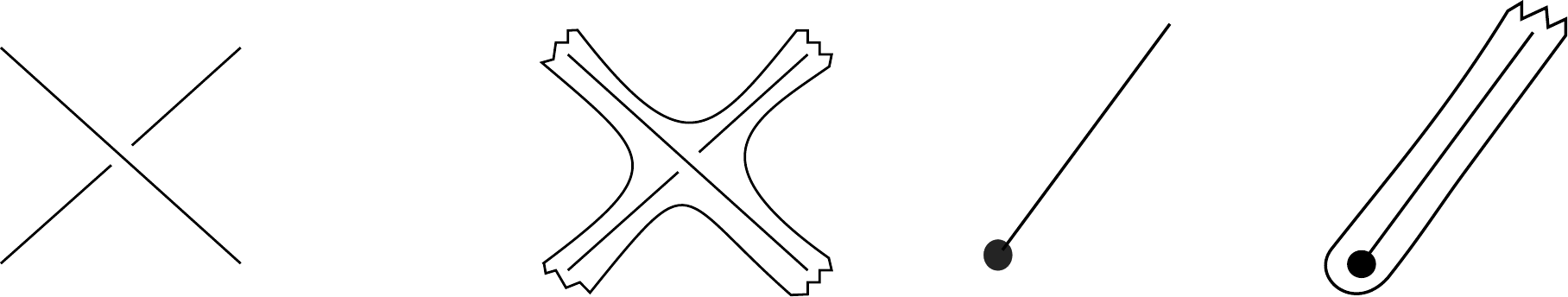}
        \caption{\bf Attaching discs to a classical crossing and to an endpoint}
        \label{fig:abs1}
    \end{subfigure}
		\vspace{3mm}
			%\hfill
			
    \begin{subfigure}[b]{1\textwidth}
        \centering  \scalebox{0.4}{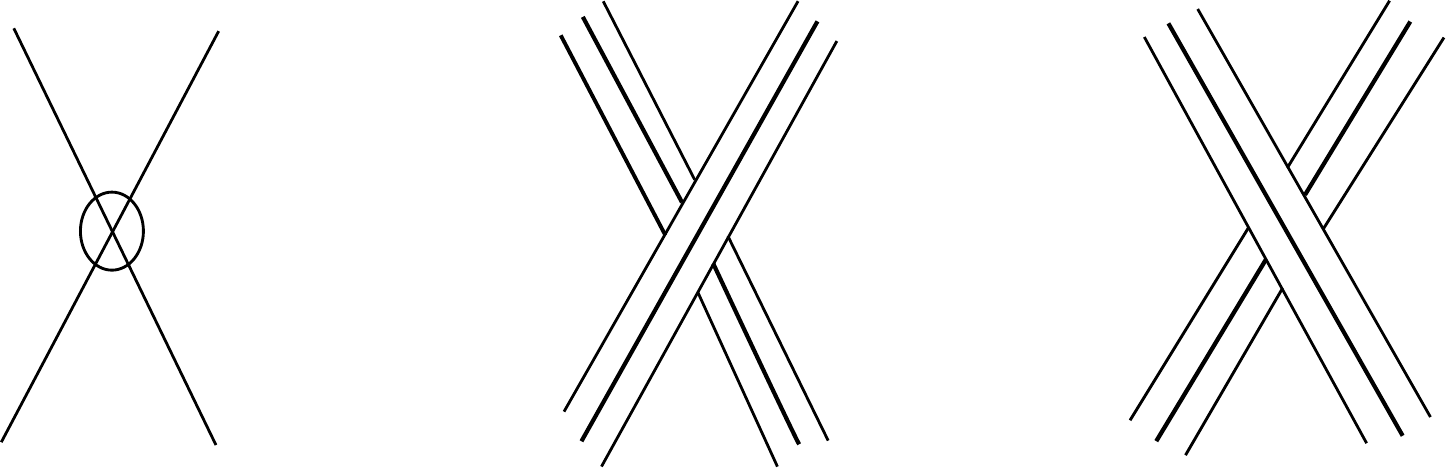}
        \caption{\bf Two ways of attaching bands to a virtual crossing}
        \label{subfig:abs}
    \end{subfigure}
		\vspace{4.5mm}
    ~ %add desired spacing between images, e. g. ~, \quad, \qquad, \hfill etc. 
		\begin{subfigure}[b]{1\textwidth}
        \centering  \scalebox{0.35}{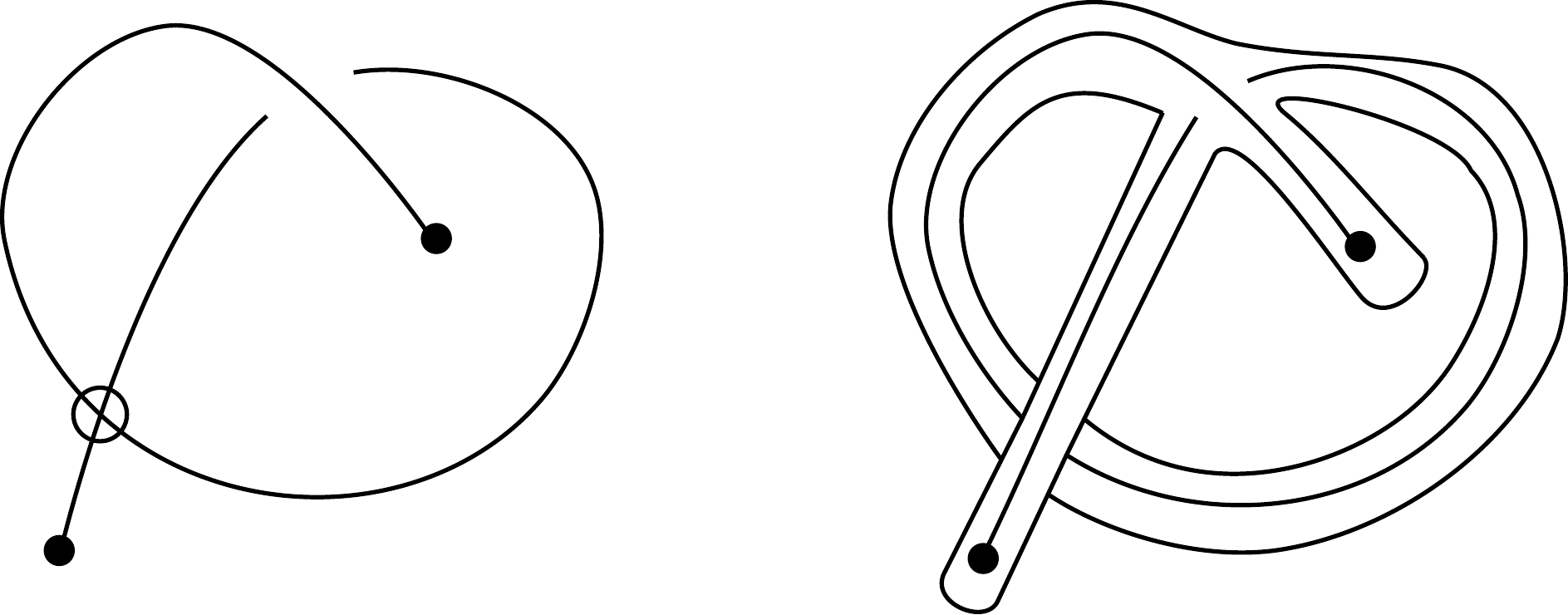}
        \caption{\bf A virtual knotoid diagram and the associated abstract knotoid diagram}
        \label{fig:abs3}
    \end{subfigure}
\caption{Abstract knotoid diagrams}
\end{figure}
 We say that two abstract knotoid diagrams are \textit{abstractly equivalent} if one can be obtained from the other one by finitely many \textit{abstract} $\Omega$-moves that are shown in Figure \ref{subfig:genabs}. Abstract $\Omega$-moves are ribbon versions of the generalized $\Omega$-moves. %The boundary components of an abstract knotoid diagram can be filled with $2$-discs. The abstract $\Omega_{i=1,2,3}$-moves can be also accomplished in these closed surfaces .
The abstract detour move is accomplished by the freedom of movement of the virtual crossings represented by non-interacting ribbon bands. An \textit{abstract knotoid} is defined to be an equivalence class of abstract knotoid diagrams under these moves.
\begin{figure}[H]
%{\huge{
\centering  \scalebox{0.65}{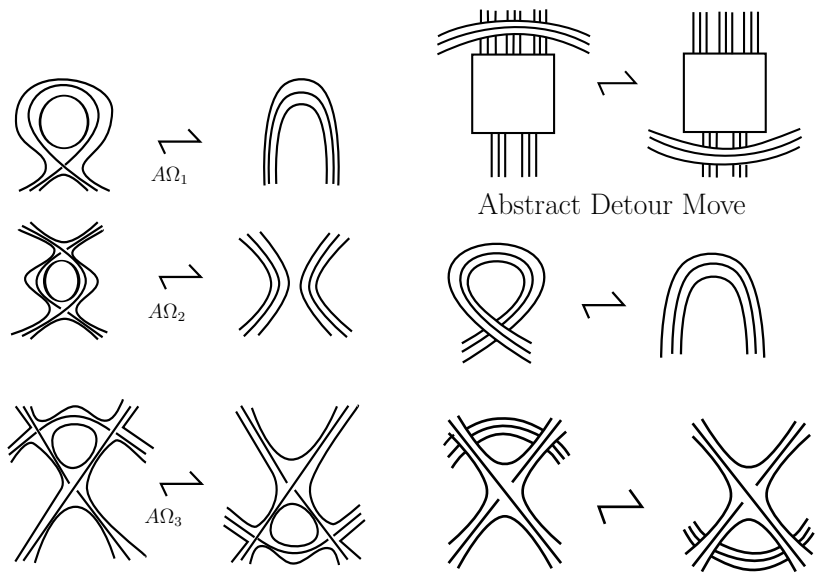}
	%			}
        \caption{\bf Generalized abstract moves }
        \label{subfig:genabs}
				\end{figure}	
\begin{prop}
The mapping
\begin{center}
$f$:Virtual Knotoid Diagrams $\rightarrow$ Abstract Knotoid Diagrams,
\end{center}
that is defined by assigning to a virtual knotoid diagram $K$ the associated abstract knotoid diagram $(F,K)$ induces a bijection
\begin{center}
$f_*$: Virtual Knotoid Diagrams/virtual eqv. $\rightarrow$ Abstract Knotoid Diagrams/abstract eqv.
\end{center}
\end{prop}
\begin{proof}
 Let $K_1$, $K_2$ be two (virtually) equivalent virtual knotoid diagrams. It is not hard to verify that any $\Omega$-move between these diagrams transforms to abstract $\Omega$- moves between the corresponding abstract knotoid diagrams, $(F_1,K_1)$ and $(F_2,K_2)$. If the given diagrams are related to each other by moves generated by the detour move then $(F_1,K_1)$ and $(F_2,K_2)$ are related by the moves generated by the abstract detour move. This shows the map $f_*$ is well-defined.

 An abstract diagram $(F,K)$ can be embedded in $S^3$ in such a way that the $2$- disks containing the classical crossings and the endpoints lie in ${S^2}\subset{S^3}$. Being an orientable surface, the abstract knotoid diagram $(F,K)$ can be projected to $S^2$ as a virtual knotoid diagram. The segments through transversal ribbon bands are projected as transversal segments and the intersection points of the transversal segments are regarded as virtual crossings of a virtual knot diagram.  It is left to the reader to check that this projection map taken with the embedding induces a well-defined map from the set of abstract knotoids to the set of virtual knotoids and this map forms the inverse of $f_*$. 
\end{proof}
 Abstract knotoid diagrams are associated to knotoids in surfaces of higher genus in the following sense. 
The abstract knotoid diagram $(F,K)$ associated to a virtual knotoid diagram $K$ is a closed connected orientable surface with boundary. The \textit{underlying graph of a virtual knotoid diagram} is the graph that is obtained by turning the classical crossings and the endpoints of $K$ into graphical vertices, and keeping the virtual crossings as they are. The underlying graph of a virtual knotoid diagram is sometimes called a \textit{virtual graph}. A virtual graph is subjected to the detour move but not the $\Omega$-moves. 

 Let $\Gamma(K)$ be the underlying graph of $K$. $\Gamma(K)$ is a connected graph with $n$ four-valent vertices corresponding to classical crossings of $K$, two one-valent vertices corresponding to the endpoints of $K$, and with $2n+1$ edges. It is a consequence of the construction of $(F,K)$ that the graph $\Gamma(K)$ is a deformation retract of $(F,K)$.  We close the boundary components of $(F,K)$ with $2$-disks to have a representation of the virtual knotoid $K$ in a closed connected orientable surface, denoted by $\overline{(F,K)}$. Let $\delta$ be the number of boundary components of $(F,K)$. Then the Euler characteristic of $\overline{(F,K)}$ is equal to  $(n+2)-(2n+1)+\delta=1-n+\delta$ and the genus of$\overline{(F,K)}$ is equal to $1+((n-1)-\delta)/2$.

 The closure $\overline{(F,K)}$ is the least genus surface among the surfaces in which the knotoid diagram $K$ can be immersed without any virtual crossings. We can add extra handles in the complement of $K$ so that $K$ is represented by a diagram without any virtual crossings in other surfaces with higher genus. On the other hand, let be given a knotoid diagram $K$ in a surface of genus $\widetilde{g}$, $\Sigma_{\widetilde{g}}$. The regular neighborhood of the diagram $N(K)$ can be regarded as an abstract knotoid diagram $(N(K),K)$ immersed in $\Sigma_{\widetilde{g}}$. If the complement of $(N(K),K)$ has genus then we cut out this extra genus to reduce the genus $\tilde{g}$ to the genus of $\overline{(N(K),K)}$.
\begin{definition}
Let $K_1$, $K_2$ be two knotoid diagrams in surfaces $\Sigma_1$, $\Sigma_2$, respectively. The surface representations $(\Sigma_1, K_1)$ and $(\Sigma_2, K_2)$ are said to be \textit{stably equivalent} if one is obtained from the other by isotopy of the surfaces, diffeomorphisms of the surfaces and the addition/subtraction of empty handles in the complement of the diagrams. 
\end{definition}
\begin{prop}
The mapping
\begin{center}
\hspace{-3cm} $\tilde{f}$: Abstract Knotoid Diagrams $\rightarrow$ Knotoid Diagrams in Surfaces of genus $g$ 
\end{center}
defined by assigning to $(F,K)$ the knotoid diagram in the closure $\overline{(F,K)}$ induces a bijection
\begin{center}
$\tilde{f}_*$: Abstract Knotoid Diagrams/abstract eqv. $\rightarrow$ Knotoid Diagrams in Surfaces/stable eqv.
\end{center}
\end{prop}
\begin{proof}
 It is easy to see that by filling the boundary components of an abstract knotoid diagram with $2$-disks, abstract $\Omega_1$ and $\Omega_3$-moves are transformed to $\Omega_1$- and $\Omega_3$-moves between the knotoid diagrams represented in the resulting surfaces of the same genus. The genus does not change under these two moves. An abstract $\Omega_2$-move may increase/decrease the genus of the surface by $1$. In the case of a change in the genus, an abstract $\Omega_2$-move corresponds to $\Omega_2$-move plus removal/addition of empty handles in the surface. Thus the map $\tilde{f}_*$ is well-defined. 

 Let $K$ be a knotoid diagram in a surface $\Sigma_g$. The regular neighborhood of the diagram in $\Sigma_g$ is an abstract knotoid diagram $(N(K),K)$. The closure of $(N(K),K)$ with $2$-disks, $\overline{(N(K),K)}$ is stably equivalent to $(\Sigma_g,K)$ since $\overline{(N(K),K)}$ is the least genus surface in which $K$ is given without any virtual crossings. So, the map is surjective.
Addition/removal of handles occur in the complement of $K$ in the surface. Thus $(N(K),K)$ is not affected by these moves. An $\Omega$-move on $K$ transforms to an abstract Reidemeister move on $(N(K),K)$ as the reader can check easily. Therefore the map $\tilde{f}_*$ is injective. This completes the proof of Proposition 3.3.
\end{proof}
%that virtual knotoids are the equivalence classes of immersions of the unit interval into a compact orientable surface $\Sigma_g$ where g is the number of virtual crossings in the spherical diagram of the knotoid, upto the isotopy of the surface, diffeomorphisms of the surface and one-handle stabilization.
 The following theorem is stated in \cite{Tu} as a remark. We give a proof of the theorem.
\begin{thm}
The theory of virtual knotoids is equivalent to the theory of knotoid diagrams in higher genus surfaces considered up to isotopy in the surface, diffeomorphisms of the surface and addition/removal of handles in the complement of knotoid diagrams.
\end{thm}
\begin{proof}
 The composition of the two bijections $f_*$ and $\tilde{f}_*$ gives a bijection between the virtual knotoids and knotoids in higher genus surfaces up to the stable equivalence.
\end{proof}
 Projecting a knotoid diagram that lies in a higher genus surface to $S^2$ results in virtual crossings. We make this projection canonical by forming the abstract knotoid diagram in the surface and then arranging a standard projection of the abstract diagram. Figure \ref{fig:projection} depicts the projection process.

 The \textit{genus} of a knotoid is the least genus among the surfaces in which the knotoid can be immersed without any virtual crossings. Virtual knotoids that can be represented by a classical knotoid diagram are called \textit{genus $0$- knotoids}. At the time of writing this paper we do not know if the theory of knotoids in $S^2$ embeds into the theory of virtual knotoids. We conjecture the following.
\begin{conj}
If two knotoid diagrams in $S^2$ are virtually equivalent then they are equivalent to each other by finitely many $\Omega$- moves and isotopy of $S^2$.
\end{conj}
 Note that any virtual knotoid invariant is also an invariant for classical knotoids since the generalized $\Omega$- moves include the $\Omega$- moves.
\begin{rem}
\normalfont
The definitions we have used so far, and Theorem 3.2 directly generalizes to virtual multi-knotoid diagrams. A \textit{virtual multi-knotoid diagram} is an immersion of finitely many oriented circles and oriented unit intervals into $S^2$ with finitely many transversal double points that correspond to classical and virtual crossings. The virtual equivalence defined for virtual knotoids generalizes to an equivalence on virtual multi-knotoids in the obvious way.
\end{rem}
\begin{figure}[H]
 \begin{center}
   \begin{tabular}{c}
     \centering  \scalebox{0.75}{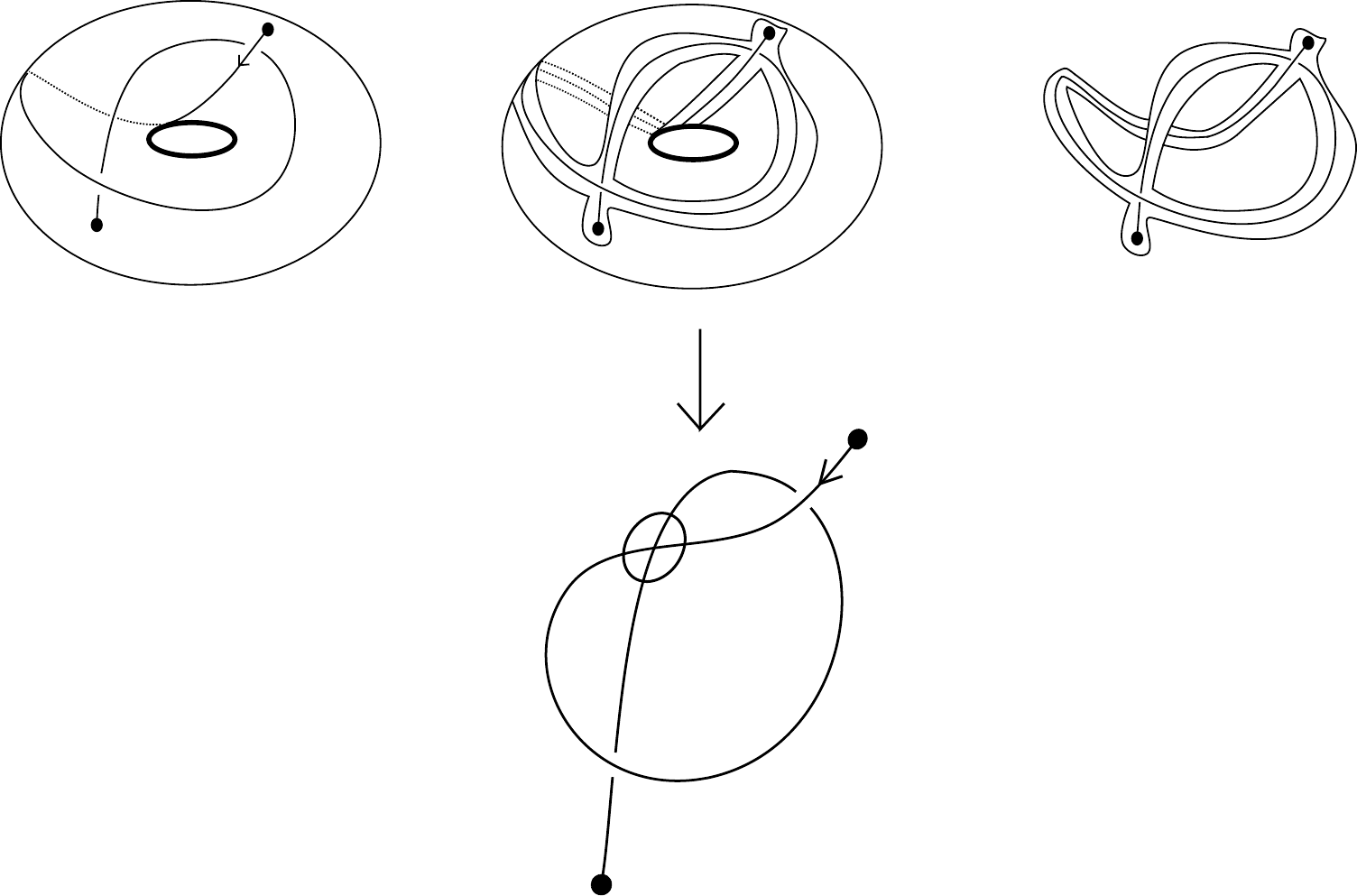}
     \end{tabular}
     \caption{\bf Virtual knotoid and abstract knotoid diagram  }
     \label{fig:projection}
\end{center}
\end{figure}
\subsection{Flat Knotoids}
\normalfont
 A \textit{flat knotoid diagram} is a diagram in $S^2$ (or in $\mathbb{R}^2$) with \textit{flat crossings} consisting in the transversal intersections of strands without any under/over-crossing information and two endpoints that are distinct from each other and from any other crossings. These endpoints are named in the same way with endpoints of knotoids, as the \textit{tail} and the \textit{head} of the diagram. \textit{Flat $\Omega_1$,$\Omega_2$,$\Omega_3$- moves} are defined on flat knotoid diagrams by ignoring the under/over- crossing information at the crossings of the move patterns $\Omega_1$, $\Omega_2$ and $\Omega_3$, respectively. These moves are referred as \textit{flat $\Omega$}-moves. The flat $\Omega$-moves and isotopy of $S^2$ (or isotopy of $\mathbb{R}^2$, respectively) induces an equivalence relation on flat knotoid diagrams, called $f$-equivalence. A \textit{flat knotoid} is defined to be an equivalence class of flat knotoid diagrams with respect to the $f$-equivalence. The analogs of $\Phi_{\pm}$-moves where the adjacent strand to the tail or the head passes through a transversal strand so that it creates/removes a flat crossing, remain forbidden for flat knotoids. 

  A \textit{flat virtual knotoid diagram} is defined to be a flat knotoid diagram with also \textit{virtual crossings} as we have described them. The detour move is defined in the same way as it is defined for virtual knotoid diagrams. The rules for changing flat crossings among themselves are identical with the rules for changing virtual crossings. A special case of the detour move, a \textit{flat partial virtual move} is available for virtual crossings with respect to flat crossings when classical crossings in the partial virtual moves are replaced by flat crossings. The moves obtained by replacing classical crossings in the forbidden moves given in Figure \ref{fig:forbid} by flat crossings, remain forbidden for flat virtual knotoid diagrams. The moves on flat knotoid diagrams that are generated by flat $\Omega$-moves and the detour move, are called  \textit{generalized flat $\Omega$- moves.}
  
 Two flat virtual knotoid diagrams are said to be \textit{f-equivalent} if there is a finite sequence of generalized flat $\Omega$- moves and isotopy of $S^2$ taking one diagram to the other. A \textit{flat virtual knotoid} is defined to be an equivalence class of flat knotoids diagrams with respect to this equivalence. 

 We say that a virtual knotoid diagram $K$ \textit{overlies} a flat virtual knotoid diagram if it is obtained from the flat diagram by choosing a crossing type as over or under for each flat crossing. The flat virtual diagram that $K$ overlies, is the \textit{underlying flat diagram} of $K$ and denoted by $F(K)$. It is clear that any generalized $\Omega$-move on $K$ induces a flat generalized $\Omega$-move on the underlying flat diagram $F(K)$. It follows that  if $K$ and $\widehat{K}$ are two virtually equivalent virtual knotoid diagrams then the underlying flat diagrams, $F(K)$ and $F(\widehat{K})$ are f-equivalent. Thus, a virtual knotoid diagram is necessarily nontrivial if it overlies a nontrivial flat virtual knotoid. Clearly, this argument holds for flat knotoid diagrams in $S^2$ or in $\mathbb{R}^2$. A classical knotoid diagram is nontrivial if it overlies a montrivial flat knotoid diagram. 

 It is well-known that any flat classical knot diagram is equivalent to the trivial knot diagram. This property of flat classical knots generalizes to flat knotoid diagrams in $S^2$ as explained below.
\begin{prop}
Any flat knotoid in $S^2$ is f-equivalent to the trivial knotoid.
\end{prop}
\begin{proof}
Any knotoid diagram in $S^2$ is equivalent to a \textit{normal knotoid diagram} that lies in $\mathbb{R}^2$ with its tail lying in the outermost region (in the unbounded region of the plane) of the diagram. This equivalence is obtained by an isotopy of ${S^2}={\mathbb{R}^2} \cup {\infty}$ \cite{Tu}. A flat knotoid diagram in $\mathbb{R}^2$ is said to be \textit{normal} if its tail in the outermost region of the diagram. Similarly with the argument above, any flat knotoid diagram in $S^2$ can be represented by a flat normal knotoid diagram. It is clear that two flat normal knotoid diagrams represent the same flat knotoid in $S^2$ if and only if they are related to each other by a finite sequence of flat $\Omega$-moves and planar isotopy.
%Any flat classical knotoid diagram in $S^2$  as classical knotoid diagrams are turned into normal knotoid diagrams \cite{Tu}.

 An \textit{ascending knotoid diagram} is a classical knotoid diagram that consists of crossings encountered firstly as an undercrossing while traversing the diagram from its tail to its head. Clearly, a flat normal knotoid diagram is f-equivalent to the trivial knotoid diagram if and only the ascending normal knotoid diagram overlying this flat diagram  is equivalent to the trivial knotoid diagram. We claim that any ascending normal knotoid diagram is equivalent to the trivial knotoid diagram. To prove our claim,  we first show that any open-ended space curve corresponding to a normal ascending knotoid diagram, is line isotopic to the trivial space curve with two endpoints attached to the special lines. Then by Theorem 2.2, it follows that an ascending normal knotoid diagram represents the trivial knotoid in $\mathbb{R}^2$, so in $S^2$. 
 
 Let $K$ be an ascending normal knotoid diagram. Let $l_1$ and $l_2$ be the two lines that are passing through the tail and the head, respectively, and perpendicular to the $xy$-plane. We fix the tail at the point $(x,y,0)$ (on the plane) on $l_1$ and start raising $K$ in the vertical direction by pulling the head up along the line $l_2$. The head is pulled up until $K$ becomes a helical space curve $c(K)$. See Figure \ref{fig:flt} for an illustration of this. % so that each crossing of $K$  one another at different $z$-levels. 
 Notice that the curve $c(K)$ can be isotoped to a curve that does not wind around the line $l_1$ since $K$ is a normal diagram. %The part of the space curve $c(K)$ any of the lines $l_1$, $l_2$, is projected to an ascending long knot diagram. It is well-known that ascending (classical) long knot diagrams are equivalent to the trivial long knot. 
Then the curve $c(K)$ is line isotopic to a curve with one endpoint on the line $l_1$ and the rest winds around $l_2$,  where the other endpoint is attached. The part of the curve $c(K)$ that winds around the line $l_2$ together with the line $l_2$ that is oriented upwards, can be regarded as a $2$-braid. By the line isotopy the parts that correspond to a braid word $\sigma_1\sigma_1^{-1}$, are eliminated so that the curve $c(K)$ corresponds to a braid word $\sigma_1^n \in B_2$, for some $n \geq 0$. We start unwinding $c(K)$ from the top by a rotation of $180$-degrees in the counterclockwise direction around the line $l_2$. Applying $n$ consecutive rotations around the line $l_2$ transforms the curve $c(K)$ into the trivial curve. In other words, the curve $c(K)$ is line isotopic to the trivial curve. Then the projection of $c(K)$ to $\mathbb{R}^2$ is the trivial knotoid diagram by Theorem 2.2. This proves that any ascending normal knotoid diagram is equivalent to the trivial knotoid. Therefore, by the argument above, any flat normal knotoid diagram is $f$-equivalent to the trivial knotoid diagram. Since any flat knotoid diagram in $S^2$ can be represented by a flat normal diagram, the statement follows. 
\end{proof}
\begin{figure}[H]
 \begin{center}
   \begin{tabular}{c}
     \centering  \scalebox{0.60}{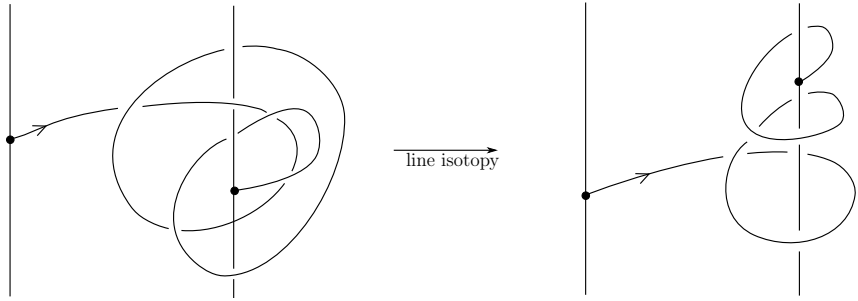}
     \end{tabular}
     \caption{\bf An ascending knotoid diagram in $\mathbb{R}^3$}
     \label{fig:flt}
\end{center}
\end{figure}

 Note that Proposition 3.6 does not hold for flat knotoids in $\mathbb{R}^2$. The ascending knotoid diagram $K$ given in Figure 1(b), when considered to be defined in the plane, is not equivalent to the trivial knotoid \cite{Tu}. It follows that the underlying flat diagram of the diagram $K$ is not $f$-equivalent to the trivial knotoid diagram. Also, unlike the flat knotoids in $S^2$, there are flat virtual knotoids which are non-trivial. In section 4.2 we discuss on an example considering Figure \ref{fig:trivialvc} whose underlying flat diagram is a nontrivial flat virtual knotoid. 

 \subsection{The Virtual Closure of Knotoids}
Every knotoid diagram in $S^2$ represents a virtual knot as pointed out in \cite{Tu}. The endpoints of a knotoid diagram can be connected with an embedded arc in $S^2$ but this time a virtual crossing is created every time the connection arc crosses a strand of the diagram, as depicted in Figure \ref{fig:vc}. The resulting virtual knot diagram can be represented in a torus by attaching a $1$-handle to $S^2$ which holds the connection arc. Connecting the endpoints of a knotoid diagram in $S^2$ in the virtual fashion induces a well-defined map from the set of classical knots to the set of virtual knots of genus at most $1$. This map is called the \textit{virtual closure map} and is denoted by $\overline{v}$,
\begin{center}
$\overline{v}$: Knotoids in $S^2$ $\rightarrow$ Virtual knots of genus $\leq$ $1$. 
\end{center}
%\begin{definition}\normalfont	
\begin{figure}[H]
%\begin{center}
     \begin{tabular}{c}
     \centering  \scalebox{0.5}{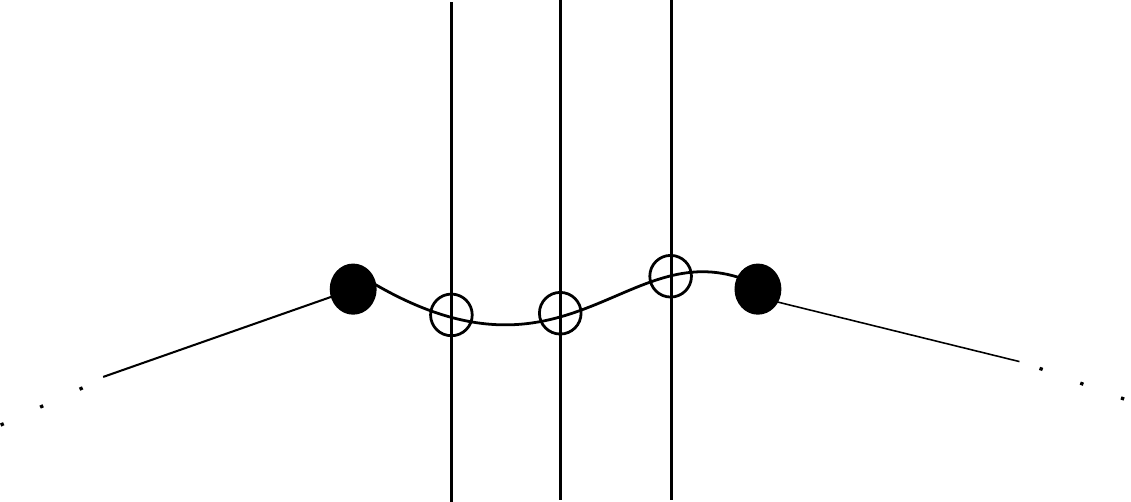}
     \end{tabular}
     \caption{\bf The virtual closure of a knotoid diagram}
     \label{fig:vc}
\end{figure}
 
 The connection arc is unique up to isotopy of $S^2$. The isotopy between any two connection arcs induces detour moves between the corresponding virtual knot diagrams. So, the choice of a connection arc does not alter the isotopy class of the resulting virtual knot. Also, it is clear that an $\Omega$-move on a knotoid diagram is transformed to a combination of generalized Reidemeister moves on the resulting knot diagram. Therefore, the virtual closure map is a well-defined map. The virtual knot assigned to a knotoid $K$ in $S^2$ via the virtual closure map is called the \textit{virtual closure} of $K$, and is denoted by $\overline{v}(K)$.

 A knot-type knotoid diagram becomes a classical knot diagram when the endpoints are connected virtually. In fact the virtual closure of a knot-type knotoid is a classical knot. Note that the underpass closure and the virtual closure of a knot-type knotoid are the same classical knots. The virtual closure of a proper knotoid diagram is a virtual knot diagram with the virtual crossings that are positioned consecutively over the connection arc. It is natural to ask if the virtual closure of a proper knotoid can be isotopic to a classical knot. At the time of writing this paper, we do not have the answer of this question.

 If a virtual knot is represented by a diagram with virtual crossings consecutively positioned on the same strand, then it is immediate to conclude that this virtual knot is in the image of the virtual closure map. In particular, the knots 2.1,3.2, 4.12, 4.43, 4.65, 4.94, 4.100 listed in \cite{Gr} are genus one virtual knots that are the virtual closures of some knotoids in $S^2$. For a virtual knot represented by a diagram with arbitrarily positioned virtual crossings, we want to know whether this virtual knot is in the image of the virtual closure map. For instance, the knotoid diagram given in Figure \ref{fig:virt3} which is listed as virtual knot 3.1 in \cite{Gr}, is a nontrivial genus one virtual knot. We prove in \cite{GK1} that $k$ is not in the image of the virtual closure map by using an extended version of the bracket polynomial of knots in higher genus surfaces that is introduced in \cite{DK1}. From this, it follows that the virtual closure map is not surjective.

 Figure \ref{fig:vc_ex} shows an example of a pair of knotoid diagrams $K_1$, $K_2$, whose virtual closures are the same virtual knot. It can be seen that the underpass closure of $K_1$ is the trefoil knot, and the underpass closure of $K_2$ is the unknot. As we discussed before, the underpass closure map, $\omega_-$ is a well-defined map on the set of knotoids in $S^2$. Thus, $K_1$ and $K_2$ are two nonequivalent knotoid diagrams , and so, the virtual closure map is not an injective map.
\begin{figure}[H]
     \begin{center}
     \begin{tabular}{c}
     \centering  \scalebox{0.5}{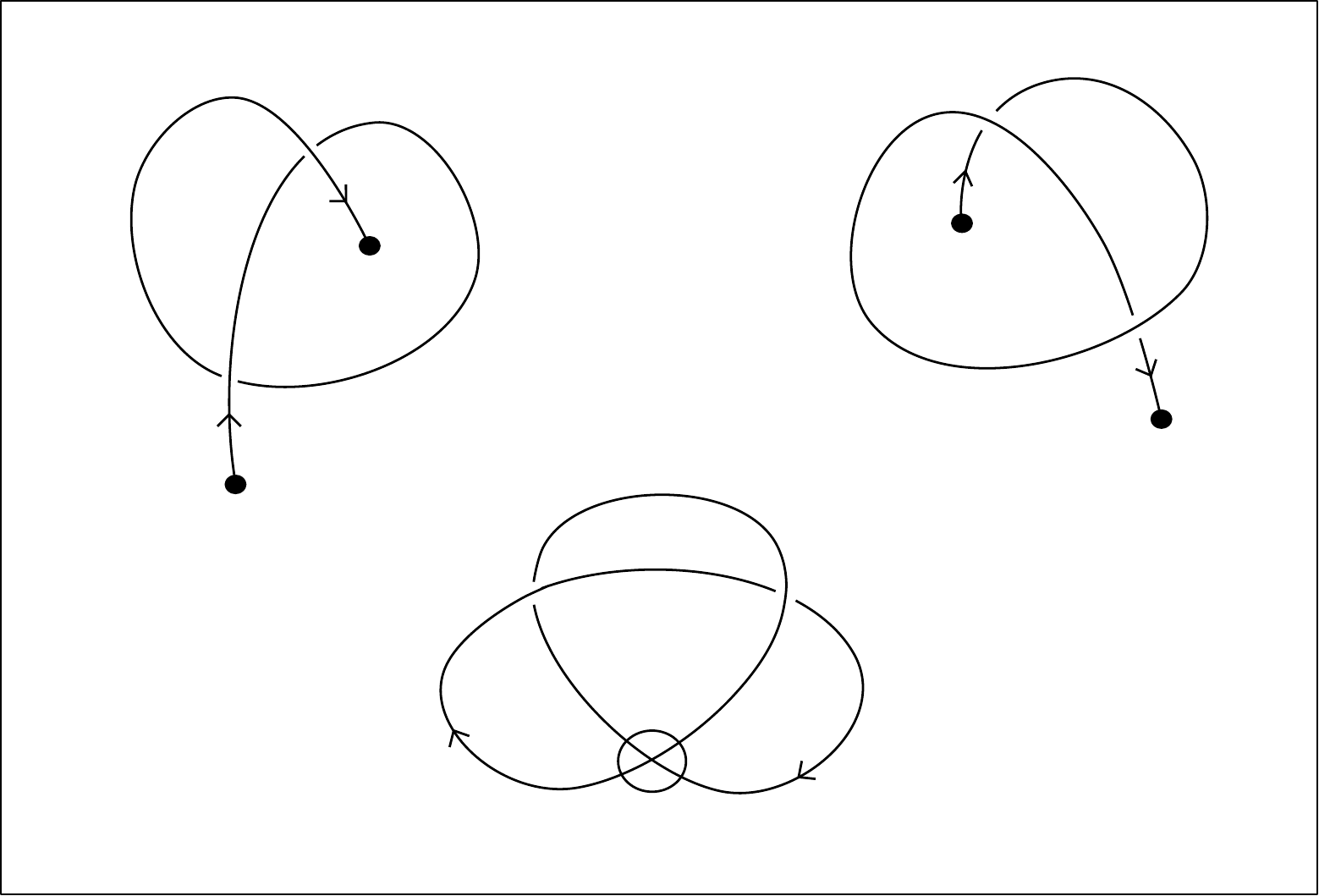}
     \end{tabular}
     \caption{\bf Nonequivalent classical knotoids with the same virtual closure}
     \label{fig:vc_ex}
\end{center}
\end{figure}
\subsection{The bracket polynomial of knotoids}
 The bracket polynomial of knotoids in $S^2$ or $\mathbb{R}^2$\cite{Tu} is defined by extending the state expansion of the bracket polynomial of knots \cite{Ka5, Ka7}. Each classical crossing of a classical knotoid diagram $K$ is smoothed either by \textit{A-} or \textit{B-type smoothing}, as shown in Figure \ref{fig:bracket}. A smoothing site is labeled by $1$ if $A$-smoothing is applied and labeled by $-1$ if $B$-smoothing is applied at a particular crossing. A \textit{state} of the knotoid diagram $K$ is a choice of smoothing each crossing of $K$ with the labels at smoothing sites. Each state of $K$ consists of disjoint embedded circular components and a single long segment component with two endpoints. The initial conditions given in Figure \ref{fig:bracket} are sufficient for the skein computation of the bracket polynomial of classical knotoids. 

 \begin{definition}
The bracket polynomial of a classical knotoid diagram $K$ is defined as
\begin{center}
$<K>=\sum_S A^{\sigma(S)}d^{\|S\|-1}$,
\end{center}
where the sum is taken over all states, $\sigma(S)$ is the sum of the labels of the state $S$, $\|S\|$ is the number of components of $S$, and $d=(-A^2-A^{-2})$.
\end{definition}
\begin{figure}[H]
\begin{center}
     \begin{tabular}{c}
		\Large{
     \centering  \scalebox{0.6}{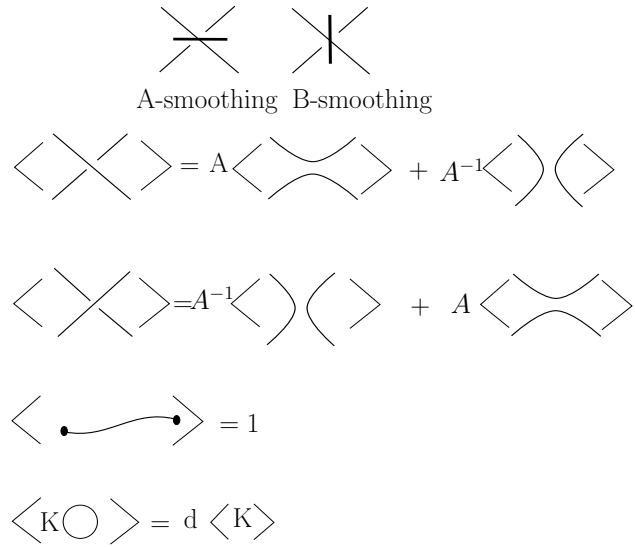}
		}
     \end{tabular}
     \caption{\bf Skein relations of the bracket polynomial}
     \label{fig:bracket}
\end{center}
\end{figure}
 	The \textit{writhe} of a classical or virtual knotoid diagram $K$, $\writhe(K)$ is the number of positive crossings (the classical crossings with sign $+1$) minus the number of negative crossings (the classical crossings with sign $-1$) of $K$. The writhe is invariant under the generalized $\Omega$-moves except that $\Omega_1$-move changes the writhe by $\pm 1$. The bracket polynomial turns into an invariant for classical knotoids with a normalization by the writhe. The \textit{normalized bracket polynomial} of a classical knotoid $K$, $f_K$ is defined as the multiplication, $f_K=(-A^3)^{-wr(K)}<K>$ \cite{Tu}.

 The normalized bracket polynomial of knotoids in $S^2$ generalizes the Jones polynomial of classical knots with the substitution $A=t^{-{1/4}}$ . Note that by connecting the endpoints of the long segment components of states of a knotoid in $S^2$, $K$ by an embedded arc in the virtual fashion, we obtain the bracket state components of the virtual knot $\overline{v}(K)$. This gives us the equality, $V(K)=V(\overline{v}(K))$, where $V(K)$ denotes the Jones polynomial of $K$. The Jones polynomial of the trivial knotoid is trivial.
\begin{figure}[H]
     \begin{center}
     \begin{tabular}{c}
     \centering  \scalebox{0.65}{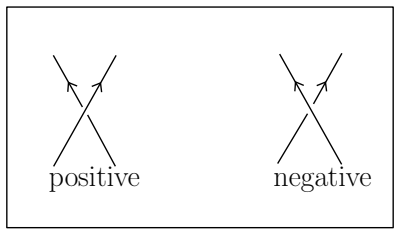}
     \end{tabular}
     \caption{\bf Crossing types}
     \label{fig:cr}
\end{center}
\end{figure}
\begin{example}\normalfont
Let $K_1$ be the knotoid diagram illustrated in Figure \ref{fig:brac}. As we compute in the figure, the bracket polynomial of $K_1$,  $<K_1>=A^2+1-A^{-4}$. This implies that $K_1$ is a non-trivial knotoid.
\end{example}
\begin{figure}[H]
\begin{center}
     \begin{tabular}{c}
     \centering  \scalebox{0.65}{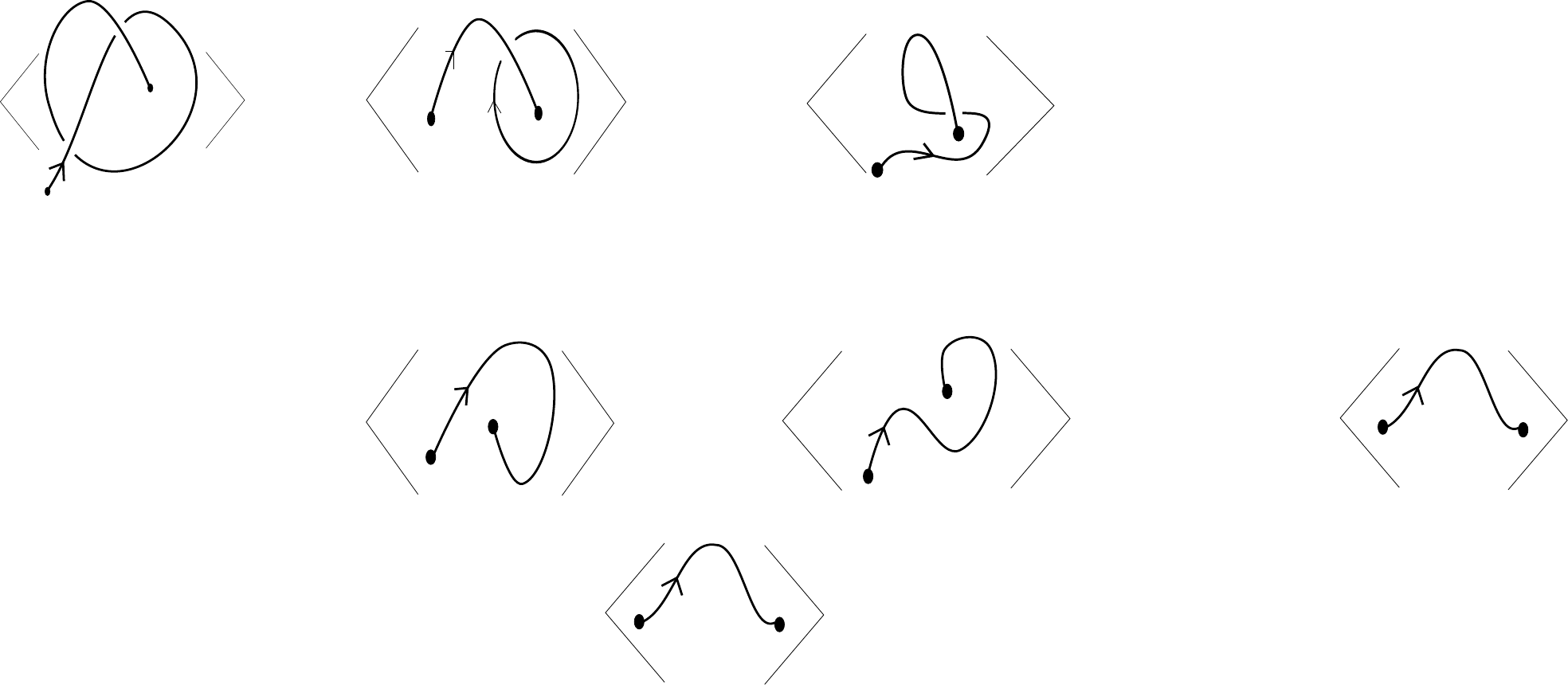}
     \end{tabular}
     \caption{\bf Computation of the bracket polynomial of $K_1$}
     \label{fig:brac}
\end{center}
\end{figure}	
	The well-known Jones polynomial conjecture can be extended to a conjecture for knotoids in $S^2$.
\begin{conj}
The normalized bracket polynomial of knotoids in $S^2$ (or the Jones polynomial) detects the trivial knotoid. 
\end{conj}
\begin{remark}\normalfont
\begin{enumerate}

\item This generalization of the knot detection conjecture for the classical Jones polynomial may have an analogue in Khovanov homology for knotoids. We shall explore this question in a sequel to the present paper. The reader should note that Khovanov homology detects the unknot \cite{KM}.

 \item Another way to show that the virtual knot given in Figure \ref{fig:virt3} is not in the image of the virtual closure map could be the following. The virtual knot given in the figure has trivial Jones polynomial \cite{Ka1}. If this knot were in the image of the virtual closure map then there would be a nontrivial classical knotoid $K$ such that $k=\overline{v}(K)$. The knotoid $K$ would have  unit Jones polynomial since the Jones polynomial of $K$ is equal to the Jones polynomial of its closure, which is the knot $k$. This would contradict Conjecture 3.7.
 
\item  The virtual closure map extends to a well-defined map from the set of virtual knotoids to the set of virtual knots. We call this map \textit{extended virtual closure map}. It is clear that the extended virtual closure map is a surjective map. Any virtual knotoid diagram whose endpoints can be connected by an embedded arc without creating any type of crossings (neither classical or virtual), can be regarded as a long virtual diagram. The virtual knotoid shown in Figure \ref{fig:trivialvc} is a nontrivial virtual knotoid. We show the non-triviality of this knotoid in Section 4.2 by the parity bracket polynomial of knotoids. The reader can verify easily that the extended virtual closure of this virtual knotoid is the trivial knot. Obviously, trivial virtual kntooid diagrams are also sent to the trivial knot by the map. Thus, the extended virtual closure map is not injective.

\item The kernel of the virtual closure map is an interesting structure to explore. A knot-type knotoid can be represented by a diagram with endpoints in the same region of the diagram. Thus if a knot type-knotoid is nontrivial then the virtual closure is classical and nontrivial. If we restrict the virtual closure map to knot-type knotoids then the kernel is trivial. At the time of writing this paper we do not know the answer of the following question. \textit{Does there exist a proper knotoid whose virtual closure is the trivial knot?} If conjecture 3.7 holds then we can answer this question as follows. The Jones polynomial of a classical knotoid is the same as the Jones polynomial of the virtual closure of the knotoid. The conjecture implies that there does not exist any proper knotoid whose virtual closure is the trivial knot. 

\item  The underpass closure map can not be extended to a well-defined map on virtual knotoids. Figure \ref{fig:different} depicts a virtual knotoid diagram that represents two virtual knots via the underpass closure map.  The arrow polynomial which will be discussed in Section 5, detects that these knots are not equivalent. In fact, to transform one diagram to the other one, we require the virtual forbidden move, $\Phi_{under}$ which is shown in Figure \ref{fig:forbid}.
\end{enumerate}
	\end{remark}
\begin{figure}[H]
\begin{center}
     \begin{tabular}{c}
     \centering  \scalebox{0.25}{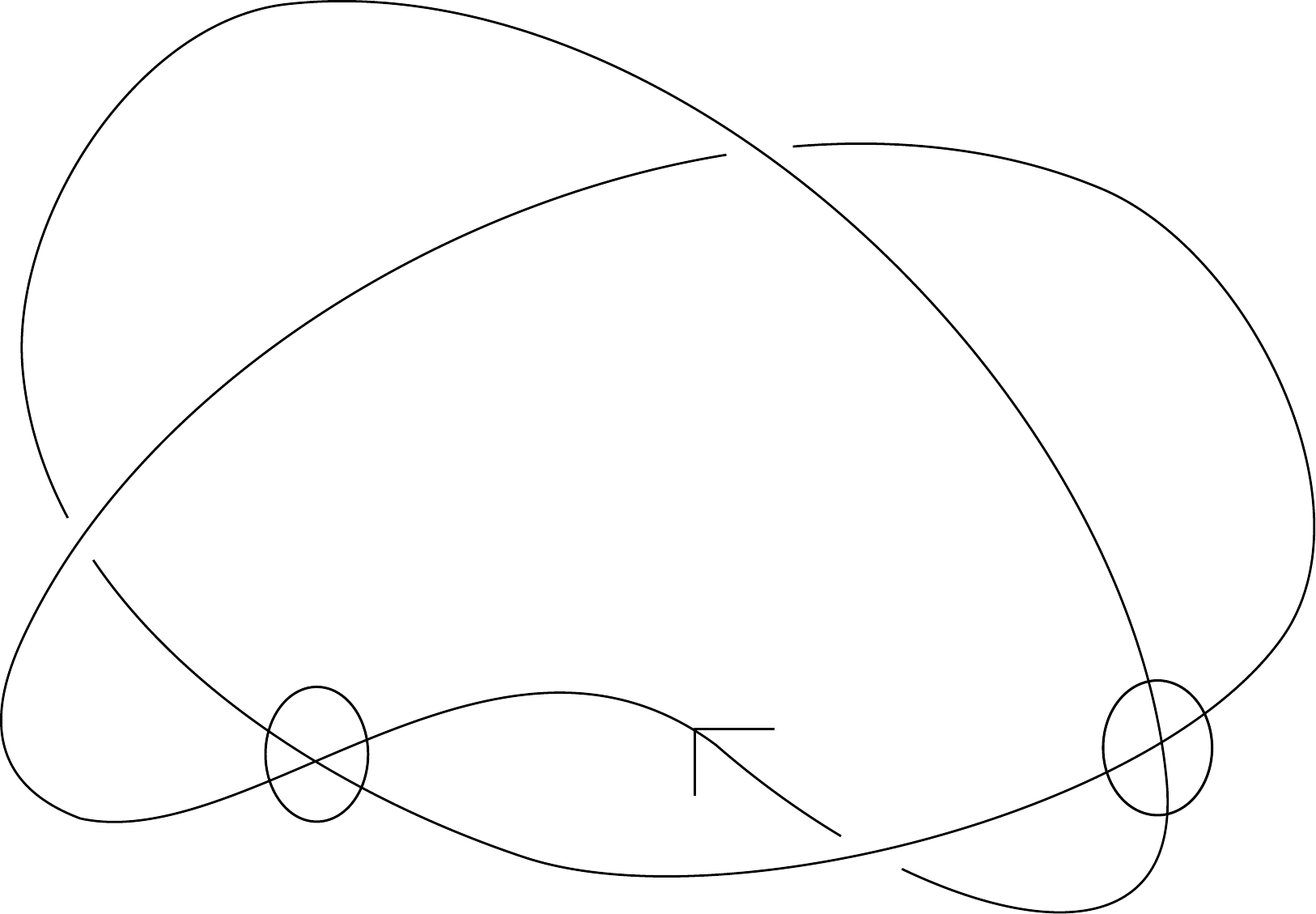}
     \end{tabular}
     \caption{\bf The virtual knot 3.1}
     \label{fig:virt3}
\end{center}
\end{figure}	
 \begin{figure}[H]
\begin{center}
     \begin{tabular}{c}
     \centering  \scalebox{0.20}{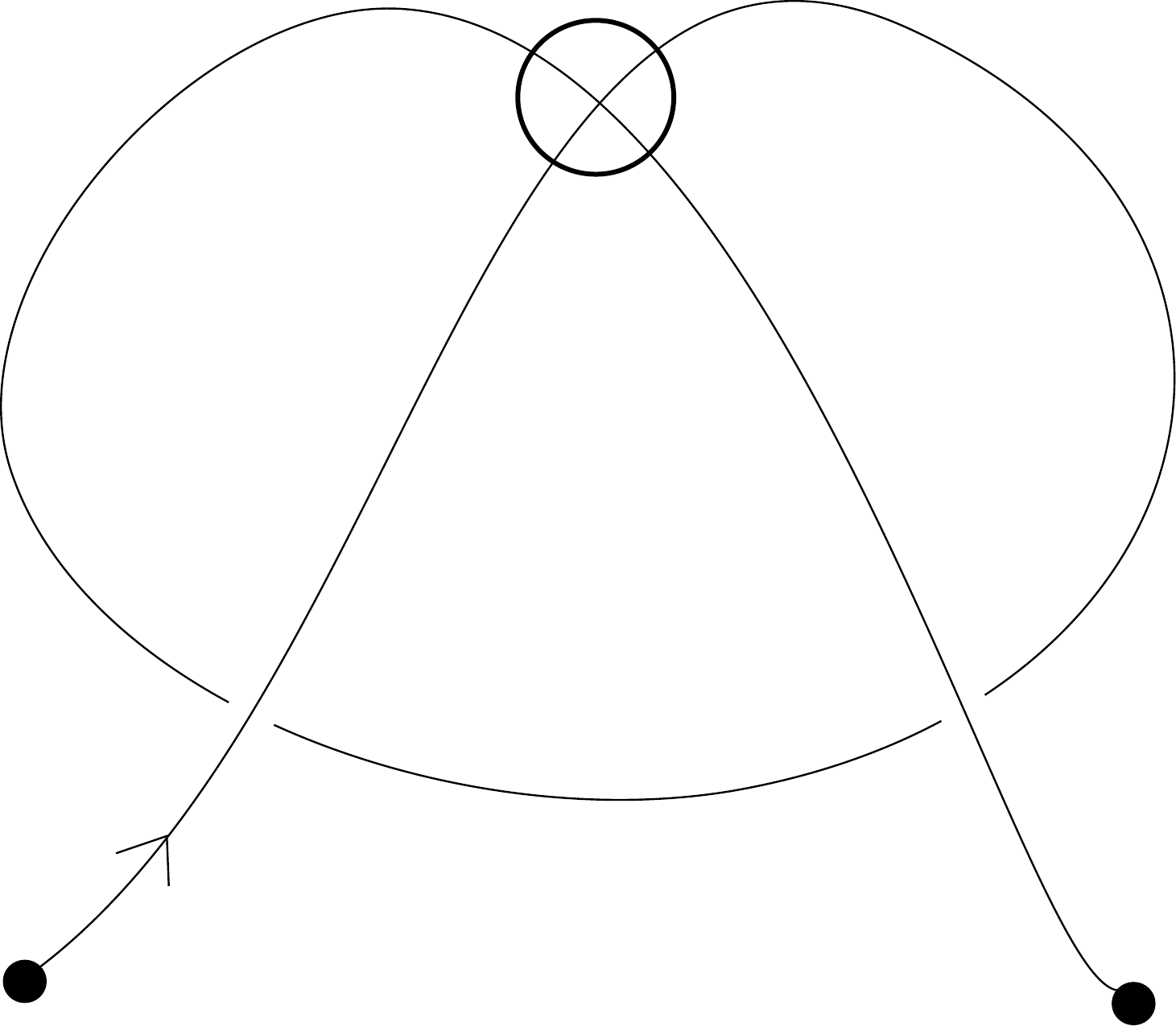}
     \end{tabular}
     \caption{\bf A nontrivial virtual knotoid with trivial virtual closure}
     \label{fig:trivialvc}
\end{center}
\end{figure}	
 \begin{figure}[H]
\begin{center}
     \begin{tabular}{c}
     \centering  \scalebox{0.55}{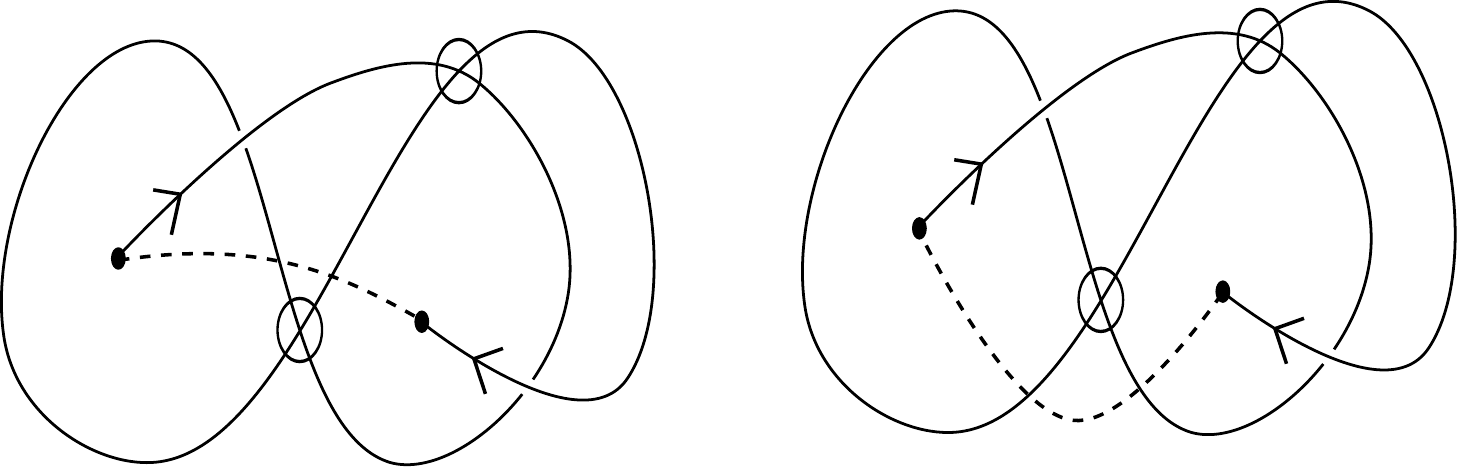}
     \end{tabular}
     \caption{\bf A virtual knotoid diagram representing different knots via $\omega_-$}
     \label{fig:different}
\end{center}
\end{figure}
	Being a well-defined map, the virtual closure map gives a way to construct many invariants for knotoids in $S^2$ which can be generalized to a virtual knotoid invariant. In fact, for any invariant of a virtual knot, denoted by $\Inv$, we can define an invariant on knotoids in $S^2$, denoted by $\Inv$ by the following formula,
\begin{center}                
$\Inv(K)=\Inv(\overline{v}(K))$,
\end{center} 
where $K$ is knotoid in $S^2$. This formula is our main motivation for constructing the invariants of this paper.
\section{Parity and Odd Writhe}
\subsection{Gauss Code}
 The \textit{Gauss code} of a knotoid diagram $K$ (classical,  virtual or flat) is a linear code that consists of a sequence of labels each of which is assigned to the classical (or flat for flat diagrams) crossings encountered during a trip along $K$ from its tail to the head. Since any crossing of $K$ is traversed twice, each label in the code appears twice. Thus, the length of the code is $2n$, where $n$ is the number of classical crossings (flat crossings for flat diagrams) of $K$. We keep the information of the passage through a crossing either as an overcrossing or an undercrossing by adding the symbols $O$ and $U$, respectively, to the code, and we keep the signs of the crossings by putting $+$ or $-$ next to the label accordingly to the sign of the crossing. The resulting code is referred as the\textit{ signed Gauss code} of $K$. Note that the symbols $O$ and $U$ and the signs of crossings are omitted in the Gauss codes of flat knotoid diagrams.

 Gauss codes have a diagrammatic representation as follows. Each label in the Gauss code is represented by $2n$ points placed upon a segment which is oriented from left to right. The points are labeled as the corresponding labels in the code. A signed and oriented chord connects each pair of the labeled points. The orientation of a chord heads from the overcrossing to the undercrossing. That is, during a travel along the knotoid diagram $K$ starting from the tail, if a crossing is first encountered as an overcrossing then the arrow of the corresponding chord heads towards the second appearance of the label. The sign of the chord is the sign of the associated crossing. For flat knotoid diagrams, we have the notion of right and left at each flat crossing. If a crossing is first encountered as going to the right then the head of the arrow on the corresponding chord heads towards the first appearance of the label. We call such a diagram with chords that represents the Gauss code of a knotoid diagram the \textit{chord diagram} of the knotoid diagram. Each knotoid diagram, including classical, virtual and flat knotoid diagrams, has a unique Gauss code and chord diagram. Figure \ref{fig:chordex} depicts the chord diagram of the knotoid diagram $K$ that is given in Figure \ref{fig:knotoid}(g). 
\begin{figure}[H]
%\centering  \scalebox{0.5}{\input{trl.pdf_tex}}
    \begin{center}
     \begin{tabular}{c}
     \centering  \scalebox{0.5}{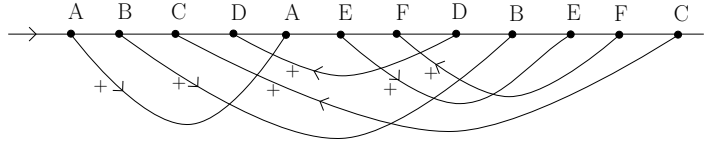}
     \end{tabular}
     \caption{\bf Chord diagram of $K$}
     \label{fig:chordex}
\end{center}
\end{figure}	
\begin{definition}
\normalfont
A single component Gauss code is said to be \textit{evenly intersticed} if there is an even number of labels between two appearances of any label.
\end{definition}
 Any classical knot diagram has evenly intersticed Gauss code \cite{RR}. The Gauss codes of classical knotoid diagrams are not necessarily evenly intersticed. For instance, the Gauss code of $K$ shown in Figure \ref{fig:chordex}, is $OA+OB+UC+UD+UA+OE+UF+OD+UB+UE+OF+OC+$, which is not an evenly-intersticed Gauss code. This fact gives rise to a well-defined parity for the crossings of classical knotoid diagrams taking values in $\mathbb{Z}_2$. A crossing of a classical, virtual or flat knotoid diagram is called \textit{odd} if there is an odd number of labels in between the two appearances of the crossing otherwise it is called an \textit{even} crossing. For the knotoid diagram $K$, the crossings $A, D, E$ and $F$ of the knotoid diagram $K$ are odd, and the crossings $B,C$ are even. Note that for the purpose of parity we may use the Gauss code of the underlying flat diagram of a knotoid diagram. %Each crossing is regarded as either \textit{even} or \textit{odd}.
\begin{thm}
The Gauss code of a knotoid diagram in $S^2$ is evenly intersticed if and only if it is a knot-type knotoid diagram.
\end{thm}
\begin{proof} 
\textit{The loop} at a (classical) crossing of a knotoid diagram in $S^2$ is defined to be the path obtained by traversing the knotoid diagram starting and ending at that crossing. There is a loop at each crossing of a knotoid diagram. Let $K$ be a proper knotoid diagram. Then one of the endpoints of $K$ is separated from the other endpoint by at least one loop at a crossing of $K$, that is, one of the endpoints is located inside at least one loop. All the strands entering the loop except the one that is adjacent to the endpoint, leave the loop by Jordan curve theorem. Thus each such strand contributes with a pair of labels to the Gauss code of the diagram. The Gauss code of $K$ along this loop is in the following pattern:  ...$c...~d~a~e~...\overline{a}...~\overline{e}~c$..., where $c$ represents the crossing that forms the loop containing the endpoint, $d$ represents the crossing of the strand adjacent to the endpoint with the loop, and $a, \overline{a}$ and $e, \overline{e}$ for the pairs of crossings created by the transversally intersecting strands which enter and leave the loop. Thus, between the two appearances of the label $c$, we have an old number of labels so that the Gauss code of $K$ is not evenly-intersticed. 
For a knot-type diagram $K$, we can assume that the tail and head lie in the outermost region of the diagram (where the $\infty$-point is located) so that none of the loops at crossings encloses them. Again by Jordan curve theorem, all the strands passing through any of the loops of $K$ enter and leave the loop so that they contribute with a pair of labels to the Gauss code of $K$. This shows that each crossing is even, that is, the Gauss code of $K$ is evenly-intersticed.
\end{proof}
\begin{lem}
The Gauss code of a knotoid diagram in $S^2$ is the same as the Gauss code of its virtual closure.
\end{lem}
\begin{proof}
The virtual closure map adds virtual crossings to a given knotoid diagram then it is clear that the map does not have any effect on the Gauss code of the diagram.
\end{proof}
\begin{definition}\normalfont
Using the parity of crossings, we define the \textit{odd writhe} for both classical and virtual knotoid diagrams as the sum of the signs of the odd crossings,
\begin{center}
\textit{Odd Writhe} of $K$ $=$ $J(K)$ $=$ $\sum_{c\in \Odd(K)}\signn(c)$,
\end{center}
where $K$ is a knotoid diagram and Odd(K) is the set of odd crossings in $K$.
\end{definition}
\begin{thm}
Odd writhe is a virtual and classical knotoid invariant.
\end{thm}
\begin{proof}
 The virtual moves that are induced by the detour move do not have an effect on the structure of Gauss code of a virtual knotoid diagram. As a result, the set of odd crossings of the virtual knotoid diagram remains the same under these moves. The odd writhe is invariant under the virtual moves.
It is left to verify the invariance under the $\Omega$-moves. The changes in Gauss codes under some of the classical moves are illustrated in Figure \ref{fig:odd}. In this figure, $A$, $B$, $C$ are the labels of the crossings inside the move patterns. and $\tau$, $\gamma$ and $\omega$ denote the words consisting of the crossing labels that are met by traversing the diagram outside the move patterns.  We observe the following. An $\Omega_1$-move adds/removes two consecutive labels in the Gauss code. The parity of the crossings outside the move region remains the same and being an even crossing, the added/removed crossing by an $\Omega_1$-move does not affect the odd writhe.

 An $\Omega_2$-move adds/removes either a pair of even crossings or a pair of odd crossings with opposite signs for any orientation type of the move. In the former case the even crossings do not have any effect on the odd writhe. In the latter case, the two odd crossings will be canceled out in the odd writhe summation for they have opposite signs. The parity of the crossings located outside the $\Omega_2$ move region, remains the same since the labels which are added/removed by one $\Omega_2$-move are located as consecutive pairs in the the Gauss code.

%Figure \ref{fig:odd} depicts corresponding Gauss codes of the diagrams related to each other by one $\Omega_3$- move. 
 The triangular move pattern of $\Omega_3$-move can contain either three even crossings or two odd crossings and one even crossing. In the former case, these even crossings are taken to even crossings by an $\Omega_3$-move and the parity of other crossings outside the move pattern remains the same thus the odd writhe is not affected. In the latter case, an $\Omega_3$- move permutes the order of the odd crossings in the $\Omega_3$-move region. The parity and the sign of the odd crossings remain the same. It is not hard to see that the parity of crossings outside the move pattern does not change. The arguments above hold for the other cases of the classical moves. Therefore, the odd writhe is a virtual knotoid invariant.
\end{proof}
\begin{figure}[H]
\begin{center}
  \begin{tabular}{c}
	\Huge{
	\centering \scalebox{0.35}{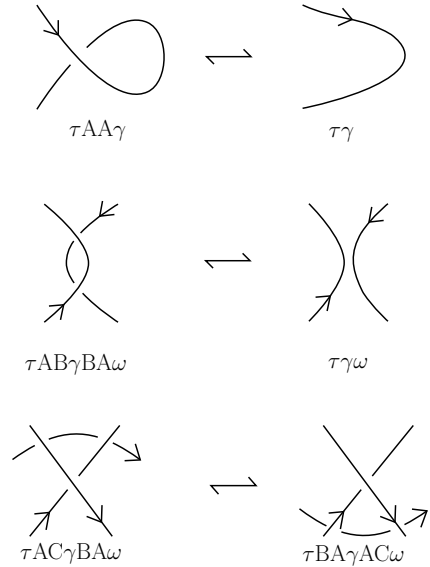}
	}
	\end{tabular}
	\caption{\bf The changes in Gauss codes under some $\Omega$- moves}
	\label{fig:odd}
	\end{center}
	\end{figure}
\begin{cor}\normalfont 
If a knotoid $K$ is a knot-type knotoid then the odd writhe of $K$ is zero.
\end{cor}
\begin{proof}
 It follows by Theorem 4.1 and Theorem 4.3.
\end{proof}
%Note that, as a result of being invariant under the $\Omega$-moves, the odd writhe is also a classical knotoid invariant. 
\begin{rem}
\normalfont
The crossings shared by any two components of a multi-knotoid diagram obstruct extending the parity to multi-knotoids. The Gauss code of the multi-knotoid diagram given in Figure \ref{fig:multi} is $O1- U2- O3- O4+ U1- O2- U3-/ U4+$. Crossings $1, 2$ and $3$ are even crossings in the circular component. These crossings become odd when the second component is considered. 
\begin{figure}[H]
\begin{center}
     \begin{tabular}{c}
     \centering  \scalebox{0.30}{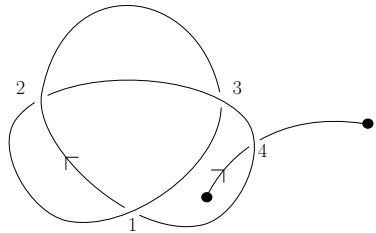}
     \end{tabular}
     \caption{\bf A multi-knotoid diagram}
     \label{fig:multi}
\end{center}
\end{figure}
 To define a well-defined parity for crossings of multi-knotoid diagrams, we apply the same method used in \cite{Ma, Kae, kk} to extend the parity to a parity of virtual links. The idea is to regard the crossings of a multi-knotoid diagram that are shared by two components as \textit{link crossings}. The parity remains the same for self-crossings of each component, as odd or even crossings. In particular for the diagram given in Figure \ref{fig:multi}, the crossings $1$, $2$, $3$ are even and the crossing $4$ is a link crossing. 
\end{rem} 
\subsection{Parity Bracket Polynomial}
 The parity bracket polynomial of V.~Manturov \cite{Ma} is a modification of the bracket polynomial that uses the parity of crossings in virtual knots and links. With the existence of even and odd crossings in knotoid diagrams, the parity bracket polynomial can be defined for both classical and virtual knotoid diagrams. For a knotoid diagram $K$, either classical or virtual, a \textit{parity state} is defined to be a labeled graph (a virtual graph for virtual diagrams). A parity state of a virtual knotoid diagram $K$ is obtained by smoothing the even crossings of $K$ by A- and B- smoothing type of the usual bracket polynomial and labeling the smoothing sites by $A$ or $A^{-1}$, respectively, and replacing the odd crossings of $K$ by graphical nodes. Note that circular and long segment components of parity states are regarded as graphs.
\begin{figure}[H]
%\centering  \scalebox{0.5}{\input{trl.pdf_tex}}
    \begin{center}
     \begin{tabular}{c}
     \centering  \scalebox{0.75}{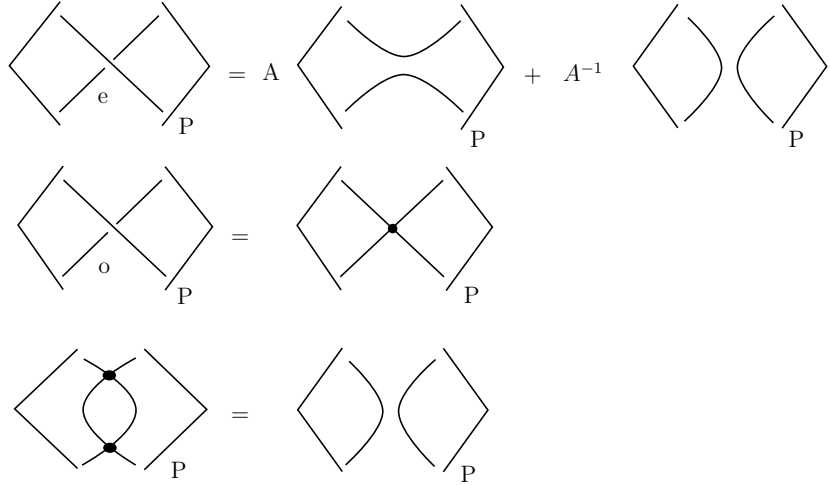}
     \end{tabular}
     \caption{\bf Parity bracket expansion}
     \label{fig:skein}
\end{center}
\end{figure}	
 The resulting states are taken up to the virtual equivalence (isotopy of $S^2$ and detour moves) and up to the \textit{reduction rule}, shown in Figure \ref{fig:skein}. The reduction rule is simply a Reidemeister two- move 
that eliminates two graphical nodes forming the vertices of a bigon. The state components that still contain nodes after applying the reduction rule, are called \textit{irreducible state components}. Each irreducible state component contributes to the polynomial as a graphical coefficient.  
\begin{definition}
The \textit{parity bracket polynomial} of a virtual or classical knotoid diagram $K$, is defined as
 \begin{center}
$<K>_P$=$\sum_S A^{n(S)}(-A^2-A^{-2})^{l(S)}G(S)$,
\end{center}
where $n(S)$ denotes the number of $A$-smoothings minus the number of $B$-smoothings, $l(S)$ is the number of components without any nodes of the parity state $S$ and $G(S)$ is the union of irreducible state components.
\end{definition}
\begin{definition}\normalfont
The \textit{normalized parity bracket polynomial} of a knotoid diagram (classical or virtual) $K$ is defined as
\begin{center}
$P_K=(-A^3)^{-\writhe(K)}<K>_P$.
\end{center}
\end{definition}
\begin{thm}
The normalized parity bracket polynomial is a virtual and a classical knotoid invariant.
\end{thm}
\begin{proof}
It can be verified by the reader that one $\Omega_1$-move adds/removes an even crossing and changes the polynomial by $-A^{\pm 3}$. Then the writhe normalization makes the parity polynomial invariant under an $\Omega_1$- move. An $\Omega_2$-move may add two crossings that are both even crossings. In this case the parity bracket polynomial is invariant under this move since the bracket polynomial is invariant under $\Omega_2$-move. If the crossings in the move pattern are both odd crossings, the reduction rule applies and eliminates the crossings. So, the polynomial does not change by an $\Omega_2$- move. It is clear that the invariance under $\Omega_3$- move, if three of the crossings in the triangular region are even, follows from the bracket polynomial invariance, and if two of them are odd and one is an even crossing then the invariance follows by an isotopy of the state component. 
\end{proof}
 The normalized parity bracket polynomial can be defined for flat virtual knotoid diagrams. Let $K$ be a flat virtual knotoid diagram, all even crossings of $K$ are smoothed out in two ways, and the smoothing sites are labeled by $A=-1$. Odd crossings of $K$ are replaced by graphical nodes and we obtain the parity states of $K$. The reduction rule applies the same on the states of $K$ for the elimination of nodes. The \textit{parity bracket polynomial} of $K$ is defined as,
\begin{center}
$<K>_P$=$\sum_S -2^{l(S)}G(S),$
\end{center}
where $l(S)$ is the number of components without any nodes of the parity state $S$ and $G(S)$ is the union of irreducible state components.
\begin{thm}
The parity bracket polynomial is an invariant of flat virtual knotoids if it is normalized by the writhe. The normalized parity bracket polynomial of a flat knotoid in $S^2$ is trivial.
\end{thm}
\begin{proof}
The invariance of the normalized parity bracket polynomial can be seen easily by checking of the invariance of the polynomial under the flat $\Omega$- moves. Since any flat classical knotoid is f-equivalent to the trivial knotoid, the second statement follows.
\end{proof}
\begin{prop}
There is no irreducible state component in the parity state expansion of a knotoid diagram in $S^2$.
\end{prop}
\begin{proof}
Let $K$ be a knotoid diagram in $S^2$. If $K$ is a knot-type knotoid diagram then none of its crossings is odd, as a result of Jordan curve theorem. Therefore, there is no graphical coefficient contributing to the parity bracket polynomial of $K$. In fact, the parity bracket polynomial of $K$ coincides with its usual bracket polynomial. On the other hand, proper knotoid diagrams have odd crossings so we observe state components with nodes in their parity state expansion. These state components include only the nodes corresponding to the odd crossings of $K$. It is clear that the set of odd crossings of $K$ is the same with the set of odd crossings of $F(K)$ where $F(K)$ is the underlying flat knotoid diagram of $K$. For this reason, the existence of any irreducible state component in parity states of $K$ would cause an irreducible state component in the parity states of $F(K)$. Thus it is sufficient to show that any graphical state component of a flat knotoid diagram can be reduced to a component that is free of nodes. Proposition 3.6 implies that a flat knotoid diagram in $S^2$ can be obtained from the trivial knotoid diagram by finitely many flat $\Omega$-moves. We induct on the flat knotoid diagrams.

 The trivial knotoid diagram has no crossing so it has no irreducible state component. We assume that graphical state components of all flat diagrams that are obtained from the trivial diagram by an application of $n$ flat $\Omega_i$- moves, can be reduced to a component without any nodes. Let $K$ be such a flat knotoid diagram. A single flat $\Omega_2$-move adds/removes either two odd or even crossings to $K$. Two crossings added/removed do not change the parity of the crossings outside the flat $\Omega_2$-move pattern, as explained in the proof of Theorem 4.4. If the crossings are even crosings then they increase/decrease the number of state components but there is no resulting graphical component. If the crossings (added) are odd crossings, they are located as the vertices of a bigon (they are paired up) that can be eliminated by the reduction rule. Two paired up odd crossings are removed by the move, and since the rest of the odd crossings are assumed to be paired up so that they can be eliminated, there is no irreducible state component created. Thus, a flat classical knotoid diagram which is obtained by applying one $\Omega_2$- move to $K$, does not have any irreducible graphical state components in its parity states.

 The flat $\Omega_3$- move does not add/remove any odd or even crossings or change the parity of the crossings outside the move pattern. Thus, the parity states of a flat classical knotoid diagram obtained by applying one flat $\Omega_3$-move to $K$ are isotopic to the parity states of $K$.

 The flat $\Omega_1$-move adds/removes an even crossing which does not change the parity of the crossings outised the move pattern and does not add any nodes to the diagram, so there are no resulting graphical state components. This completes the induction argument. The parity state components of a flat classical knotoid diagram are reduced to state components that are free of nodes. Therefore, the parity state components of a knotoid diagram in $S^2$ are reducible.
\end{proof}
\begin{lem}
For a knotoid in $S^2$, $K$, we have $<K>_P=<\overline{v}(K)>_P$.
\end{lem}
\begin{proof}
Let $\tilde{K}$ be a classical knotoid diagram representing $K$. %The contribution of each state component of $\tilde{K}$ to the polynomial is the same as the contribution of the components which are the virtual closures of the long segment component of each state.
The classical crossings of $\overline{v}(\tilde{K})$ are the same with the classical crossings of $\tilde{K}$. The skein relations or the reduction rule of the parity bracket polynomial are not applied to the virtual crossings of a virtual knot diagram, so to the virtual crossings of $\overline{v}(\tilde{K})$. By Proposition 4.6, all nodes in a parity state of $\tilde{K}$ are eliminated so that each parity state of $\tilde{K}$ is reduced to consist of disjoint simple closed curves and a long segment component. If any of the virtual crossings of $\overline{v}(\tilde{K})$ passes through a bigon whose vertices are reducible graphical nodes of a parity state of $\tilde{K}$ then we can move these virtual crossings out of the bigon by the detour move that is available for the parity states. After moving the virtual crossings out of the bigon, we can eliminate the nodes by the reduction rule as we do for the parity state of $\tilde{K}$. Therefore, any parity state components of $\overline{v}(\tilde{K})$ can be obtained by connecting the endpoints of the long segment component of the corresponding parity state component of $\tilde{K}$ in the virtual fashion. This shows that $<K>_P=<\overline{v}(K)>_P$.
\end{proof}
\begin{cor}
If there are graphical coefficients in the parity bracket polynomial of a virtual knot $K$ then $K$ is not the virtual closure of a knotoid in $S^2$.
\end{cor}
\begin{proof}
It follows by Proposition 4.6 and Lemma 4.7.
\end{proof}
In a virtual knotoid diagram, the parity states are not necessarily reducible. We give a combinatorial explanation for the reducibility of a parity state. We label each edge of a given state of which we illustrate a small portion, in Figure \ref{fig:par}. The nodes that share exactly two edges (labeled as $b$ and $e$ in the figure) form a reducible bigon if and only if the edges appear in the order $e~b~f~c$ during a full tour in the counterclockwise direction around one of the nodes and in the order $b~e~a~d$ around the other node. More precisely, the shared edges $b$ and $e$ appear in cyclic order around the nodes. Note that the detour moves do not change the labels. Then up to detour moves, a parity state of a virtual knotoid diagram will have a removable bigon between the nodes.
\begin{figure}[H]
%\centering  \scalebox{0.5}{\input{trl.pdf_tex}}
    \begin{center}
     \begin{tabular}{c}
     \centering  \scalebox{0.4}{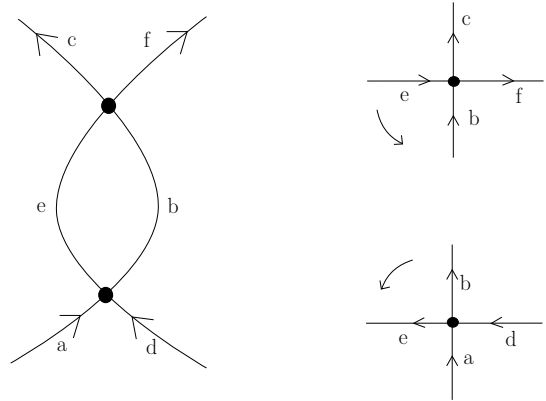}
     \end{tabular}
     \caption{\bf Labels at the nodes of a reducible bigon}
     \label{fig:par}
\end{center}
\end{figure}	
 \begin{example}\normalfont
Both of the classical crossings of the virtual knotoid diagram given in Figure \ref{fig:trivialvc} are odd crossings. There is only one parity state of this diagram that is a graphical state, obtained by replacing these crossings by nodes. As seen in Figure \ref{fig:exlabel}, the two nodes have orders: $a~c~b~d$ and $e~b~d~c$, respectively. Since the shared edges $b$ and $d$ do not appear in the required order, the state is not reducible. We conclude that the parity bracket polynomial consists of one summand that is a graphical coefficient. Thus, the non-triviality of this virtual knotoid whose virtual closure is trivial, is verified by the parity bracket polynomial. Moreover, by Proposition 4.6, this virtual knotoid diagram is not virtually equivalent to a classical knotoid diagram, and in fact, it represents a genus one virtual knotoid.
\end{example}
\begin{figure}[H]
 \begin{center}
     \begin{tabular}{c}
		\centering  \scalebox{0.50}{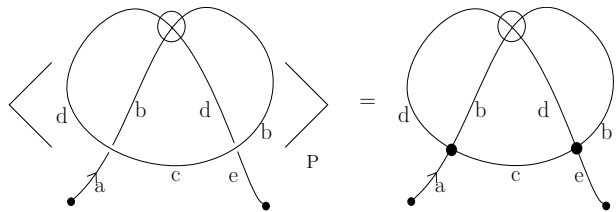}
		\end{tabular}
     \caption{\bf $1^{\text{st}}$ node: $acbd$ $2^{\text{nd}}$ node: $ebdc$ }
     \label{fig:exlabel}
\end{center}
\end{figure}	
 The condition given above that is necessary for the elimination of the nodes of a graphical state component applies in the same way to the flat case. The parity bracket polynomial of the underlying flat diagram of the knotoid diagram given in Figure \ref{fig:trivialvc} is the same as the polynomial of the overlying virtual knotoid, consisting of one graphical coefficient. Therefore, the parity bracket polynomial of this flat virtual knotoid diagram is not trivial. This completes the argument in Section 3.2 that flat virtual knotoids are not necessarily trivial.
\begin{rem}\normalfont
The normalized parity bracket polynomial extends to an invariant for virtual multi-knotoids. Even crossings of a multi-knotoid diagram are smoothed in the usual way. Together with odd crossings, also link crossings (crossings between distinct components) of a multi-knotoid diagram are replaced by graphical nodes. We extend the procedure for calculation of the parity bracket polynomial to include the link crossings as follows. The graphical state components containing the nodes corresponding to link crossings are eliminated by the same reduction rule. Irreducible graphical state components contribute to the polynomial as graphical coefficients. The parity bracket polynomial of a virtual multi-knotoid is defined in the same way by expanding the state summation, and the normalization of the polynomial with writhe is a virtual multi-knotoid invariant.

 We have showed that the graphical components of a classical knotoid diagram are all reduced by the reduction rule, in other words, they are free of nodes. The parity states of a classical multi-knotoid diagram may contain irreducible graphical states. Figure \ref{fig:multii} depicts a multi-knotoid diagram with two components and one link crossing. It is clear that the diagram has a nontrivial parity bracket polynomial that is equal to one graphical coefficient. Thus,  the multi-knotoid represented by this diagram is a nontrivial multi-knotoid. 
\begin{figure}[H]
  \begin{center}
     \begin{tabular}{c}
     \centering  \scalebox{0.50}{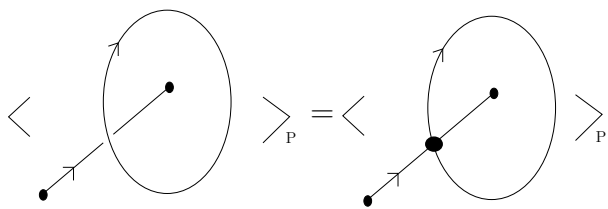}
     \end{tabular}
     \caption{\bf A multi-knotoid with nontrivial parity bracket}
     \label{fig:multii}
\end{center}
\end{figure}	
\end{rem}
 \subsection{Affine Index Polynomial}
 The affine index polynomial was defined for virtual knots and links by L.H. Kauffman \cite{Ka6}. %It is a one variable Laurent polynomial taking values in the ring $\mathbb{Z}[t,t^{-1}]$ whose 
The affine index polynomial of knotoids, either classical or virtual, is based on an integer labeling assigned to flat knotoid diagrams in the following way. A flat knotoid diagram, classical or virtual, is associated with a graph (virtual graph in the case of flat virtual diagrams) where the flat classical crossings and the endpoints are regarded as the vertices of the graph. An \textit{arc} of an oriented flat knotoid diagram is an edge of the graph it represents, that extends from one vertex to the next vertex. Note that tail and the head of the diagram are considered to be vertices of the graph). Given a knotoid diagram $K$, the labeling of each arc of the underlying flat knotoid diagram of $K$, $F(K)$, begins with the first arc which connects the tail and the first flat crossing. The integer labeling rule at a flat crossing is illustrated in Figure \ref{fig:lab}. At each flat crossing, the labels of the arcs change by one; if the incoming arc labeled by $a$, $a\in \mathbb{Z}$ crosses the crossing towards left then the next arc is labeled by $a+1$, if the incoming arc crosses the crossing towards right then it is labeled by $a-1$. There is no change of labels at virtual crossings. Note that the numbers at $c$, $w_+(c)$ and $w_-(c)$ are defined as differences of labels so that the weights are well-defined. Since the weights are well-defined up to this integer labeling, it is convenient to label the first arc with $0$.  
\begin{figure}[H]
\begin{center}
     \begin{tabular}{c}
     \centering \scalebox{0.40}{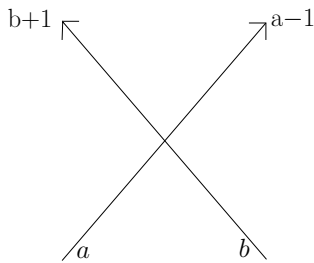}
     \end{tabular}
     \caption{\bf Integer labeling at a flat crossing}
     \label{fig:lab}
\end{center}
\end{figure}
%\begin{definition}\normalfont
 Let $c$ be a classical crossing of $K$. We define two numbers at $c$ resulting by the labeling of $F(K)$. These numbers that are denoted by $w_+(c)$ and $w_-(c)$, are defined as follows.
\begin{align*}
     w_+(c)&=a-(b+1)\\
     w_-(c)&=b-(a-1),
\end{align*}	
where $a$ and $b$ are the labels for the left and the right incoming arcs at the corresponding flat crossing to $c$, respectively. the numbers $w_+(c)$ and $w_-(c)$ are called \textit{positive} and \textit{negative} weights of $c$, respectively. \\
The\textit{ weight} of $c$ is defined as
\[ w_K(c)=
 \begin{cases}
 w_+(c),&  \text{if the sign of $c$ is a positive crossing} \\
 w_-(c),&  \text{if the sign of $c$ is a negative crossing}.\
 \end{cases}
\]
%\end{definition}
\begin{definition}\normalfont
 The \textit{affine index polynomial} of a virtual or classical knotoid diagram $K$ is defined by the following equation.
\begin{center}
                        $P_K(t)=\sum_c{\sgn(c)(t^{w_K(c)}-1)}$,
\end{center}												
where the sum is taken over all classical crossings of a diagram of $K$ and $\sgn(c)$ is the sign of $c$.
\end{definition}

  The underlying flat diagram of the virtual closure of a knotoid diagram is labeled as the same as the knotoid diagram since virtual crossings do not add any new arcs or labels. In fact, we have $P_K(t)= P_{\overline{v}(K)}(t)$, where $K$ is a knotoid diagram in $S^2$ and $\overline{v}(K)$ is the virtual closure of $K$.% The virtual closure map connects the first and the last arc of a knotoid diagram so these arcs are labeled with the same integer in the closure. As a result, the last arc of a knotoid diagram is always labeled by $0$.
\begin{thm}
The affine index polynomial is a virtual and classical knotoid invariant.
\end{thm}
\begin{proof} 
The polynomial $P_K$, by its definition, is independent of the moves generated by the detour move. It is left to check the invariance under oriented $\Omega$- moves. Note that for the verification of invariance of oriented virtual knot invariants, it is sufficient to check the oriented Reidemeister moves, given in Figure \ref{fig:3}, two types of the first move, one type of the second move and one type of the third move where there is a cyclic triangle in the middle and two of the crossings have the same sign and the third crossing has the opposite sign \cite{Po}. The reader can verify easily that this argument applies directly to virtual knotoid invariants. The integer labeling is uniquely inherited under these moves. The local changes (inside the disks where the move pattern lies) in labels is shown in Figure \ref{fig:3}. It can be seen in the figure that the $\Omega_1$-move adds a crossing with zero weight. The $\Omega_2$-move adds/removes two crossings with opposite signs but with same weights. The $\Omega_3$-move does not change weights or signs of the three crossings in the move pattern. Therefore, the affine index polynomial remains unchanged under these moves. 
\end{proof}
\begin{figure}[H]
  \begin{center}
     \begin{tabular}{c}
     \centering  \scalebox{0.51}{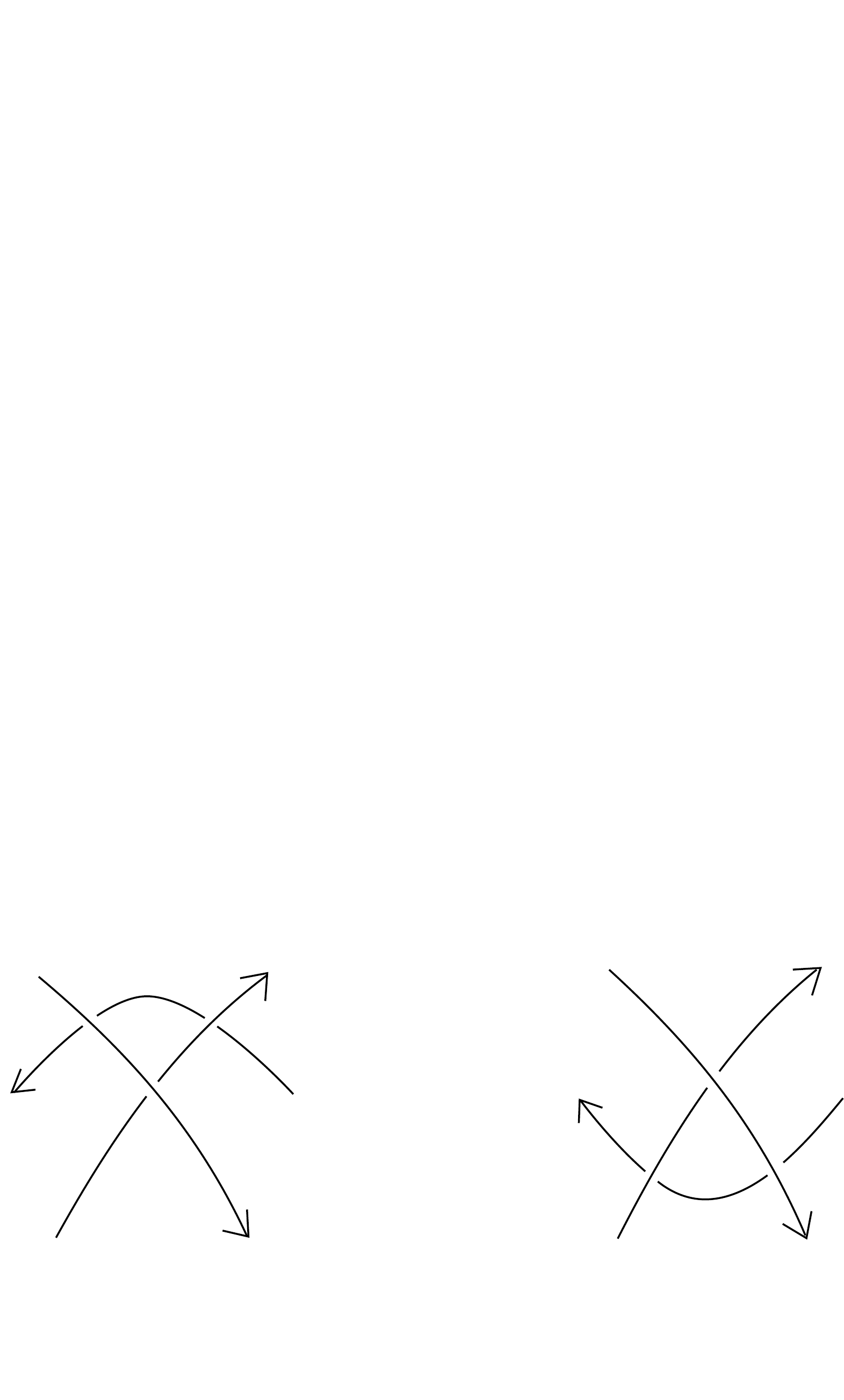}
     \end{tabular}
		\vspace{0.6cm}
     \caption{\bf The invariance of the affine index polynomial under the oriented moves}
     \label{fig:3}
\end{center}
\end{figure}
%\begin{rem}\normalfont
%The affine index polynomial is generalized to an invariant of virtual links by A.Cheng and H.Gao \cite{CG}. We will study generalizations of these invariants to classical and virtual multi-knotoids in a sequel to the present paper.
%\end{rem}
\subsection{A comparison of affine index polynomials: Knotoids vs Knots}
 The affine index polynomial of a knotoid in $S^2$ is the same as the affine index polynomial of its virtual closure that is a virtual knot as we noted before. We show in the following that the affine index polynomial has different properties for classical knotoids than the polynomial has for virtual knots.
\begin{enumerate}
\item
Knotoids in $S^2$ may have nontrivial affine index polynomial although the polynomial is trivial for all classical knots. All the crossings of a knot-type knotoid diagram are even by Theorem 4.1, and so the weights of the crossings are zero. A knot-type knotoid can always be represented by a knot-type knotoid diagram, therefore knot-type knotoids have trivial affine index polynomial. On the other hand, proper knotoids have odd crossings with nonzero weights. A proper knotoid may have nontrivial affine index polynomial. This difference is used to determine whether a knotoid is proper or knot-type knotoid: If a given classical knotoid diagram has nonzero affine index polynomial, then we conclude that this knotoid diagram represents a proper knotoid.\\
\item
The \textit{inverse} of an oriented virtual knot diagram is obtained by reversing the orientation of the diagram. For the affine index polynomial of a virtual knot $k$, we have
\begin{center}
          $P_K(t)=P_{\overline{K}}(t^{-1})$,
\end{center}	
where $K$ is an oriented diagram of $k$ and $\overline{K}$ is the inverse of $K$ \cite{Ka6}. Thus the affine index polynomial can be used to distinguish a virtual knot from its inverse.

 The \textit{inverse} of a knotoid diagram (classical or virtual) $K$ is defined by the reversing the orientation of the diagram, and denoted by $\overline{K}$. The tail of $K$ becomes the head of the inverse of $K$. The affine index polynomial fails to distinguish a knotoid diagram from its inverse as we explain in the following.
\end{enumerate} 
\begin{definition}\normalfont
The weights of crossings of a knotoid diagram $K$ are said to be \textit{symmetric} if for any classical crossing of $K$, $c_1$ with a nonzero positive weight $w_+(c_1)$, there is another classical crossing $c_2$ with a nonzero positive weight $w_+(c_2)$ such that $w_+(c_2)$=$-w_+(c_1)$. Such two crossings with opposite positive weights are said to be \textit{paired} crossings.
\end{definition}
\begin{lem}
\label{symweight}
The weights of the crossings of a flat knotoid diagram in $S^2$ are symmetric.
\end{lem}
\begin{proof}
Proposition 3.6 implies that any flat knotoid diagram in $S^2$ can be obtained from the trivial knotoid diagram by a finite sequence of the flat $\Omega$- moves, and also by isotopy of $S^2$. Using this fact, we proceed by induction on flat knotoid diagrams in $S^2$. The trivial diagram has no crossings so conventionally it satisfies the lemma. The flat diagrams shown in Figure \ref{fig:ind}, are obtained by applying one or two $\Omega_i$-moves to the trivial knotoid diagram. The weights of crossings of these diagrams are symmetric, as can be seen in the figure. Let us assume that the weights of crossings of any flat classical knotoid diagrams that are obtained by applying $n>0$ flat $\Omega_i$- moves to the trivial knotoid diagram are symmetric. Let $K$ be such a flat knotoid diagram and $K_1, K_2, K_3$ be flat knotoid diagrams that are obtained by applying one flat $\Omega_1$, $\Omega_2$ and $\Omega_3$-move to $K$, respectively. A flat $\Omega_2$ move adds/removes two crossings to $K$. Since the weights of the other crossings outside the move region are not affected, the symmetry of weights of $K$ is not destroyed. If the crossings located in the move region are even then they both have zero weights. If the crossings are odd then they are paired crossings. Thus, the weights of crossings of $K_2$ are symmetric.
 
 A flat $\Omega_3$- move does not change the weights of the three crossings, $A,B,C$ that are located in the triangular region of the move or the weights of the crossings outside the move region. If $A,B,C$ are even crossings then they have zero weight and they are taken to crossings with zero weight by a flat $\Omega_3$-move. If two of these crossings are odd and one of them is even, it is assumed that the odd crossings are paired with some other crossings of $K$ (either two of them with each other or with other crossings in the rest of the diagram). Thus the weights of $K_3$ are symmetric. A flat $\Omega_1$- move adds/removes one crossing with zero weight to the given diagram $K$ and does not change the weights of the remaining crossings. Therefore, the weights of the crossings of $K_1$ are symmetric. This completes the induction and proves that the weights of crossings of any flat knotoid diagram in $S^2$ are symmetric.
\end{proof}
\vspace{-0.64cm}
\begin{figure}[H]
\begin{tabular}{c}
\centering  \scalebox{0.57}{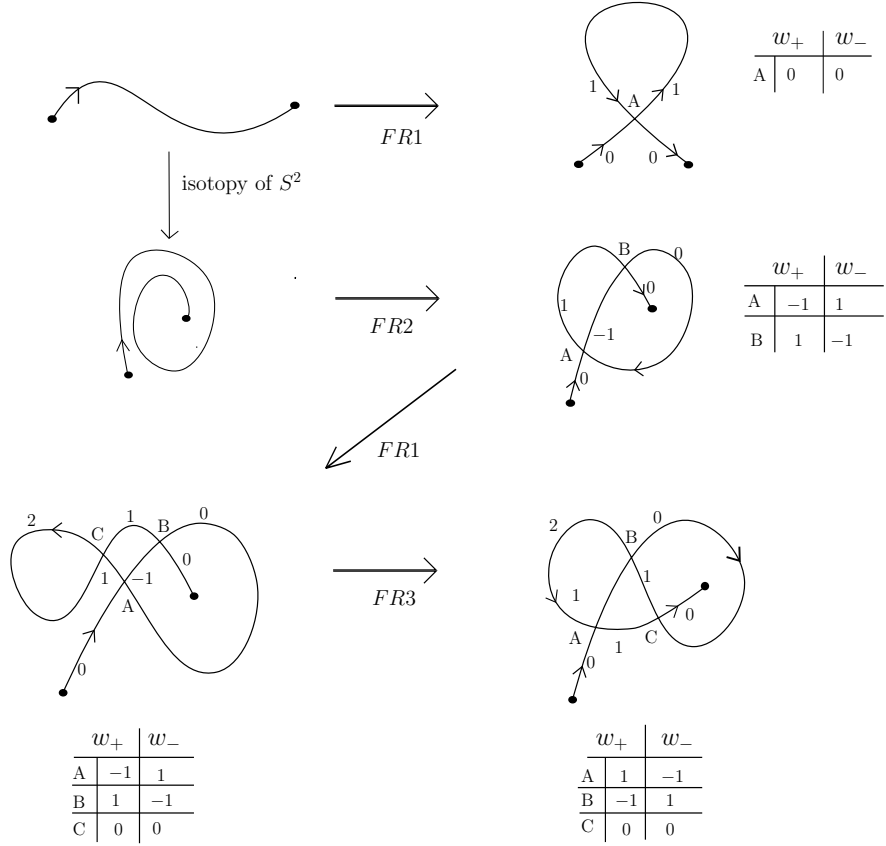}
		\end{tabular}
			 \caption{\bf Induction step}
         \label{fig:ind}
\end{figure}
\begin{thm}
\label{sym}
The affine index polynomial of a knotoid $K$ in $S^2$ is symmetric with respect to $t\leftrightarrow t^{-1}$. Therefore, $P_K(t)=P_{\overline{K}}(t)$, where $\overline{K}$ denotes the inverse of $K$. 
\end{thm}
\begin{proof}
Lemma \ref{symweight} shows that any crossing of a knotoid diagram in $S^2$ with a nonzero positive weight, is paired with another crossing. If the signs of paired crossings are different then the contributions of these crossings to the polynomial are canceled out. Let $c_1$ and $c_2$ be two paired crossings with the same sign, then they contribute to the polynomial either as the summands $(t^n-1)$ and $t^{-n}-1$ or $-t^n+1$ and $-t^{-n}+1$, respectively, where $n$ is the weight of $c_1$ and $-n$ is the weight of $c_2$. Since the affine index polynomial is a classical knotoid invariant, the symmetry of the affine index polynomial follows. It can be verified easily by the reader that reversing the orientation of $K$ only permutes the set of crossings and the weight chart of $K$. The affine index polynomial remains the same by reversing the orientation of $K$. Therefore we have $P_K(t)=P_K(t^{-1})=P_{\overline{K}}(t^{-1})=P_{\overline{K}}(t)$.
% and }$sgn(c_1)=sgn(c_2)$. 
\end{proof}
 There are virtual knotoids which do not satisfy Theorem 4.10. For instance, any virtual knotoid diagram whose underlying flat diagram is shown in Figure \ref{fig:counterex}, has non-symmetric affine index polynomial, by its weight chart. Consequently, none of these virtual knotoid diagrams is virtually equivalent to a classical knotoid diagram. 
\begin{figure}[H]
\begin{center}
\begin{tabular}{c}
\centering  \scalebox{0.55}{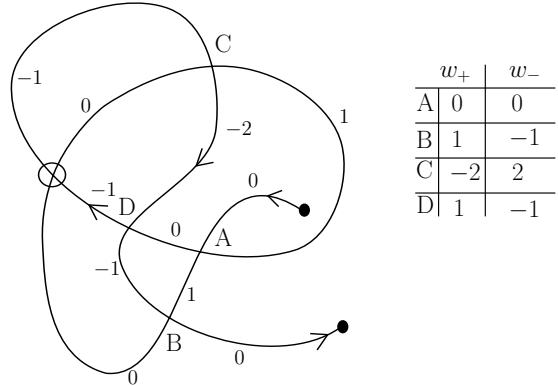}
\end{tabular}
\caption{\bf A flat virtual knotoid with non-symmetric weights}
\label{fig:counterex}
\end{center}
\end{figure}
 \begin{thm}
If the affine index polynomial of a virtual knot is not symmetric with respect to $t\leftrightarrow t^{-1}$ then it is not the virtual closure of a knotoid in $S^2$. 
\end{thm}
\begin{proof}
The affine index polynomial remains unchanged by the virtual closure map since the virtual crossings added via the map, do not change the weights of any of the (classical) crossings and have no contribution to the polynomial.
Thus, we have $P_K(t)= P_{\overline{v}(K)}(t)$, where $K$ is a knotoid in $S^2$. The statement follows by this equality and by Theorem 4.10.
\end{proof}
\subsection{The height of a knotoid and the affine index polynomial}
 The \textit{height} (or the \textit{complexity} with respect to Turaev's terminology in \cite{Tu}) of a knotoid diagram in $S^2$ is the minimum number of crossings that a shortcut creates during the underpass closure. The \textit{height of a knotoid} in $S^2$, $K$ is defined as the minimum of the heights, taken over all equivalent classical knotoid diagrams to $K$ and is denoted by $h(K)$. The height is an invariant of  knotoids in $S^2$ \cite{Tu}. A knotoid in $S^2$ is of knot-type if and only if its height is zero or equivalently a knotoid in $S^2$ has nonzero height if and only it is a proper knotoid \cite{Tu}.

 It is often hard to compute the height with an attempt of direct computation, for we should take into account all the equivalent knotoid diagrams. The affine index polynomial provides the following estimation for the height.
\begin{thm}
Let $K$ be a knotoid in $S^2$.% Let $m$ be the maximal degree of the affine index polynomial of $K$. 
The height of $K$ is greater than or equal to the maximum degree of the affine index polynomial of $K$.
\end{thm}
\begin{proof}
Let $\tilde{K}$ be a knotoid diagram representing $K$. We label the underlying flat knotoid diagram of $\tilde{K}$ with respect to the labeling rule given in Figure \ref{fig:lab}. The \textit{algebraic intersection number} of a loop at a crossing $C$, $l(C)$ (see Section 4.1) with a strand of $\tilde{K}$ is defined to be the total number of times that the strand intersects the loop from left to right minus the total number of times that the strand intersects the loop from right to left. Figure \ref{fig:chee} illustrates two possible types of loops at the crossing $C$ one of which is oriented in the counterclockwise, and the other in the clockwise direction. The algebraic intersection numbers of the loop at $C$ with the piece of strand shown in the figure, are $-1$ and $+1$, respectively. In both pictures, the incoming arcs towards the crossing $C$ are labeled by some integer $a$. Assuming that the strand shown is the only one intersecting the loops, it can be verified that $-1$ is equal to the negative weight of the crossing $C$ of the first loop and $+1$ is equal to the positive weight of the crossing $C$ of the second loop. Then the following generalization is clear. If the sum of the algebraic intersection numbers of the loop $l(C)$ at the crossing $C$ with intersecting strands is equal to $n$ then $n$ is equal to either $w_-(C)$ or $w_+(C)$, depending on the orientation of the loop $l(C)$.   
\begin{figure}[H]
 \begin{center}
     \begin{tabular}{c}
		\Huge{
     \centering \scalebox{0.45}{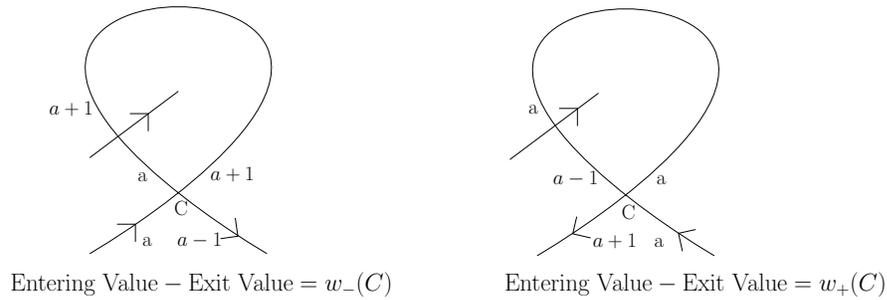}
		}
     \end{tabular}
     \caption{\bf The weights with respect to the orientation of the loop at $C$}
     \label{fig:chee}
\end{center}
\end{figure}
Let $m$ be the maximum degree of the affine index polynomial of $K$. Then there exists a crossing of $\tilde{K}$ with weight $m$. In fact, $m$ is the maximal weight among the weights of crossings of $\tilde{K}$. Let $\tilde{C}$ be one of the crossings of $\tilde{K}$ with weight $m$ and $l(\tilde{C})$ be the loop at $\tilde{C}$.% (the path within $\tilde{K}$ that starts and ends at $C$).

 Figure \ref{fig:seifertsmooth} shows the way to smooth a classical crossing of a knotoid diagram according to the orientation. Each crossing which are met twice while traversing along the loop $l(\tilde{C})$, are all smoothed accordingly to the orientation. This implies that each self-intersection of the loop $l(\tilde{C})$ is smoothed. Smoothing the self-intersections of the loop $l(\tilde{C})$ results in oriented embedded circles (in $S^2$) and a long oriented segment containing the tail and the head of $\tilde{K}$. Note that the long segment may intersect the resulting circles and itself. The \textit{algebraic intersection number} of one of the resulting circles with the long segment is defined as the total times of the segment intersects the circle from left to right minus the total times of the segment intersects the circle from right to left. Let $I_{\tilde{K}}$ denote the sum of the algebraic intersection numbers of the resulting circles with the long segment.
\begin{figure}[H]
 \begin{center}
     \begin{tabular}{c}
     \centering  \scalebox{0.5}{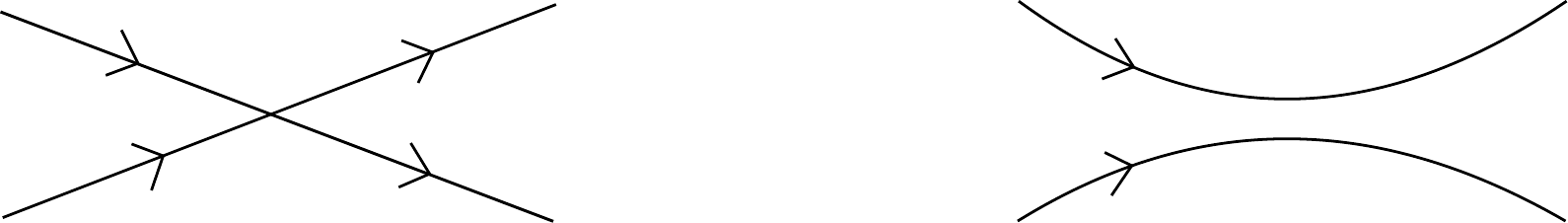}
     \end{tabular}
     \caption{\bf Smoothing a crossing of $\tilde{K}$ in the oriented way}
     \label{fig:seifertsmooth}
\end{center}
\end{figure}
None of the crossings of $\tilde{K}$ that contributes to non-trivially to the total algebraic intersection number, is smoothed since such a crossing is met only once. As a result, $I_{\tilde{K}}$ is equal to the sum of algebraic intersection numbers of the loop $l(\tilde{C})$ with the strands intersecting $l(\tilde{C})$. This shows that the sum of algebraic intersection numbers of the circles with the long segment is equal to either $w_-(\tilde{C})$ or $w_+(\tilde{C})$. Thus, the absolute value of $I(K)$ is equal to the , $|I(K)|=m$.

 On the other hand, it is easy to verify that the number $|I_K|$ can be at most as large as the number of the circles that are enclosing the endpoints (the tail or the head). In particular, $|I_K|$ is equal to the number of the circles if all intersections are positive. Thus we have that $m$ is at most as the number of circles enclosing the endpoints.

 The height of the diagram $\tilde{K}$ is at least as large as the number of the circles enclosing the endpoints, by the Jordan curve theorem. With this we have $h(\tilde{K})\geq m$.

 The affine index polynomial is a knotoid invariant so $m$ appears as the maximum degree of the affine index polynomial of any classical knotoid diagram equivalent to $\tilde{K}$. This implies that there is a crossing with weight $m$ in each representative knotoid diagram of $K$. Applying the same procedure explained above to the loops of the crossings with weight $m$ in each representative diagram gives us the inequality, $h(K)\geq m$ for any representative classical diagram $K$ and the statement follows.
\end{proof}
Figure \ref{fig:af} gives an illustration for the proof on the knotoid diagram $K$ in Figure \ref{fig:knotoid}(g).
\begin{figure}[H]
\centering
    \begin{subfigure}[b]{0.3\textwidth}
        \centering  \scalebox{0.75}{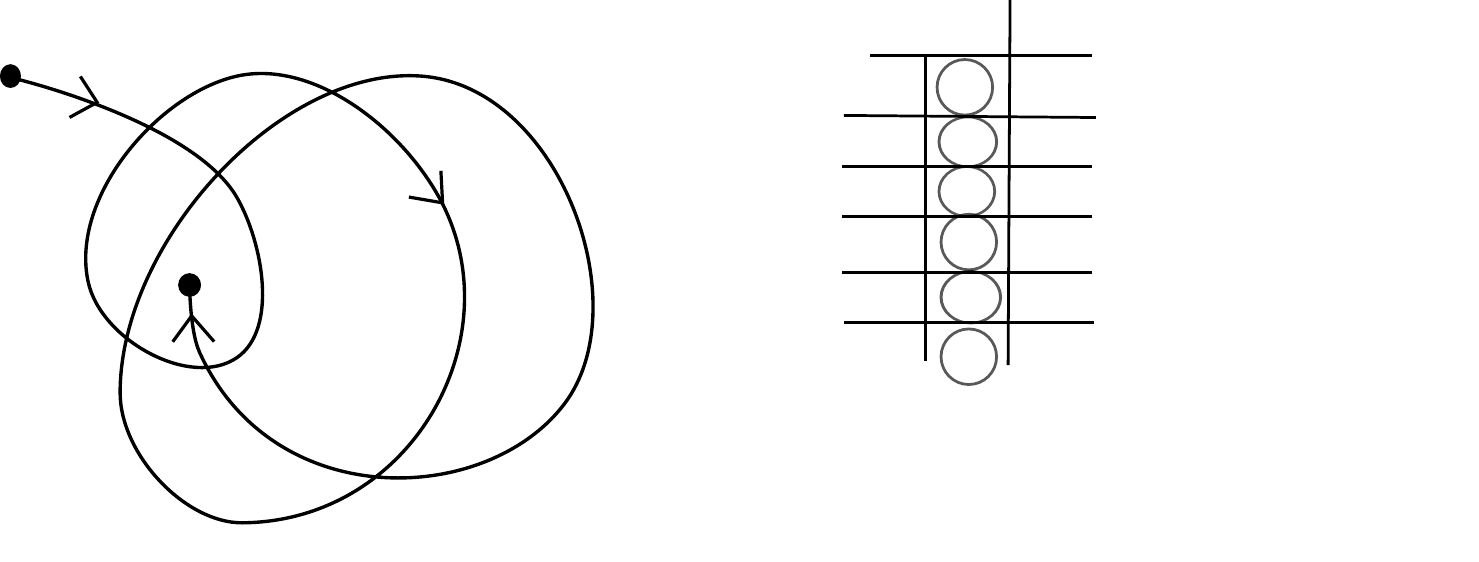}
     \caption{\bf $F(K)$ with integer labels}
     \label{subfig:labelz}
\end{subfigure}	\\
\vspace{2cm}
\begin{subfigure}[b]{0.3\textwidth}
     \centering  \scalebox{0.75}{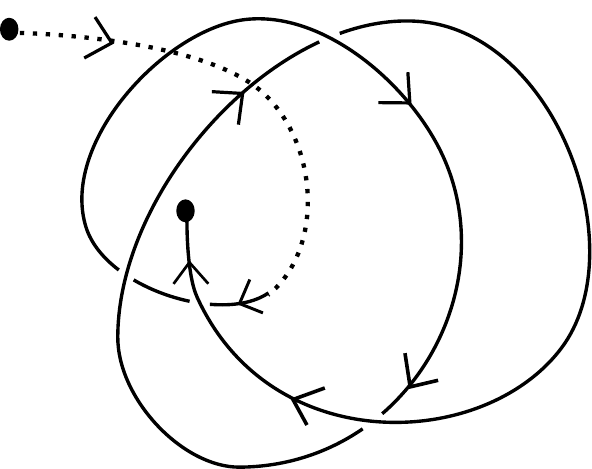}
     \caption{\bf Loop of the crossing C}
     %\label{Figure}
\end{subfigure}
\qquad
\begin{subfigure}[b]{0.6\textwidth}
     \centering  \scalebox{0.5}{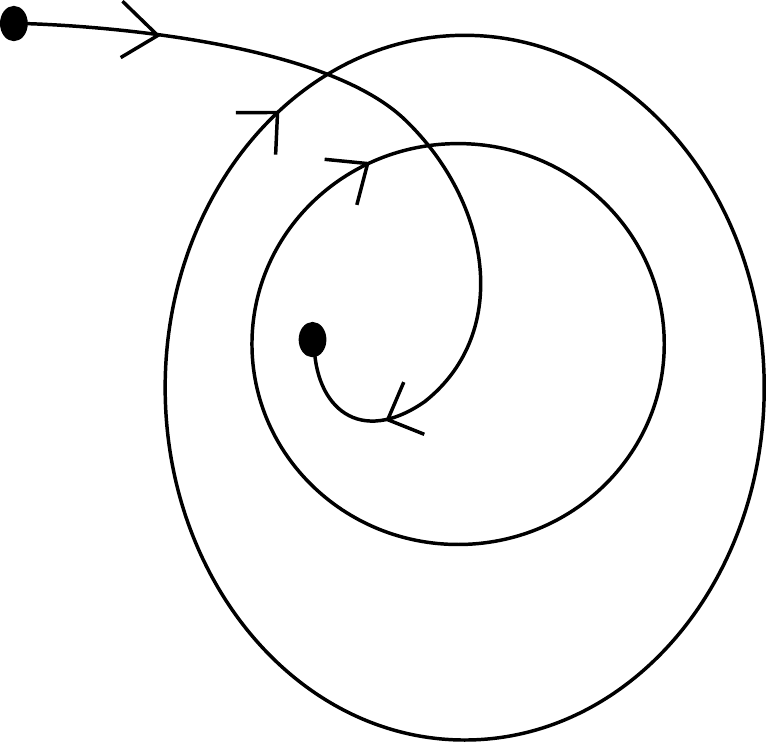}
     \caption{\bf Resulting circles and the long segment}
     %\label{fig:seifertcircles}
		\end{subfigure}	
		\caption{\bf An illustration for the proof of Theorem 4.12}
		\label{fig:af}
		\end{figure}
 One immediate consequence of Theorem 4.12 is that we are able to tell the height of the knotoids that can be represented with a spiral diagram with positive crossings. In particular, the affine index polynomials of the knotoids each represented by a diagram overlying the flat diagrams in \ref{fig:spi} with positive crossings are the following. $P_{K_1}(t)=t+t^{-1}-2$, $P_{K_2}(t)=t^2+t+t^{-1}+t^{-2}-4$ and $P_{K_3}(t)=t^3+t^2+t+t^{-1}+t^{-2}+t^{-3}-6$. The heights of the given diagrams are $1$, $2$ and $3$, respectively. Then by Theorem 4.12, it is concluded that the heights of the knotoids are $1$, $2$ and $3$, respectively. This is generalized as follows. The affine index polynomial of a classical knotoid represented by an $n$- fold spiral knotoid diagram has a term of the form $t^n+t^{-n}$ if all crossings of the diagram are positive. The maximum degree of the affine index polynomial is $n$ and the height of the spiral diagram is $n$. By Theorem 4.12, we conclude that the height of the knotoid is $n$. This shows that we have an infinite set of knotoids whose height is given by the affine index polynomial. 
\begin{figure}[H]
 \begin{center}
     \begin{tabular}{c}
     \centering  \scalebox{0.75}{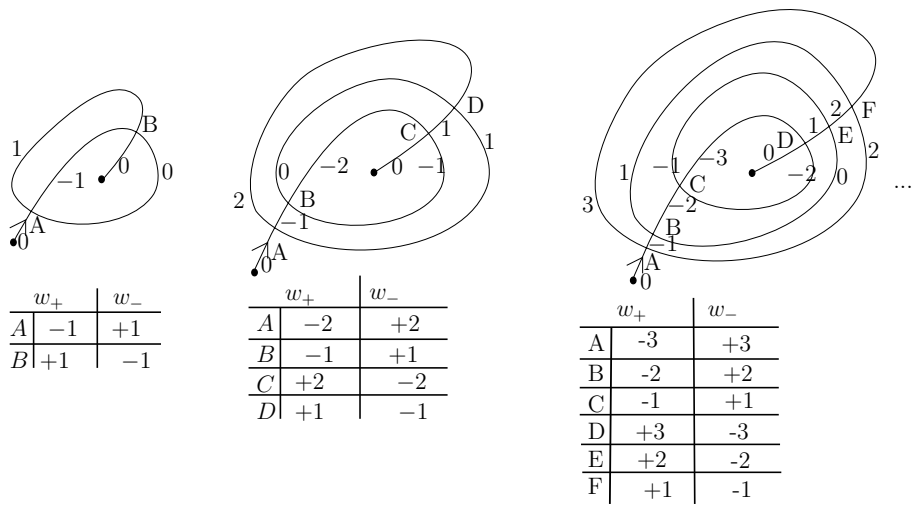}
     \end{tabular}
     \caption{\bf Flat spiral knotoid diagrams}
     \label{fig:spi}
\end{center}
\end{figure}	
There are examples of proper knotoids with trivial affine index polynomial so that the affine index polynomial gives trivial lower bound for the height of knotoids. More precisely, the knotoid $K$ represented by the diagram that overlies the $3$-fold flat spiral diagram in Figure \ref{fig:spi} with negative crossings $B$, $C$ and $D$ and positive crossings $A$, $E$, $F$, has trivial affine index polynomial. This implies that the affine index polynomial gives trivial information about the height of the knotoid represented. Here we reveal the following question: \textit{ Are there any other knotoid invariants giving any nontrivial information about the height}? The arrow polynomial discussed in the following section, gives an answer to this question. It is showed that the knotoid $K$ represents a non-trivial knotoid and in fact, a proper knotoid with height $3$, by a use of the arrow polynomial.
    ~ %add desired spacing between images, e. g. ~, \quad, \qquad, \hfill etc. 
	\section{The Arrow Polynomial}
We define the arrow polynomial for knotoids in analogy with the arrow polynomial of virtual knots and links which was defined by H.A. Dye and L.H. Kauffman [DK] and independently by Y.~Miyazawa \cite{Mi}. The construction of the arrow polynomial of knotoids both for classical and virtual, is based on the \textit{oriented state expansion} of the bracket polynomial of knotoids which is shown in Figure \ref{fig:orientedstate}. 
\begin{figure}[H]
     \begin{center}
     \begin{tabular}{c}
     \centering  \scalebox{0.5}{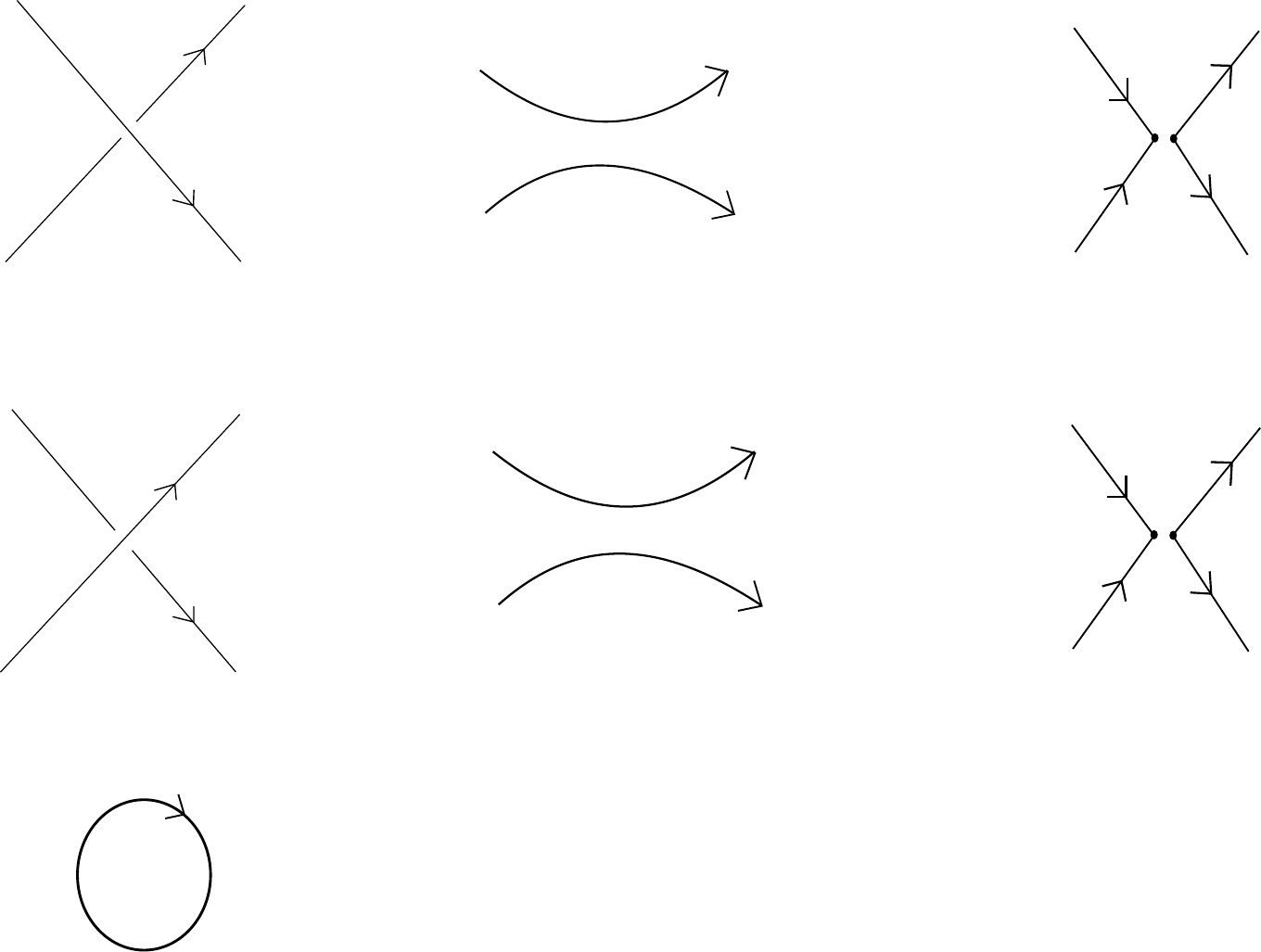}
     \end{tabular}
     \caption{\bf Oriented state expansion}
     \label{fig:orientedstate}
\end{center}
\end{figure}	
Oriented state expansion of knotoids involves \textit{oriented and disoriented smoothings} of all classical crossings that result in \textit{oriented states} circular components and one long state component or only a single long state component. The state components which are obtained by disoriented smoothings include an extra combinatorial structure in the form of \textit{paired cusps}. Each cusp has two arcs either going into the cusp or going out from the cusp. A cusp can be denoted by an angle which locally divides $S^2$ into two parts. One part is the span of the acute angle and the other part is the span of the obtuse angle. We call the part which is the span of the acute angle as \textit{inside} of the cusp and the part which is the span of the obtuse angle as \textit{outside} of the cusp.

 There is a list of rules which reduce the number of cusps in a state component that are determined accordingly to the virtual equivalence that is generated by the isotopy of $S^2$ or $\mathbb{R}^2$ and the detour move (only isotopy of $S^2$ or $\mathbb{R}^2$ for the classical case). This list is given in Figure \ref{fig:rules}. The basic reduction rule consists of cancellation of two consecutive cusps both with insides on the same side of the segment connecting them. Two consecutive cusps on a state component which have insides on the opposite sides of the segment connecting them, are not canceled out. Specifically, any two consecutive cusps on a circular component are canceled if they have insides in the same local region that the component forms. Therefore, a circular component with two such cusps turns into an embedded circular component which contributes to the polynomial as $d=(-A^2-A^{-2})$. Any two consecutive cusps on a long state component which have insides on the same side of the segment connecting them, are canceled out as well. Such a long state component turns into an embedded arc in $S^2$ and contributes to the polynomial with the same value of an embedded circular state, as $d=(-A^2-A^{-2})$. Two cusps on a circular component with insides on the opposite local sides of the circle are kept as graphical nodes. This component is regarded as a circular graphical state. Two cusps on a long state component whose insides are on opposite sides of the segment connecting them, are not reduced as well. Such a long state component is regarded as a graphical state. The graphical components contribute to the polynomial as extra variables. A circular graph component with surviving cusps can be turned into a circular graph without any virtual crossings by the detour move so that it can be depicted as a circular graph with cusps forming zig-zags on the component. A circular component with two cusps forming a zig-zag contributes as $K_1$ to the polynomial. In general, a circular graph with zig-zags formed by $2i$ alternating cusps, contributes as a variable, $K_i$ to the arrow polynomial. A long state component with zig-zags formed by $2i$ alternating cusps contributes as an additional variable, as $\Lambda_i$ to the arrow polynomial.
\begin{definition}\normalfont
We define the \textit{arrow polynomial} of a virtual or classical knotoid diagram $K$ as,
\begin{center}
${A[K]=\sum_{S} {A^{i-j}}{(-A^2-A^{-2})^{\|S\|-1}}{<\hat{S}>}}$,
\end{center}
where the sum runs over the oriented bracket states, $i$ is the number of state markers touching $A$ labels and $j$ is the number of state markers touching $A^{-1}$ labels in the state $S$, as in the usual the bracket sum, $\|S\|$ is the number of components of the state $S$ and $<\hat{S}>$ is the product of variables, ${K_{i_1}}^{j_1}...{K_{i_n}}^{j_n}\Lambda_i$, associated to the components of $S$ with surviving cusps.
\end{definition}
The variables $K_i$ and $\Lambda_i$ constitute an infinite set of commuting variables, commuting with each other also with the variable $A$ of the arrow polynomial. It is left to the reader to show an $\Omega_1$- move changes the arrow polynomial of a virtual knotoid by $-A^{\pm 3}$. 
\begin{figure}[H]
     \begin{center}
     \begin{tabular}{c}
     \centering  \scalebox{0.5}{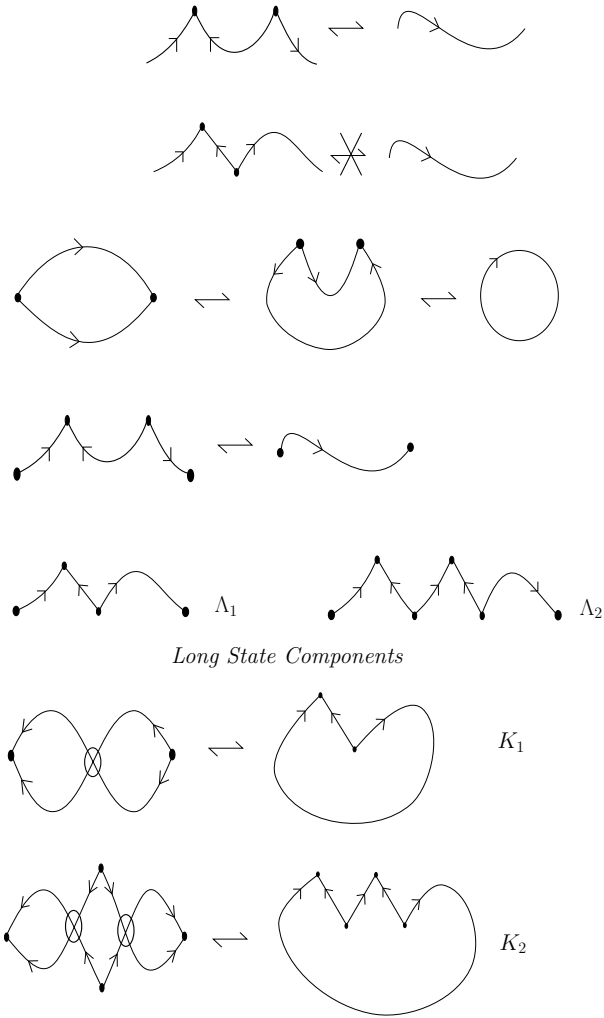}
     \end{tabular}
     \caption{\bf Reduction rules for the arrow polynomial}
     \label{fig:rules}     
\end{center}
\end{figure}
\begin{thm}
The normalization of arrow polynomial by $(-A^3)^{-\writhe(K)}$, where $\writhe(K)$ is the writhe of $K$, is a virtual and classical knotoid invariant.
\end{thm}
\begin{proof}
The proof follows similar as the proof of the invariance of the arrow polynomial for virtual knots/links. See \cite{DK,Ka2}.
\end{proof}
See Figure \ref{fig:ex} for an example of a knotoid with nontrivial arrow polynomial.
\begin{figure}[H]
\centering\scalebox{.8}{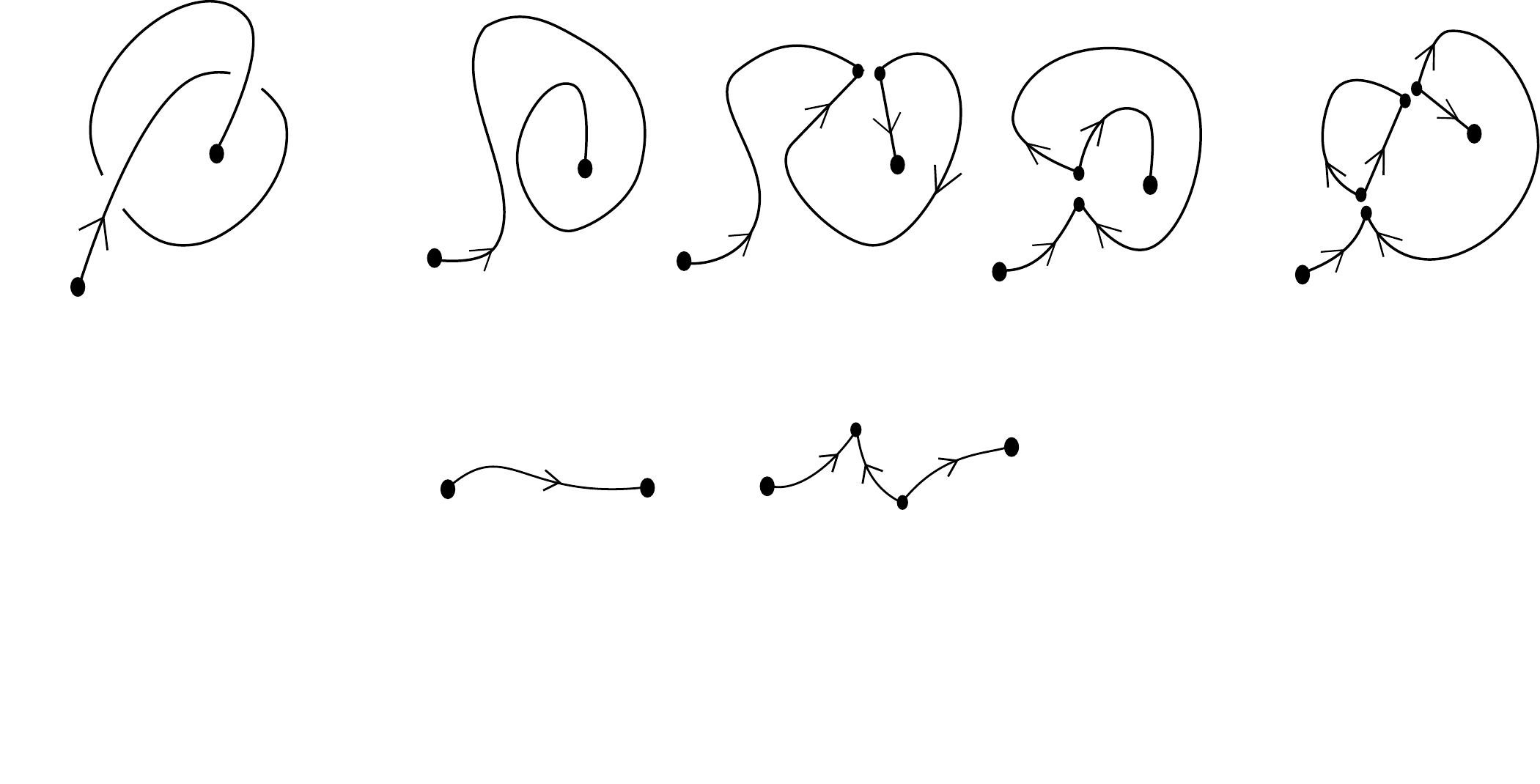}
\caption{An example}
\label{fig:ex}
\end{figure}
\begin{definition}
\normalfont
The \textit{\textsl{K}-degree} of a summand of the arrow polynomial of a virtual knotoid which is of the form,
$A^m({ K_{i_1}}^{j_1}{K_{i_2}}^{j_2}...{K_{i_n}}^{j_n})\Lambda_i$, is equal to
\begin{center}
${i_1}\times{j_1}+...+{i_n}\times{j_n}$.
\end{center}
The \textit{\textsl{K}-degree} of the arrow polynomial of a virtual knotoid is defined to be the maximum $K$- degree taken among the $K$-degrees of summands of the arrow polynomial of the knotoid.
\end{definition}
Note that the $\textsl{K}$-degree of the arrow polynomial of a virtual knot/link is defined in a similar way, as the maximum $\textsl{K}$-degree among the $K$-degrees of the summands of the arrow polynomial \cite{DK}.
\begin{definition}
\normalfont
The $\Lambda$-degree of a summand of the arrow polynomial of a virtual knotoid which is in the form, $A^m({ K_{i_1}}^{j_1}{K_{i_2}}^{j_2}...{K_{i_n}}^{j_n})\Lambda_i$ is equal to $i$. The \textit{$\Lambda$-degree} of the arrow polynomial of a virtual knotoid is defined to be the maximum $\Lambda$- degree among the  $\Lambda$-degrees of all the summands of the polynomial.
\end{definition}
Note that for a classical knotoid diagram $K$ in $S^2$ or a virtual one, the oriented state components of the virtual closure of $K$, $\overline{v}(K)$ is obtained by connecting the endpoints of each long state component in the oriented state expansion of $K$, in the virtual fashion (with an embedded arc creating virtual crossings whenever it meets with the component). Therefore, instead of assigning $\Lambda_i$ to long state components with $2i$ cusps which are not reduced by the reduction rules, if we assigned $K_i$ as a variable, we would have 
\begin{center}
$A[K]=A[\overline{v}(K)]$.
\end{center}
 The arrow polynomial gets more effective as an invariant of virtual knotoids by assigning to long state components with $2i$ irreducible cusps, the variable $\Lambda_i$, $i=1,2,...$. Figure \ref{fig:11} depicts the oriented state expansion of the knotoid diagram given before in Figure \ref{fig:trivialvc}. The reader can easily verify that the arrow polynomial of the virtual knot that is the virtual closure of this knotoid diagram is trivial. In other words, assigning $K_1$ to the long state components results in trivial arrow polynomial. Assigning $\Lambda_1$ to the long state components, however, results in a non-trivial arrow polynomial, as shown in Figure \ref{fig:11}. The arrow polynomial detects the non-triviality of this knotoid.

 Another example is the virtual knotoid represented by the knotoid diagram $K$, shown in Figure \ref{fig:Slavik}. The virtual closure of this knotoid is the Slavik's Knot \cite{DK} whose normalized arrow polynomial is trivial. The arrow polynomial of the knotoid is $A[K]=(A^{-9}+A^{-7}+3A^{-5}+5A^{-1}+A+6A^3++2A^5+3A^7 )+(-A^3-A^{-1}+A^{-3}+A+A^5)\Lambda_1$. This implies that the normalized arrow polynomial is non-trivial and shows that the non-triviality of this virtual knotoid is detected by the normalized arrow polynomial defined by assigning $\Lambda_1$-variable to the long state components. 
\begin{figure}[H]
    \centering
    %\begin{subfigure}[b]{0.3\textwidth}
        \centering  \scalebox{0.75}{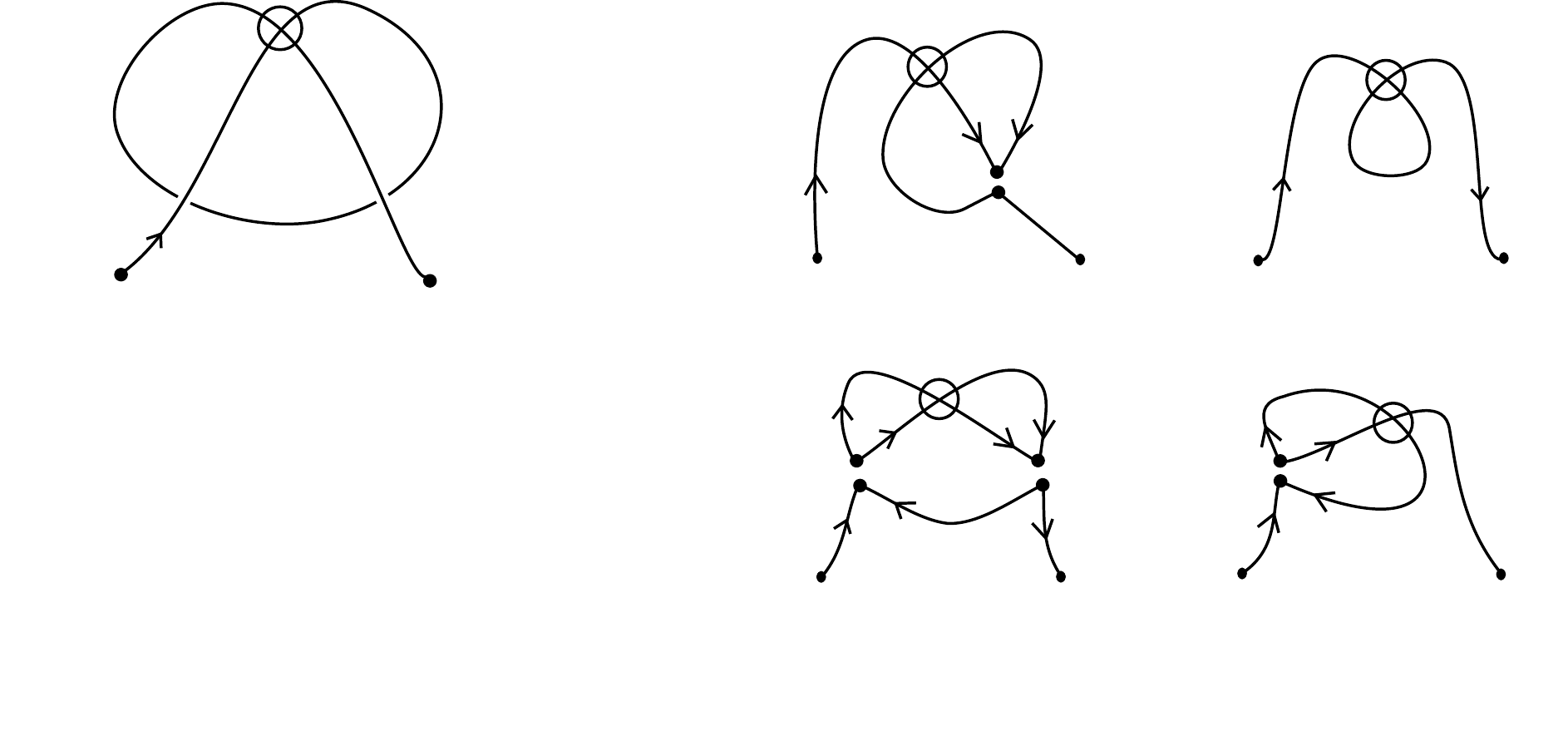}
        \caption{\bf The arrow polynomial of the knotoid in Figure \ref{fig:trivialvc}}
        \label{fig:11}
    \end{figure}
	\begin{figure}[H]
    \centering
    %\begin{subfigure}[b]{0.3\textwidth}
        \centering  \scalebox{0.30}{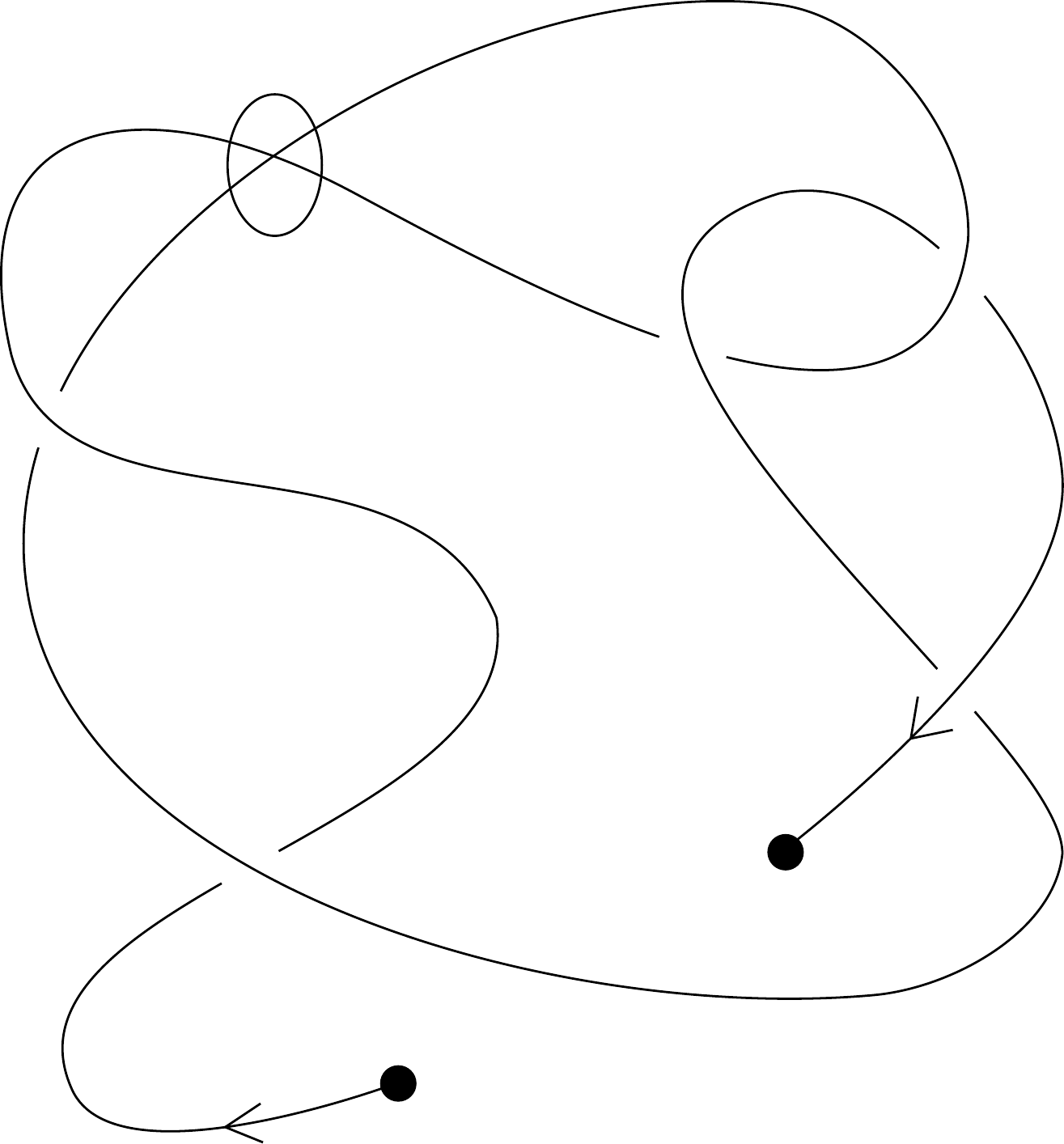}
        \caption{\bf The knotoid closing to Slavik's knot}
        \label{fig:Slavik}
    \end{figure}
		\begin{thm}[\cite{Ka2}]
In a classical knot or link diagram, all state components of the arrow polynomial reduce to loops that are free from cusps.
\end{thm}
%The endpoints on a long state component in the oriented state expansion of a knotoid diagram can be connected in the virtual fashion. That is, they can be connected with an embedded arc that creates a virtual crossing on the component every time it meets with a piece of the component. Let $K$ be a classical knotoid diagram. It is clear that by connecting the endpoints of each long state component in the oriented state expansion of $K$ in the virtual fashion we obtain the oriented states of the virtual closure of $\overline{v}(K)$. 
 It is clear that the virtual closure of a knot-type knotoid diagram is a classical knot diagram. The oriented state components of a knot-type knotoid diagram become the oriented state components of a classical knot diagram when the endpoints of long components are connected virtually. Then it follows by the Theorem 5.2 that cusps do not survive in any of the state components of a knot-type knotoid diagram, and we have the following corollary.  
\begin{cor}\normalfont
The normalized arrow polynomial of a knot-type knotoid coincides with the normalized bracket polynomial of the knotoid.
\end{cor}
 Cusps may survive in a long state component of a proper knotoid diagram. If the $\Lambda$-degree of the arrow polynomial of a knotoid is nonzero then it is immediate to conclude by the discussion above that it not a knot-type but a proper knotoid. For example, the knotoid diagram shown in Figure \ref{fig:ex} represents a proper knotoid since the arrow polynomial of the knotoid has $\Lambda$- degree $1$. 

 The circular components of an oriented state of any classical knotoid diagram are all free of cusps. This follows by the same reasoning with the reasoning of Theorem 5.2. As a conclusion, the $K$-degrees of any summand of the arrow polynomial of a classical knotoid is zero.

  For virtual knotoids, cusps can survive in circular state components as well as they can survive in long state components. It means that both the $K$- and $\Lambda$- degrees of the arrow polynomial of a virtual knotoid may be nontrivial. We know that the knotoid diagram, given in Figure \ref{fig:11} is not virtually equivalent to a classical knotoid since the $K$-degree of the arrow polynomial is $1$.
\begin{rem}\normalfont
Direct computation shows that the arrow polynomials of the knotoids represented by the diagrams $K_1$ and $K_2$ given in Figure \ref{fig:different},  are $A[K_1]=1-A^{-4}+A^4+(-A^{-2}+A^2)\Lambda_1$ and $A[K_2]=2-A^{-4}-A^4+A^8+(-A^{-6}+A^{-2})\Lambda_1$, respectively. It can be easily verified that the normalized arrow polynomials of $K_1$ and $K_2$ are different. Therefore $K_1$ is not equivalent to $K_2$.
\end{rem}
\begin{rem}\normalfont
The arrow polynomial generalizes to a virtual multi-knotoid invariant directly as follows. All crossings including the crossings shared by two components of a given oriented virtual multi-knotoid diagram are smoothed in the same way. The resulting oriented state components, including oriented circular components and oriented long state components, are labeled by either $A$ or $A^{-1}$ at each smoothing site. The arrow polynomial for multi-knotoids is defined as the summation of all products of labels assigned to oriented state components. If $K$ is a multi-knotoid diagram without any virtual crossings then it follows by a similar discussion with the proof of Theorem 5.2 that the circular components of $K$ are free of cusps, and cusps can survive only on the long state components. 
\end{rem}
\subsection{The Arrow Polynomial and the Height of Knotoids}
The arrow polynomial can be used for estimating the height of a knotoid in $S^2$. Let us firstly recall more from virtual knot theory.
\begin{definition}
\normalfont The \textit{virtual crossing number} of a virtual knot/link is the minimum number of virtual crossings over all representative diagrams.
\end{definition}
The problem of determining the virtual crossing number of a virtual knot or link is a fundamental problem in virtual knot theory. There is a relation between the virtual crossing number and the maximal $K$-degree of the arrow polynomial of a virtual knot, as stated by the following theorem.
\begin{thm} \cite{DK}
The virtual crossing number of a virtual knot/link is greater than or equal to the maximal \textit{K}-degree of the arrow polynomial of that virtual knot/link.
\end{thm}
Let $K$ be a knotoid in $S^2$. The oriented state components of the virtual closure of $K$, $\overline{v}(K)$ are the same with the oriented state components of $K$ when the long state components are closed in the virtual fashion. Therefore the $\Lambda_i$-variables assigned to long state components with surviving cusps of a knotoid transform to $K_i$-variables assigned to the circular components with surviving cusps in the arrow polynomial of the virtual knot which is the virtual closure of the knotoid. Using this idea, we show that the $\Lambda$-degree of the arrow polynomial can be used as a lower bound for the height of knotoids in $S^2$. 
\begin{thm}
The height of a knotoid $K$ in $S^2$ is greater than or equal to the $\Lambda$-degree of its arrow polynomial.
\end{thm}
\begin{proof}
By Theorem 5.3 and the discussion above we have the following inequality,
\begin{center}
The $\Lambda$-degree of $A[K]$ $ \leq$ The virtual crossing number of the knot $\overline{v}(K)$.
\end{center}
It is clear that the least number of virtual crossings obtained by closing a classical knotoid diagram virtually, is equal to the height of that diagram. Let $\tilde{K}$ be a classical knotoid diagram representing $K$. Then, $h(\tilde{K})$ is equal to the number of virtual crossings of $\overline{v}(\tilde{K})$, where $h(\tilde{K})$ denotes the height of the knotoid diagram $\tilde{K}$. So, the virtual crossing number of the virtual knot $\overline{v}(K)$ is less than or equal to $h(\tilde{K})$. By this and the first inequality, we have the following.
\begin{center}
The $\Lambda$-degree of $A[K]$ $\leq$ $h(\tilde{K})$.
\end{center}
The inequality above holds for any classical knotoid diagram equivalent to $K$ since the $\Lambda$-degree of the polynomial is invariant under the $\Omega$-moves. Therefore we have,
\begin{center}
The $\Lambda$-degree of $A[K]\leq h(K)$,
\end{center}
where $h(K)$ denotes the height of the knotoid $K$.
\end{proof}
%\vspace{-2cm}
Thus we have two tools; the affine index polynomial and the arrow polynomial for the estimation of the height of a knotoid in $S^2$. There are cases that both of the polynomials give the same estimation for the height and there are cases in which one of the polynomials give a more accurate estimation. We show some examples for each of these cases.
\begin{example}\normalfont
It can be verified by the reader that the affine index polynomial of the knotoid $K$ which is overlying the flat knotoid diagram given in Figure \ref{fig:spi} with the crossings $B$, $C$ and $D$ are negative and the rest of the crossings are positive, is trivial . The arrow polynomial of the knotoid $A[K]$, $A[K]=1+(-A^{-3}+A^{-2}+A^2+A^6)\Lambda_1+(-2A^{-4}-2A^4+4)\Lambda_2+(-A^{-6}+A^{-2}+A^2+A^6)\Lambda_3$. The $\Lambda$- degree of the arrow polynomial is $3$ so by Theorem 5.4, the height of the knotoid $K$ is at least $3$. It is not difficult to see that the height of the given diagram is also $3$. Thus the height of the knotoid $K$ is $3$. 
\end{example}
\begin{example}\normalfont
Figure \ref{fig:fve} shows the knotoid 5.7 \cite{Ba} and the corresponding weight chart. It can be verified easily that the odd writhe and the arrow polynomial of knotoid \textbf{5.7} are trivial. The arrow polynomial of the knotoid, $A[K_{5.7}]=(-A^{-3}+A-2A^5+A^9)+(A^{-9}-2A^{-5}+2A^{-1}-2A^3+A^7)\Lambda_1$. Thus it is a non-trivial arrow polynomial with $\Lambda$-degree $1$. This tells us that the height of the knotoid 5.7 is at least $1$. Since the knotoid 5.7 can be represented by a diagram with height $1$, we conclude that the height of the knotoid 5.7 is $1$. 
\end{example}
\begin{figure}[H]
    %\centering
   %acing between images, e. g. ~, \quad, \qquad, \hfill etc. 
			%\hfill
		%	\begin{tabular}{c}
		%\includegraphics[width=7cm]{4_9.png}
		%\end{tabular}
		\centering  \scalebox{.35}{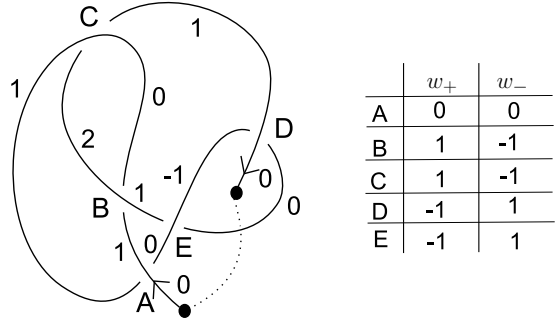}
      \caption{The weight chart of knotoid 5.7}
\label{fig:fve}	
 \end{figure}
		\begin{example}\normalfont
The arrow polynomial of the knotoid $K$, represented by the diagram in Figure \ref{fig:knotoid}(g) is 
$A[K]=A^6-(A^{-4}-A^4)\Lambda_1-(A^{-2}-A^2)\Lambda_2$. Thus the $\Lambda$-degree of the arrow polynomial of the knotoid $K$ is $2$. The affine index polynomial of $K$, $P_K(t)=t^2+2t+2t^{-1}+t^{-2}-6$, as can be computed easily by the weight chart given in Figure 33(a) showing the weights of crossings of the knotoid diagram. Thus, the maximal degree of the affine index polynomial of $K$ is also equal to $2$. So both the affine index polynomial and the arrow polynomial give the same lower bound for the height. Since $K$ can be represented by the given diagram with height $2$, we conclude that the height of $K$ is $2$. 
\end{example}
\begin{example}\normalfont
The reader can easily see that the height of the knotoid diagram $K$ given in Figure \ref{fig:knotoid}(f) is equal to $2$. We want to find out if there exists an equivalent knotoid diagram to $K$ with less height. The affine index polynomial of $K$ is $P_K(t)=2t+2t^{-1}-4$, as can be verified by Figure \ref{fig:lastxx}. The arrow polynomial of $K$ is $A[K]=(-A^{-5}+2A^{-1}-A^3-A^7) + 2(A-A^5)\Lambda_1$. The affine index polynomial and the arrow polynomial both assure that the height of $K$ is at least $1$. Therefore we have, $1 \leq h(K) \leq 2$. This is a case where our tools discussed in this paper can not give an exact estimation for the height.
\end{example}
\begin{figure}[H]
    %\centering
   %acing between images, e. g. ~, \quad, \qquad, \hfill etc. 
			%\hfill
			%\HUGE{
		\centering  \scalebox{.45}{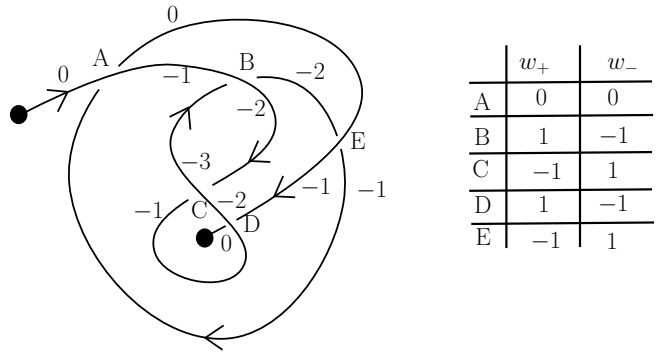}
		%}
      \caption{The weight chart of $K$}
\label{fig:lastxx}	
 \end{figure}
\section{Discussion}
 We end the paper with a full list of the questions that are discussed throughout the paper and possible future directions for the study of knotoids. 
\begin{enumerate}
\item  Determination of the kernel of the virtual closure map: We have nontrivial virtual knotoids closing virtually to the trivial knot. However nontrivial knot-type knotoids close to nontrivial knots. 
\textit{Is there a proper knotoid (a classical knotoid with nonzero height) whose virtual closure is the trivial knot?}\\
\item \normalfont Determination of the image of the virtual closure map: We show that the virtual closure map is not a surjective map. The proof will appear in \cite{GK1}. Here we ask the following question. \textit{How to determine if a given virtual knot is in the image of $\overline{v}$?} \\
\item\normalfont A generalization of the first question: \textit{Is there a proper knotoid whose virtual closure is a classical knot}
or \textit{do proper knotoids always close (virtually) to a virtual knot of genus 1?}\\
\item\normalfont Conjecture: \textit{The Jones polynomial for knotoids in $S^2$ detects the triviality of classical knotoids.}
Let $\overline{K}$ be a virtual knot with trivial Jones polynomial. If the conjecture holds, we will be able to conclude that the virtual closure of any proper knotoid is nontrivial, by using the equality $V(K)=V(\overline{v}(K))$, where $V$ denotes the Jones polynomial. 
%Also, there are virtual knots that are candidates to be in the image of $\overline{v}$ for being genus $1$ virtual knots. One of them is given in Figure \ref{fig:virt3} in Section 3.3. This virtual knot (listed in \cite{Gr} as the knot 3.1) is known to be nontrivial although it has trivial Jones polynomial. Assuming the conjecture for the Jones polynomial holds, we can conclude that this virtual knot is not the virtual closure of a classical knotoid. In fact, in our paper named Parity in knotoids, we show that this virtual knot is not in the image of the virtual closure map by a more direct method.
\\
\item We want to know more about the height of knotoids and its relations with both the affine index polynomial and the arrow polynomial.
We have given examples where the estimation of the arrow polynomial is more powerful than the affine index polynomial in detecting the height of a given classical knotoid. Does there exist an example for which the index polynomial is superior to the arrow polynomial in height determination? \\
\item Khovanov homology can be extended to an invariant of knotoids. There is a direct analog of Khovanov homology for classical knotoids. The analogs of Khovanov homology for virtual knots \cite{DKK, Ma3} can be applied to virtual knotoids. It is worth investigating Khovanov homology for knotoids. We can ask the following question: \textit{Does Khovanov homology for knotoids detect the trivial knotoid?} Note that Khovanov homology detects the unknot \cite{KM}.\\
%\item We wish to generalize the affine index polynomial and its generalizations \cite{CG, Ch} to multi-knotoids in a nontrivial way.\\
%\item The invariants of knotoids defined in this paper are based on the fact that the endpoints can be located in different local regions of the 2-sphere. We can look at a knotoid as a generalization of a long knot. Invariants of long knots \cite{KM,Ma4} can take into account the ordering implicit in going from one end to the other. We will study analogs of this for knotoids in subsequent publications. By mixing considerations of region location of the endpoints plus the ordered structure we will obtain subtle invariants of knotoids. \\
\item Let $C$ be an open oriented curve in $3$-dimensional space. The set of knotoids associated to $C$ that are obtained by projecting the curve to planes deserves investigation since the physical properties of the curve can be studied in this way.
\end{enumerate}
\section*{\ackname}
 This research  has been co-financed by the European Union (European Social Fund - ESF) and Greek national funds through the Operational Program "Education and Lifelong Learning" of the National Strategic Reference Framework (NSRF) - Research Funding Program: THALES: Reinforcement of the interdisciplinary and/or inter-institutional research and innovation.

 The first author is grateful to her advisor Prof. Sofia Lambropoulou for suggesting the subject of knotoids for her PhD study.

\end{document}